\documentclass [11pt,oneside,a4paper,mathscr]{amsart}

\usepackage{amscd,amsmath,amssymb,euscript}
\usepackage[frame,cmtip,curve,arrow,matrix,line,graph]{xy}
\usepackage{tikz-cd}
\usepackage{hyperref}
\usepackage{stmaryrd}
\usepackage{bm}

\title[]{Quasi-BPS categories for K3 surfaces}
\author{Tudor P\u adurariu and Yukinobu Toda}

\usepackage{enumerate}

\newtheorem{thm}{Theorem}[section]
\newtheorem{cor}[thm]{Corollary}
\newtheorem{prop}[thm]{Proposition}
\newtheorem{conj}[thm]{Conjecture}
\newtheorem{lemma}[thm]{Lemma}

\theoremstyle{definition}
\newtheorem{defn}[thm]{Definition}

\newtheorem{thm*}[thm]{Theorem$^*$}
\newtheorem{remark}[thm]{Remark}

\newtheorem{problem}{Problem}[section]

\newcommand{\comment}[1]{}

\renewcommand{\leq}{\leqslant}
\renewcommand{\geq}{\geqslant}

\newcommand{\X}{\mathscr{X}}

\newcommand{\Coh}{\operatorname{Coh}}

\newcommand{\Ker}{\operatorname{Ker}}
\newcommand{\id}{\operatorname{id}}
\newcommand{\Ext}{\operatorname{Ext}}
\newcommand{\Ind}{\operatorname{Ind}}
\newcommand{\Hom}{\operatorname{Hom}}

\newcommand{\Spec}{\operatorname{Spec}}

\newcommand{\ch}{\operatorname{ch}}

\newcommand{\rk}{\operatorname{rank}}

\newcommand{\inclusion}{\ar@<-0.3ex>@{^{(}->}[r]}

\newcommand{\Tr}{\mathop{\rm Tr}}

\newcommand{\ssslash}{/\!\!/}
\newcommand{\llangle}{\langle \! \langle}
\newcommand{\rrangle}{\rangle \! \rangle}
\newcommand{\diasquare}{\ar@{}[rd]|\square}

\usetikzlibrary{positioning,fit,calc}
\tikzstyle{block}=[draw=black, width=1cm, minimum height=2cm, align=center] 
\tikzstyle{block2}=[draw=black, text width=2cm, minimum height=1cm, align=center] 
\tikzstyle{block3}=[draw=black, text width=2cm, minimum height=1cm, align=center] 

\makeatletter

\@addtoreset{equation}{section}	
\makeatother
\begin{document}
	\maketitle
 
\begin{abstract}
We introduce and begin the study of quasi-BPS categories for K3 surfaces, which are a categorical version of the BPS cohomologies for K3 surfaces.

We construct 
semiorthogonal decompositions of derived categories of 
coherent sheaves on 
moduli stacks of semistable objects on K3 surfaces, where 
each summand is a categorical Hall product of quasi-BPS categories. 
We also prove the wall-crossing equivalence of 
quasi-BPS categories, which generalizes Halpern-Leistner's 
wall-crossing equivalence of moduli spaces of 
stable objects for primitive Mukai vectors on K3 surfaces. 

We also introduce and study a reduced quasi-BPS category. 
When the weight is 
coprime to the Mukai vector, 
the reduced quasi-BPS 
category is proper, smooth, and its Serre functor 
is trivial \'{e}tale locally on the 
good moduli space. 
Moreover we prove that its topological K-theory recovers the BPS invariants of K3 surfaces, 
which are known to be equal to the Euler characteristics of Hilbert schemes of points on K3 surfaces. 
We regard reduced quasi-BPS categories as noncommutative hyperkähler varieties which are categorical versions of crepant resolutions of singular symplectic moduli spaces of semistable 
objects on K3 surfaces.
\end{abstract}

\section{Introduction}
Let $S$ be a K3 surface, $v$ a Mukai vector, and 
$w$ an integer. 
The purpose of this paper is to introduce and study 
a category 
\begin{align}\label{intro:T}
\mathbb{T}=\mathbb{T}_S(v)_w^{\rm{red}}
\end{align}
called \textit{(reduced) quasi-BPS category}. 
When $v$ is primitive, 
\eqref{intro:T} is equivalent to the derived category of 
twisted sheaves over the moduli space $M$ of 
stable objects on $S$ with Mukai vector $v$, 
which is a holomorphic symplectic manifold. 
When $v$ is not necessarily primitive,
but $w$ is coprime to $v$, 
we show that $\mathbb{T}$ is proper, 
smooth, 
and has trivial Serre functor 
\'{e}tale locally on the good moduli space
$M$ of semistable objects with Mukai vector $v$, 
which is a singular 
symplectic variety. 
So 
we obtain a category $\mathbb{T}$ which we
regard
as a (twisted) categorical (\'{e}tale locally) crepant resolution of singularities of $M$. 

The construction of the category (\ref{intro:T}) is 
motivated by enumerative geometry: quasi-BPS categories
are a categorical replacement of BPS cohomologies~\cite{DM, D, KinjoKoseki, DHSM, DHSM2}, 
constructed from semiorthogonal 
decompositions of derived categories of 
moduli stacks of semistable sheaves 
which approximate the PBW theorem in 
cohomological Donaldson-Thomas (DT) theory studied in loc. cit. 
Below, we first mention our main results, and
then explain how the construction of the category (\ref{intro:T}) is motivated by DT theory and the study of singular symplectic varieties.

\subsection{Semiorthogonal decompositions into quasi-BPS categories}
For a K3 surface $S$, 
let 
\begin{align*}\mathrm{Stab}(S)
\end{align*}
be the main 
connected component of the 
space of Bridgeland stability 
conditions~\cite{Brs2} on $D^b(S)$. 
Let $\Gamma=\mathbb{Z} \oplus \mathrm{NS}(S) \oplus 
\mathbb{Z}$ be the Mukai lattice. 
For $\sigma \in \mathrm{Stab}(S)$
and $v \in \Gamma$, 
consider the moduli stacks 
\begin{align*}
\mathfrak{M}_S^{\sigma}(v) \hookleftarrow 
\mathcal{M}_S^{\sigma}(v) \to M_S^{\sigma}(v),
\end{align*}
where $\mathfrak{M}_S^{\sigma}(v)$ is the 
derived moduli stack of $\sigma$-semistable 
objects in $D^b(S)$ with Mukai vector $v$, 
$\mathcal{M}_S^{\sigma}(v)$ is its classical truncation, 
and $M_S^{\sigma}(v)$ is its good moduli space.  
Below we write $v=dv_0$ for $d \in \mathbb{Z}_{\geq 1}$
and $v_0$ a primitive Mukai vector
with $\langle v_0, v_0\rangle=2g-2$. 
We say \textit{$w\in\mathbb{Z}$ is coprime with $v$} if $\gcd(w,d)=1$.
We use the following structures on 
the derived category of $\mathfrak{M}_S^{\sigma}(v)$:

 \textbf{(The weight decomposition)}: every point in $\mathfrak{M}_S^{\sigma}(v)$
admits scalar automorphisms $\mathbb{C}^{\ast}$, and thus there is an
orthogonal 
decomposition of 
$D^b(\mathfrak{M}_S^{\sigma}(v))$ into $\mathbb{C}^{\ast}$-weight categories
    \begin{align*}
D^b(\mathfrak{M}_S^{\sigma}(v))=\bigoplus_{w\in \mathbb{Z}} D^b(\mathfrak{M}_S^{\sigma}(v))_w.    
\end{align*}

\textbf{(The categorical Hall product)}: for a decomposition $d=d_1+\cdots+d_k$, 
the stack of filtrations of $\sigma$-semistable objects 
induces \textit{the categorical Hall product} defined by Porta-Sala~\cite{PoSa}:
\begin{align}\label{intro:cathall}
    \boxtimes_{i=1}^k D^b(\mathfrak{M}_S^{\sigma}(d_i v_0)) \to D^b(\mathfrak{M}_S^{\sigma}(v)). 
\end{align}

Davison--Hennecart--Schlegel Mejia \cite[Theorem 1.5]{DHSM} proved that the Hall algebra of a K3 surface is generated by its BPS cohomology. The categorical analogue of their result is the following result, which can also be regarded as a partial categorification of a BBDG-type decomposition theorem, see Subsection \ref{subsec:intro:topK}:

\begin{thm}\emph{(Theorem~\ref{thm:sodK3})}\label{intro:thm1}
Let $\sigma \in \mathrm{Stab}(S)$ be a generic 
stability condition. 
Then there exists a subcategory (called quasi-BPS category)
\begin{align}\label{intro:qBPS}
\mathbb{T}_S^{\sigma}(v)_w \subset D^b\left(\mathfrak{M}_S^{\sigma}(v)\right)_w
\end{align}
such that
    there is a semiorthogonal decomposition 
  \begin{align}\label{intro:sod}
 D^b\left(\mathfrak{M}_S^{\sigma}(v)\right)
    =\left\langle 
    \boxtimes_{i=1}^k \mathbb{T}_{S}^{\sigma}(d_i v_0)_{w_i+(g-1)d_i(\sum_{i>j}d_j-\sum_{i<j}d_j)}
    \right\rangle. 
    \end{align}
    The right hand side is after all partitions $(d_i)_{i=1}^k$ of 
    $d$ and all weights $(w_i)_{i=1}^k\in\mathbb{Z}^k$ such that \[\frac{w_1}{d_1}<\cdots<\frac{w_k}{d_k},\] 
    and each fully-faithful functor in (\ref{intro:sod}) is given 
    by the categorical Hall product (\ref{intro:cathall}). 
\end{thm}
The order of the summands in the semiorthogonal 
    decomposition (\ref{intro:sod}) is not immediate to state 
    and we do not make it explicit in this paper,
    see Remark~\ref{rmk:order}. 
If $v$ is primitive, then \[\mathbb{T}_S^{\sigma}(v)_w=
D^b\left(\mathfrak{M}_S^{\sigma}(v)\right)_w.\]
In general, the category $\mathbb{T}_S^{\sigma}(v)_w$ is 
uniquely determined by the semiorthogonal decomposition (\ref{intro:sod}). 
Locally on $M_S^{\sigma}(v)$, the category 
$\mathbb{T}_S^{\sigma}(v)_w$ is defined to be
the subcategory of objects which are Koszul dual to 
matrix factorizations with some weight conditions 
for the maximal torus of the stabilizer groups. 
Such a subcategory was first considered by \v{S}penko--Van den Bergh~\cite{SVdB} 
to construct noncommutative crepant resolutions of quotients of quasi-symmetric representations by reductive 
groups. It was later used in~\cite{hls} to prove the ``magic window 
theorem" for GIT quotient stacks, and in~\cite{P} to 
give PBW type decompositions for categorical (and K-theoretic) Hall algebras of symmetric quivers with potential. 

We regard
the subcategory (\ref{intro:qBPS}) as a global version of these
categories in the case of moduli stacks of semistable objects on K3 
surfaces. The main tool in investigating the category \eqref{intro:qBPS} is its local description via categories of matrix factorizations on the moduli stacks of representations of Ext-quivers of $\sigma$-polystable objects. We study quasi-BPS categories in this local context in \cite{PTquiver}.


\subsection{Quasi-BPS categories for reduced stacks}\label{subsec12}
The derived stack $\mathfrak{M}_S^{\sigma}(v)$ is never 
classical because of the existence of the 
trace map $\Ext^2(E, E) \twoheadrightarrow \mathbb{C}$
for any object $E \in D^b(S)$. 
Let 
\begin{align*}
\mathfrak{M}_S^{\sigma}(v)^{\rm{red}} \hookrightarrow 
\mathfrak{M}_S^{\sigma}(v)
\end{align*}
be the reduced 
derived stack, which roughly speaking is obtained 
by taking  
the traceless part of its obstruction theory. 
By~\cite{KaLeSo}, it is known that the
reduced derived stack is classical when $g\geq 2$. 
We also have a reduced version of the 
quasi-BPS category 
\begin{align}\label{reduced:bps}
    \mathbb{T}_S^{\sigma}(v)^{\rm{red}}_w
    \subset D^b\left(\mathfrak{M}_S^{\sigma}(v)^{\rm{red}}\right)_w
\end{align}
and a reduced version of the semiorthogonal decomposition (\ref{intro:sod}), 
see Theorem~\ref{thm:sodK32}. 
When $v$ is primitive, we have 
\begin{align}\label{T:twist}
\mathbb{T}_S^{\sigma}(v)^{\rm{red}}_w=D^b\left(M_S^{\sigma}(v), \alpha^w\right),
\end{align}
where $\alpha$ is the Brauer class which represents the 
obstruction to the existence of a universal object, and 
$M_S^{\sigma}(v)$ is a projective holomorphic symplectic manifold~\cite{Mu2, BaMa2}. 
From the above description, we have the 
following properties of the 
category (\ref{reduced:bps})
when $v$ is primitive:
(i) the category $\mathbb{T}_S^{\sigma}(v)^{\rm{red}}_w$
is smooth and proper;
(ii) 
    the Serre functor $S_{\mathbb{T}}$ 
    of $\mathbb{T}_S^{\sigma}(v)^{\rm{red}}_w$ is isomorphic 
    to the shift functor $[\dim M_S^{\sigma}(v)]$;
    (iii)
    by Halpern-Leistner~\cite{halpK32}, 
    the category $\mathbb{T}_S^{\sigma}(v)^{\rm{red}}_w$
    is independent of $\sigma$ up to equivalence. 
The proof of the above properties 
relies on the description (\ref{T:twist})
for primitive $v$, and a priori there is no 
reason that these properties hold for non-primitive 
$v$. 
Nevertheless, we have the following: 
\begin{thm}\emph{(Corollary~\ref{cor:smooth}, Theorem~\ref{thm:Serre:etale}, Theorem~\ref{thm:walleq})}\label{intro:thm4}
Suppose that $g\geq 2$, $\sigma, \sigma'\in\mathrm{Stab}(S)$ are generic stability conditions, and $w$ is coprime to $v$. Then:

\textbf{(i) (smooth and properness):} the category 
$\mathbb{T}_S^{\sigma}(v)_w^{\rm{red}}$ is smooth and proper; 

\textbf{(ii) (\'etale locally trivial Serre functor):}
the Serre functor $S_{\mathbb{T}}$
of $\mathbb{T}_S^{\sigma}(v)^{\rm{red}}_w$ is 
trivial \'etale locally on $M_S^{\sigma}(v)$; 

\textbf{(iii) (wall-crossing equivalence):} 
there is an equivalence
$\mathbb{T}_S^{\sigma}(v)_w^{\rm{red}} \simeq
\mathbb{T}_S^{\sigma'}(v)_w^{\rm{red}}$. 
Hence we may write the quasi-BPS category as $\mathbb{T}_S(v)_w^{\rm{red}}$. 
\end{thm}
The key point in the proof
of (i) above is Lemma \ref{prop:catsupp} (the categorical support 
lemma), which says that
any object in $\mathbb{T}_S(v)_w^{\rm{red}}$
has a nilpotent singular support if $w$ is coprime to $v$. 
By combining with 
the strong generation,  
we conclude that $\mathbb{T}_S(v)_w^{\rm{red}}$ is smooth 
and proper if $w$ is coprime to $v$.
In particular, 
it admits a Serre functor $S_{\mathbb{T}}$. 
We expect that $S_{\mathbb{T}}$ is globally isomorphic 
to $[\dim M_S^{\sigma}(v)]$. However, 
currently there is a technical subtlety of proving 
this, and we only prove it is trivial \'etale locally
in (ii). 
Globally, we prove an isomorphism
$S_{\mathbb{T}}\cong [\dim M_S^{\sigma}(v)]$
on the level of cohomologies, see Corollary~\ref{cor:isomcoh}, and also 
for perfect complexes, see~Corollary~\ref{cor:serreperf}. 
In view of parts (i) and (ii)  of Theorem \ref{intro:thm4},
we view $\mathbb{T}_S(v)^{\mathrm{red}}_w$ as a categorical version of a crepant resolution of $M^\sigma_S(v).$

It is an interesting question to see the relation 
between (reduced) quasi-BPS categories and categorical crepant resolutions in the sense of Kuznetsov~\cite{KuzICM} or noncommutative crepant 
resolutions in the sense of Van den Bergh \cite{VdB22}. We plan to investigate this relation in future work.   

The main tool in proving Theorem \ref{intro:thm4} is its local version for stacks of representations of preprojective algebras constructed from Ext-quivers of $\sigma$-polystable objects, see \cite{PTquiver}. Along the way, we obtain generation statements for singular support quotient categories of more general quasi-smooth stacks that may be of independent interest, see Theorem \ref{thm:regular2}.

\subsection{Topological K-theory of quasi-BPS categories}\label{subsec:intro:topK}
We finally relate topological K-theory of quasi-BPS category with the cohomology of the BPS sheaf 
$\mathcal{BPS}_v$
on $M_S^{\sigma}(v)$ studied in~\cite{DHSM} (i.e. with BPS cohomology). 
Note that $\mathcal{BPS}_v=\mathrm{IC}_{M_S^{\sigma}(v)}=\mathbb{Q}_{M_S^{\sigma}(v)}[\dim M_S^{\sigma}(v)]$ if $v$ is a primitive Mukai vector and $\sigma$ is generic, and in general it is a semisimple perverse sheaf which contains $\mathrm{IC}_{M_S^{\sigma}(v)}$ as a proper direct summand.

For a dg-category $\mathcal{D}$ and $i\in \mathbb{Z}$,
we denote by $K_i^{\rm{top}}(\mathcal{D})$
the topological K-theory of $\mathcal{D}$ defined by Blanc~\cite{Blanc}. 
We prove the following: 
\begin{thm}\emph{(Theorem~\ref{thmKtop})}\label{intro:thm:K}
Suppose that $\sigma$ is a generic Gieseker stability condition, $g\geq 2$, and 
$w$ is coprime to $v$. 
For $i\in \mathbb{Z}$, we have the identity:
\begin{align*}
    &\dim K_{i}^{\rm{top}}(\mathbb{T}_S^{\sigma}(v)_w) =
    \sum_{j\in \mathbb{Z}}\dim 
    H^{i+2j}(M_S^{\sigma}(v), \mathcal{BPS}_v). 
\end{align*}
\end{thm}
The above result is motivated by 
categorification of BPS invariants in Donaldson-Thomas theory, 
which will be explained in the next subsection. 
We regard Theorem~\ref{intro:thm:K} as a weight-independence phenomenon reminiscent of the (numerical and cohomological) $\chi$-independence phenomenon 
\cite{MT, TodGV, MaulikShen, KinjoKoseki}. 
It is an interesting problem to define a primitive part $\mathrm{P}K^{\mathrm{top}}_i(\mathbb{T}_S^{\sigma}(v)_w)\subset K^{\mathrm{top}}_i(\mathbb{T}_S^{\sigma}(v)_w)$ whose dimension is independent of \textit{all} $w\in\mathbb{Z}$.

Theorem \ref{intro:thm:K} can be seen as part of the more general problem of categorifying perverse sheaves of interest \cite{PTtop}, \cite{P3}. Such a problem is the first step in categorifying instances of the BBDG decomposition theorem \cite{BBD}. In the context of good moduli space maps for objects in certain Calabi-Yau $2$-categories, a BBDG-type decomposition was proved by Davison \cite{DavPurity}. 
Theorem \ref{intro:thm1} can be seen as a partial categorification of the decomposition theorem for the morphism $\mathcal{M}^\sigma_S(v) \to M^\sigma_S(v)$.

\subsection{Motivation from Donaldson-Thomas theory}
We now explain how the study of 
quasi-BPS categories is motivated by DT theory. 
Let $X$ be a smooth Calabi-Yau 3-fold. For a given numerical class $v$
and a stability condition $\sigma$ on $D^b(X)$, 
the DT invariant is defined to be a rational number
\begin{align}\label{intro:DT}
\mathrm{DT}^{\sigma}(v) \in \mathbb{Q}
\end{align}
which virtually counts $\sigma$-semistable (compactly supported) objects with 
numerical class $v$, see~\cite{Thom, JS, PiYT}. 
It is defined via the moduli stack $\mathcal{M}_X^{\sigma}(v)$ of 
$\sigma$-semistable objects with numerical class 
$v$ or its good moduli space \[\mathcal{M}_X^{\sigma}(v) \to M_X^{\sigma}(v).\] 
When $\sigma$-semistable objects coincide with $\sigma$-stable objects 
(e.g. $v$ is primitive and $\sigma$ is generic), then the DT invariant 
is an integer and can be also computed as 
\begin{align*}
\mathrm{DT}^{\sigma}(v)=\int_{[M_X^{\sigma}(v)]^{\rm{vir}}}1 =
\int_{M_X^{\sigma}(v)} \chi_B \ de \in \mathbb{Z},
\end{align*}
where $\chi_B$ is the Behrend constructible function~\cite{MR2600874}.
Otherwise, (\ref{intro:DT}) is defined as the weighted Euler characteristic
with respect to the Behrend function of the `log' of $\mathcal{M}_X^{\sigma}(v)$ 
in the motivic Hall algebra, see~\cite{JS}.

For a generic $\sigma$, the BPS invariant $\Omega^{\sigma}(v)$ is inductively defined 
by the multiple cover formula
\begin{align*}
    \mathrm{DT}^{\sigma}(v)=\sum_{k\geq 1, k|v}\frac{1}{k^2} \Omega^{\sigma}(v/k). 
\end{align*}
Although (\ref{intro:DT}) is a rational number in general, the BPS number 
$\Omega^{\sigma}(v)$ is an integer. 
The integrality of $\Omega^{\sigma}(v)$ is conjectured in~\cite[Conjecture~6]{K-S}, \cite[Conjecture~6.12]{JS} 
and proved in~\cite{DM} 
combined with~\cite{Todext}. 
We address the following categorification problem of BPS invariants:
\begin{problem}\label{prob1}
Is there a dg-category $\mathbb{T}^{\sigma}(v)$
which recovers $\Omega^{\sigma}(v)$ by taking the Euler characteristic of an additive invariant, e.g. 
\[\chi(K^{\rm{top}}(\mathbb{T}^{\sigma}(v))):=\dim_\mathbb{Q} K_0^{\rm{top}}(\mathbb{T}^{\sigma}(v))_\mathbb{Q}-\dim_\mathbb{Q} K_1^{\rm{top}}(\mathbb{T}^{\sigma}(v))_\mathbb{Q}=
-\Omega^{\sigma}(v)?\]
\end{problem}
The above problem is open even if $v$ is primitive, and in this case 
it is related to the gluing problem of matrix factorizations,
see~\cite{T} for the case of local surfaces 
and \cite{RHH} for work in progress addressing the general case. 

Now, for a K3 surface $S$, we consider the 
local K3 surface 
\begin{align}\label{loc:K3}
    X=\mathrm{Tot}_S(K_S)=S \times \mathbb{A}^1_\mathbb{C}. 
\end{align}
The ($\mathbb{C}^{\ast}$-equivariant) DT category for the moduli 
stack $\mathcal{M}_X^{\sigma}(v)$ is defined in~\cite{T} via 
categorical dimensional reduction 
\begin{align*}
    \mathcal{DT}(\mathcal{M}_X^{\sigma}(v)):=D^b(\mathfrak{M}_S^{\sigma}(v)). 
\end{align*}
We regard the subcategory 
$\mathbb{T}_S^{\sigma}(v)_w \subset \mathcal{DT}(\mathcal{M}_X^{\sigma}(v))$
as a categorification of the BPS invariant for local K3 surface when $(v, w)$ is coprime. 
Indeed, Theorem~\ref{intro:thm:K} implies that 
\begin{align}\label{isom:Ktop}
\chi(K^{\rm{top}}(\mathbb{T}_S^{\sigma}(v)_w))=-\Omega^{\sigma}(v),
\end{align}
where the right hand side is explicitly computed in terms of 
Hilbert schemes of points, see the next subsection. 
Thus the category $\mathbb{T}_S^{\sigma}(v)_w$ gives a solution to Problem~\ref{prob1} for the local K3 surface \eqref{loc:K3}. 

\subsection{Motivation from hyperkähler geometry}
Let $S$ be a K3 surface, and consider the local K3 surface (\ref{loc:K3}). 
The BPS invariant in this case is completely known: 
\begin{align}\label{BPS:S}
    \Omega^{\sigma}(v)=-\chi(S^{[\langle v, v \rangle/2+1]}), 
\end{align}
where, for a positive integer $n$, we denote by $S^{[n]}$ the Hilbert scheme of $n$ points on $S$.
The above identity is conjectured by the second named author~\cite{TodK3} and 
proved by Maulik--Thomas~\cite[Corollary 6.10]{MTK3}. The identity (\ref{BPS:S}) is an instance of 
\textit{the $\chi$-independence phenomena} (e.g. when $v=(0, \beta, \chi)$, the right hand side of (\ref{BPS:S}) is 
independent of $\chi$), see~\cite[Conjecture~6.3]{MR2892766}, \cite[Conjecture~2.15]{TodGV}
and~\cite{MaulikShen, KinjoKoseki} for the recent development of $\chi$-independence phenomena. 

When $v$ is primitive, the identity (\ref{BPS:S}) holds since
$M_S^{\sigma}(v)$ is a holomorphic symplectic manifold~\cite{Mu2}
deformation equivalent to $S^{[\langle v, v\rangle/2+1]}$. 
However, it is much less obvious and mysterious when $v$ is not primitive. 
For non-primitive $v$, 
the good moduli space $M=M_S^{\sigma}(v)$ is a singular 
symplectic variety. O'Grady~\cite{Ogra1}
constructed a symplectic 
resolution of singularities
\begin{equation}\label{symplecticresolution}
    \widetilde{M} \to M
\end{equation}
when $v=2v_0$ for a primitive $v_0$ with 
$\langle v_0, v_0 \rangle=2$. 
But this turned out to be the 
only exceptional case:
Kaledin--Lehn--Sorger~\cite{KaLeSo}
proved that 
$M$ does not admit a 
symplectic resolution
in all other cases with $\langle v, v\rangle \geq 2$. 
By~\cite[Proposition~1.1]{Funil}, the existence of a symplectic resolution \eqref{symplecticresolution} is equivalent 
to the existence of a crepant resolution of $M$, so $M$ does not admit a crepant resolution except in the example studied by O'Grady.
Instead of a usual (geometric) crepant resolution, 
it is interesting to investigate if 
$M$ admits a crepant resolution
of singularities in a categorical sense:

\begin{problem}\label{prob:2}
Is there a categorical version of a crepant resolution of $M_S^{\sigma}(v)$?
\end{problem}




Inspired by Theorem \ref{intro:thm4}, we regard the category 
$\mathbb{T}_S(v)_w^{\rm{red}}$
as a categorical version of a
(twisted, \'etale local) crepant resolution of $M_S^{\sigma}(v)$. Note that, even in the situation of the O'Grady resolution \eqref{symplecticresolution}  (that is, if $d=2$ and $\langle v_0, v_0\rangle=2$), the category $\mathbb{T}_S(2)_1^{\mathrm{red}}$ is different from $D^b(\widetilde{M})$ because its topological K-theory is a proper direct summand of the topological K-theory of $\widetilde{M}$, see by Theorem \ref{intro:thm:K} and \cite{MR4338453}.

In view of (\ref{isom:Ktop}) and (\ref{BPS:S}), we further 
expect $\mathbb{T}_S(v)_w^{\rm{red}}$
to be a 
``non-commutative hyperkähler variety" deformation equivalent to 
$S^{[\langle v, v\rangle/2+1]}$. 
In particular, it is natural to investigate how $\mathbb{T}_S(v)_w^{\rm{red}}$ is analogous to $D^b(M)$ 
for a smooth projective hyperkähler variety of $K3^{[\langle v, v\rangle/2+1]}$-type. 
More precisely, we may expect the following, 
which we regard as a categorical $\chi$-independence 
phenomenon: 
\begin{conj}\emph{(Conjecture~\ref{conj:HK})}\label{conj:HK2}
For any $g\geq 0$ and any $w\in \mathbb{Z}$ coprime with $v$, the category $\mathbb{T}_S(v)_{w}^{\rm{red}}$ 
 is deformation equivalent to $D^b\left(S^{[\langle v, v\rangle/2+1]}\right)$. 
	\end{conj}

Recall that $\langle v_0, v_0\rangle=2g-2$.
The above conjecture is easy to check for $g=0$, see Proposition~\ref{prop:g=0}. 
For $g=1$, we conjecture that 
the category $\mathbb{T}_S(v)_w^{\rm{red}}$ is 
equivalent to the derived category of a K3 surface (possibly twisted and not necessarily isomorphic to $S$), and we show this follows from an explicit computation of the quasi-BPS categories of $\mathbb{C}^3$ studied in \cite{PTzero, PT1}, 
see Conjectures~\ref{conj:K3}, \ref{conj:C2} and Proposition~\ref{prop:conj}.
In the forthcoming
paper~\cite{PaTobps}, we prove 
Conjecture~\ref{conj:C2} for $d=2$, 
which implies Conjecture~\ref{conj:HK2} for $(d, w)=(2, 1)$. 
More precisely, there is an equivalence 
$D^b(S) \stackrel{\sim}{\to} \mathbb{T}_S(2v_0)_1^{\rm{red}}$ in this case. 

\subsection{Acknowledgements}
We thank Tasuki Kinjo, Davesh Maulik, Yalong Cao, Junliang Shen, Georg Oberdieck, and Jørgen Rennemo for discussions 
related to this work. T.~P.~is grateful to Columbia University in New York and to Max Planck Institute for Mathematics in Bonn for their hospitality and financial support during the writing of this paper.
The project of this paper started when Y.~T.~was visiting Columbia University 
in April 2023. Y.~T.~thanks Columbia University for their hospitality. 
	Y.~T.~is supported by World Premier International Research Center
	Initiative (WPI initiative), MEXT, Japan, and Grant-in Aid for Scientific
	Research grant (No.~19H01779) from MEXT, Japan.

\section{Preliminaries}

In this section, we introduce notations and review definitions related to stacks, matrix factorizations, and window categories. We also include a table with the most important notation we use later in the paper.

\begin{figure}
	\centering
\scalebox{0.7}{
	\begin{tabular}{|l|l|l|}
		\hline
	Notation & Description & Location defined \\\hline
 $S^{[n]}$ & Hilbert scheme of $n$ points on $S$ &
Equation \eqref{BPS:S}
\\ \hline
$\widehat{X}_y, \widehat{X}$ & formal fiber &
Subsection \ref{notation}
\\ \hline
$\langle \mathcal{S} \rangle$ & category generated by set $\mathcal{S}$ &
Subsection \ref{notation2}
\\ \hline
$\llangle S \rrangle$ & cocomplete subcategory generated by set $\mathcal{S}$ &
Subsection \ref{notation2}
\\ \hline
$\Theta$ & Koszul equivalence &
Equation \eqref{equiv:Phi}
\\ \hline
$\mathscr{Z}, \mathscr{S}$ & fixed and attracting stacks &
Equation \eqref{attracting}
\\ \hline
$p, q$ & maps from the attracting stack & Equation \eqref{attracting}
\\ \hline
$\mathfrak{M}$ & quasi-smooth stack & Subsection \ref{subsection:etale}
\\ \hline
$\mathcal{M}$ & classical truncation of a quasi-smooth stack & Subsection \ref{subsection:etale}
\\ \hline
$\pi\colon \mathcal{M}\to M$ & good moduli space map & Subsection \ref{subsection:etale}
\\ \hline
$\Omega_\mathfrak{M}[-1]$ & $(-1)$-shifted cotangent stack & Subsection \ref{subsec:shifted}
\\ \hline
$\mathcal{C}_\mathscr{Z}$ & subcategory of objects with singular support in $\mathscr{Z}$ & Subsection \ref{subsec:shifted}
\\ \hline
$\mathbb{W}^\ell_{m_\bullet}$ & window category for a quasi-smooth stack & Subsection \ref{subsec:shifted}
\\ \hline
$D^b(\mathfrak{M})/\mathcal{C}_\mathscr{Z}$ & singular support quotient & Subsection \ref{subsec:shifted}
\\ \hline
$\mathbb{W}^{\mathrm{int}}_\delta$ & intrinsic window (``magic") category & Definition \ref{def:intwind}
\\ \hline
$\bm{d}$ & dimension vector of a quiver &
Subsection \ref{subsec311}
\\ \hline
$\X(\bm{d})=R_Q(\bm{d})/G(\bm{d})$ & stack of representations of quiver $Q$ &
Equation \eqref{stack:X}
\\ \hline
$Q^{\circ,d}$ & doubled quiver &
Subsection \ref{subsec312}
\\ \hline
$\mathscr{P}(\bm{d})$ & stack of representations of preprojective algebra &
Equation \eqref{P:dzero}
\\ \hline
$\mathscr{Y}(\bm{d})$ & stack of representations of the doubled quiver &
Equation \eqref{P:dzero}
\\ \hline
$(Q,W)$ & tripled quiver with potential &
Subsection \ref{subsec:triple}
\\ \hline
$M(\bm{d})\, (M(\bm{d})_\mathbb{R}$ etc.) & weight lattice (space etc.) of maximal torus $T(\bm{d})$ of $G(\bm{d})$ & Subsection \ref{subsec:qBPS}
\\ \hline
$W$ & Weyl group of $G(\bm{d})$ & Subsection \ref{subsec:qBPS}
\\ \hline
$\underline{\bm{d}}$ & total dimension & Subsection \ref{subsec:qBPS}
\\ \hline
$\tau_{\bm{d}}$ & Weyl-invariant weight with sum of coefficients $1$ & Subsection \ref{subsec:qBPS}
\\ \hline
$\bm{W}(\bm{d})$ & polytope constructed from the weights of $R_Q(\bm{d})$ & Subsection \ref{subsec:qBPS}
\\ \hline
$\mathbb{T}(\bm{d})_{\delta}$, $\mathbb{T}(\bm{d})_w$ & quasi-BPS categories & Definition \ref{def:Tdelta}
\\ \hline
$\mathbb{T}^\ell(\bm{d})_{\delta}$, $\mathbb{T}^\ell(\bm{d})_w$ & quasi-BPS categories for a stability condition & Subsection \ref{subsec:double}
\\ \hline
$N(S)$ & numerical Grothendieck group & Subsection \ref{subsec41}
\\ \hline
$\mathfrak{M}^\sigma_S(v)$ & moduli stack of sheaves on K3 surface &
Subsection \ref{subsection:moduliK3}
\\ \hline
 $\mathfrak{M}^\sigma_S(v)^{\mathrm{red}}$ & reduced moduli stack of sheaves on K3 surface &
Subsection \ref{subsection:moduliK3}
\\ \hline
$\mathbb{T}^\sigma_S(v)_\delta$, $\mathbb{T}^\sigma_S(v)_\delta^{\mathrm{red}}$ & quasi-BPS categories for K3 surfaces (for a line bundle) & 
Equation \eqref{qbps:T}
\\ \hline
$\mathbb{T}^\sigma_{S,y}(v)_{\delta_y}$ & formal quasi-BPS category & 
Remark \ref{rmk:qbps}
\\ \hline
$\mathbb{T}^\sigma_S(v)_w$, $\mathbb{T}^\sigma_S(v)_w^{\mathrm{red}}$ & quasi-BPS categories for K3 surfaces (for a weight) & 
Definition \ref{def:qbps0}
\\ \hline
$\widehat{\mathbb{T}}_p(d)_w$, $\widehat{\mathbb{S}}^{\mathrm{gr}}_p(d)_w$ & formal quasi-BPS categories generated from the ambient space &
Subsection \ref{subsec:prop}
\\ \hline
$\mathcal{N}_{\mathrm{nil}}$ & substack where an endomorphism is nilpotent &
Subsection \ref{subsec62}
\\ \hline
$\mathrm{Et}/M$ & category of \'etale morphisms over $M$ &
Subsection \ref{subsec:ssuport:gen}
\\ \hline
$\mathcal{T}_U$ & presheaf of singular support quotients &
Equation \eqref{def:mathcaltu}
\\ \hline
$S_\mathbb{T}$, $S_{\mathbb{T}/M}$ & absolute and relative Serre functor & 
Subsection \ref{subsec71}
\\ \hline
$a_\mathcal{P}$ & trace map for matrix factorizations & 
Equation \eqref{induce:ap}
\\ \hline
$\mathrm{tr}_\mathcal{E}$ & trace map for coherent sheaves & 
Equation \eqref{const:tre}
\\ \hline
$S^d_w$  & set of partitions of $d$ & Subsection \ref{subsec81}
\\ \hline
 $\mathcal{BPS}_A$  & BPS sheaf for $A\in S^d_w$ & Subsection \ref{subsec81}
\\ \hline
$\mathcal{BPS}_v$,  $\mathcal{BPS}_{v,w}$  & BPS sheaves for K3 surfaces & Subsection \ref{subsec81}
\\ \hline
	\end{tabular}
}
	\vspace{.5cm}
	 \caption{Notation used in the paper}
	 \label{table:notation}
\end{figure}

 \subsection{Notations for (derived) stacks}\label{notation}
All the spaces $\mathscr{X}$ considered are quasi-smooth (derived) stacks over $\mathbb{C}$, see \cite[Subsection 3.1]{T} for references. 
The classical truncation of $\mathscr{X}$ is denoted by 
$\mathscr{X}^{\rm{cl}}$. 
We denote by $\mathbb{L}_\mathscr{X}$ the cotangent complex of 
$\mathscr{X}$.

For $G$ an algebraic group and $X$ a dg-scheme with an action of $G$, denote by $X/G$ the corresponding
quotient stack. 
When $X$ is affine, we denote by $X\ssslash G$ the quotient dg-scheme with dg-ring of regular functions $\mathcal{O}_X^G$. 
For a morphism $f \colon X\to Y$ and for a closed point $y \in Y$, 
we denote by $\widehat{X}_y$ the 
following base change
\begin{align*}
\widehat{X}_y := X \times_{Y} \Spec \left(\widehat{\mathcal{O}}_{Y, y}\right).
\end{align*}
We call $\widehat{X}_y$ the \textit{formal fiber}, though it is a scheme over a complete 
local ring rather than a formal scheme. 
When $X$ is a $G$-representation, $f\colon X\to Y:=X\ssslash G$, and $y=0$, we omit the subscript $y$ from the above notation. 

We use the terminology of \textit{good moduli spaces} of Alper, see \cite[Section 8]{MR3237451} for examples of stacks with good moduli spaces.

\subsection{DG-categories}

For $\X$ a quasi-smooth stack,
we denote by $D^b(\X)$ the bounded derived category 
of coherent sheaves and by $\mathrm{Perf}(\X)$ the subcategory of perfect complexes on $\X$,
see~Subsection~\ref{subsec:qsmooth} for more details and for more categories of (quasi)coherent sheaves. 

\subsubsection{Generation of dg-categories}\label{notation2}
Any dg-category considered is a $\mathbb{C}$-linear 
pre-triangulated dg-category, in particular its homotopy category is a 
triangulated category. 
For a pre-triangulated dg-category $\mathcal{D}$ and 
a full subcategory $ \mathcal{C} \subset \mathcal{D}$, we say 
that $\mathcal{C}$ \textit{classically generates} $\mathcal{D}$ if $\mathcal{D}$ 
coincides with the smallest thick pre-triangulated subcategory 
of $\mathcal{D}$ which contains $\mathcal{C}$. 
If $\mathcal{D}$ is furthermore cocomplete, then 
we say that $\mathcal{C}$ \textit{generates} $\mathcal{D}$ if 
$\mathcal{D}$ coincides with the smallest thick pre-triangulated 
subcategory of $\mathcal{C}$ which contains $\mathcal{C}$ and is closed 
under taking colimits. 

We also recall some terminology related to 
strong generation. 
For a set of objects $\mathcal{S} \subset \mathcal{D}$, 
we denote by $\langle \mathcal{S} \rangle$
the smallest subcategory which contains $S$ 
and is closed under shifts, finite direct sums, and 
direct summands. 
If $\mathcal{D}$ is cocomplete, 
we denote by $\llangle S \rrangle$
the smallest subcategory 
which contains $S$ and is closed under shifts, arbitrary 
direct sums, and direct summands. 
For subcategories $\mathcal{C}_1, \mathcal{C}_2 \subset 
\mathcal{D}$, we denote by 
$\mathcal{C}_1 \star \mathcal{C}_2 \subset \mathcal{D}$
the smallest subcategory which contains 
objects $E$ which fit into distinguished triangles
$A_1 \to E \to A_2\to A_1[1]$ with $A_i \in \mathcal{C}_i$ for $i\in\{1,2\}$,
and is closed under shifts, finite direct sums, and direct 
summands. 
We say that $\mathcal{D}$ is \textit{strongly generated} by 
$C \in \mathcal{D}$ if
$\mathcal{D}=\langle C \rangle^{\star n}$
for some $n\geq 1$.  
This is equivalent to 
$\Ind \mathcal{D}=\llangle C \rrangle^{\star n}$
for some $n\geq 1$, see~\cite[Proposition~1.9]{Neeman}.
A dg-category $\mathcal{D}$ is called \textit{regular} 
if it 
has a strong generator. 
A dg-category $\mathcal{D}$ is called \textit{smooth} if the diagonal dg-module of $\mathcal{D}$ is perfect. It is proved in~\cite[Lemma~3.5, 3.6]{VL}
that if $\mathcal{D}$ is smooth, then $\mathcal{D}$ is regular.




\subsubsection{Semiorthogonal decompositions}
Let $R$ be a set. Consider a set $O\subset R\times R$ such that for any $i, j\in R$ we have $(i,j)\in O$, or $(j,i)\in O$, or both $(i,j)\in O$ and $(j,i)\in O$. 
Let $\mathbb{T}$ be a pre-triangulated dg-category. We will construct semiorthogonal decompositions
\begin{equation}\label{sodinitial}
\mathbb{T}=\langle \mathbb{A}_i \mid i \in R \rangle
\end{equation}
with summands 
pre-triangulated subcategories $\mathbb{A}_i$ indexed by $i\in R$
such that for any $i,j\in R$ with $(i, j)\in O$ and for any objects $\mathcal{A}_i\in\mathbb{A}_i$, $\mathcal{A}_j\in\mathbb{A}_j$, we have 
$\Hom_{\mathbb{T}}(\mathcal{A}_i,\mathcal{A}_j)=0$. 

Let $\pi\colon \X\to S$ be a morphism from a quasi-smooth stack to a scheme $S$ and assume $\mathbb{T}$ is a subcategory of $D^b(\X)$. We say the decomposition \eqref{sodinitial} is \textit{$S$-linear} if $\mathbb{A}_i\otimes \pi^*\mathrm{Perf}(S)\subset \mathbb{A}_i$.

\subsection{Graded matrix factorizations}\label{subsec:graded}
References for this subsection are \cite[Section 2.2]{T3}, \cite[Section~2.2]{T}, \cite[Section~2.3]{MR3895631}, \cite[Section~1]{PoVa3}.
Let $G$ be an algebraic group and let $Y$ be a smooth affine scheme with an action of $G$. 
Let 
$\mathscr{Y}=Y/G$ be the corresponding quotient stack
and let $f$ be
a regular function 
\[f \colon \mathscr{Y}\to \mathbb{C}.\]
Assume that there exists an extra action of $\mathbb{C}^{\ast}$
on $Y$ which commutes with the action of $G$ on $Y$, trivial on 
$\mathbb{Z}/2 \subset \mathbb{C}^{\ast}$, and $f$ is weight two 
with respect to the above $\mathbb{C}^{\ast}$-action. 

Consider the category of graded matrix factorizations 
\begin{align*}\mathrm{MF}^{\mathrm{gr}}(\mathscr{Y}, f).
\end{align*}
Its objects are pairs $(P, d_P)$ with $P$ a $G\times\mathbb{C}^*$-equivariant coherent sheaf on $Y$ and $d_P \colon P\to P(1)$ a $G\times\mathbb{C}^*$-equivariant morphism
satisfying $d_P^2=f$. Here $(1)$ is the twist by the character 
$\mathrm{pr}_2 \colon G \times \mathbb{C}^{\ast} \to \mathbb{C}^{\ast}$. 
Note that
as the $\mathbb{C}^{\ast}$-action is trivial on $\mathbb{Z}/2$, 
we have the induced action of
$\mathbb{C}^{\star}=\mathbb{C}^{\ast}/(\mathbb{Z}/2)$ on $Y$
and $f$ is weight one with respect to the above $\mathbb{C}^{\star}$-action. 
The objects of $\mathrm{MF}^{\mathrm{gr}}(\mathscr{Y}, f)$ can be alternatively described as tuples 
\begin{align}\label{tuplet:graded}
(E, F, \alpha \colon 
E\to F(1)', \beta \colon F\to E),
\end{align}
where $E$ and $F$ are $G\times\mathbb{C}^{\star}$-equivariant coherent sheaves
on $Y$, $(1)'$ is the twist by the character 
$G \times \mathbb{C}^{\star} \to \mathbb{C}^{\star}$, 
and $\alpha$ and $\beta$ are $\mathbb{C}^{\star}$-equivariant morphisms
such that $\alpha\circ\beta$ and $\beta\circ\alpha$ are multiplication by $f$.

For a pre-triangulated subcategory $\mathbb{M}$ of $D^b(\mathscr{Y})$, 
define $\mathrm{MF}^{\rm{gr}}(\mathbb{M}, f)$ as
the full subcategory of $\mathrm{MF}^{\rm{gr}}(\mathscr{Y}, f)$
with objects totalizations of 
pairs $(P, d_{P})$ with $P \in \mathbb{M}$ equipped with $\mathbb{C}^{\ast}$-equivariant 
structure, see~\cite[Subsection~2.6.2]{PTzero}. 
If $\mathbb{M}$ is generated by
a set of vector bundles $\{\mathscr{V}_i\}_{i\in I}$
on $\mathscr{Y}$, then $\mathrm{MF}^{\rm{gr}}(\mathbb{M}, f)$ is generated by 
matrix factorizations whose 
factors are direct sums of vector bundles 
from $\{\mathscr{V}_i\}_{i\in I}$, 
see~\cite[Lemma~2.3]{PTzero}. 

Functoriality of categories of graded matrix factorizations for pullback and proper pushfoward is discussed in \cite{PoVa3}.
In Subsection~\ref{subsec:trace}, we 
will also consider the category 
$D^{\rm{gr}}(Y)$ for a possibly singular 
affine variety $Y$ with a $\mathbb{C}^{\ast}$-action 
as above. It consists of objects (\ref{tuplet:graded})
with $f=0$, so its definition is the same 
as $\mathrm{MF}^{\rm{gr}}(Y, 0)$, but when $Y$ is singular an object (\ref{tuplet:graded}) 
may not be isomorphic to the one 
such that $E, F$ are locally free of finite rank. 
See~\cite{EfPo} for factorization categories over
possibly singular varieties. 
Note that if the $\mathbb{C}^{\ast}$-action on $Y$
is trivial, then $D^{\rm{gr}}(Y)=D^b(Y)$.

\subsection{The Koszul equivalence}
Let $Y$ be a smooth affine scheme
with an action of an algebraic group $G$, let $\mathscr{Y}=Y/G$, and let $V$ be a $G$-equivariant vector bundle on $Y$. 
We always assume that $Y$ is either of finite type 
over $\mathbb{C}$, or is a formal fiber 
of a map $X \to X\ssslash H$ for a finite type scheme $X$ and an algebraic group $H$ as in Subsection \ref{notation}. 
Let $\mathbb{C}^*$ act on the fibers of $V$ with weight $2$ and consider a section 
$s\in \Gamma(Y, V)$. 
It induces a map $\partial \colon V^{\vee} \to \mathcal{O}_Y$. 
Let $s^{-1}(0)$ be the derived zero locus of $s$ with dg-algebra of regular functions
\begin{align}\label{Kscheme}
\mathcal{O}_{s^{-1}(0)}:=\mathcal{O}_Y\left[V^{\vee}[1];\partial\right].
\end{align}
Consider the quotient (quasi-smooth) stack
\begin{equation}\label{koszuldef}
\mathscr{P}:=s^{-1}(0)/G.
\end{equation}
We call $\mathscr{P}$ the \textit{Koszul stack} associated with 
$(Y, V, s, G)$. 
There is a natural inclusion
\[j\colon \mathscr{P}\hookrightarrow\mathscr{Y}.\]
The section $s$ also induces the regular function \begin{equation}\label{defreg}
f\colon \mathscr{V}^{\vee}:=\text{Tot}_Y\left(V^{\vee}\right)/G\to\mathbb{C}
\end{equation}
defined by
$f(y,v)=\langle s(y), v \rangle$ for $y\in Y(\mathbb{C})$ and $v\in V^{\vee}|_y$.
Consider the category of graded matrix factorizations $\text{MF}^{\text{gr}}\left(\mathscr{V}^{\vee}, f\right)$ with respect to the $\mathbb{C}^*$-action
mentioned above. The Koszul equivalence, also called dimensional reduction in the literature, says the following:

\begin{thm}\emph{(\cite{I, Hirano, T})}\label{thm:Kduality}
There is an equivalence 
\begin{align}\label{equiv:Phi}
	\Theta \colon D^b(\mathscr{P}) \stackrel{\sim}{\to}
	\mathrm{MF}^{\mathrm{gr}}(\mathscr{V}^{\vee}, f) 
	\end{align}
	given by $\Theta(-)=\mathcal{K}\otimes_{\mathcal{O}_{\mathscr{P}}}(-)$, where $\mathcal{K}$ is the Koszul factorization, see~\cite[Theorem~2.3.3]{T}. 
	\end{thm}
 We will use the following lemma: 
\begin{lemma}\emph{(\cite[Lemma~2.6]{PTquiver})}\label{lem:genJ}
    Let $\{V_a\}_{a\in A}$ be a set of
    $G$-representations
    and let $\mathbb{S} \subset \mathrm{MF}^{\rm{gr}}(\mathscr{V}^{\vee}, f)$
    be the subcategory generated by 
    matrix factorizations whose 
    factors are direct sums of vector bundles $\mathcal{O}_{\mathscr{V}^{\vee}} \otimes V_a$. 
    Then an object $\mathcal{E} \in D^b(\mathscr{P})$
    satisfies $\Theta(\mathcal{E}) \in \mathbb{S}$
    if and only if $j_{\ast}\mathcal{E} \in D^b(\mathscr{Y})$
    is generated by $\mathcal{O}_{\mathscr{Y}} \otimes V_a$ for $a\in A$. 
\end{lemma}

\subsection{Window categories}\label{subsection:window}
\subsubsection{Attracting stacks}\label{attractingloci}
Let $Y$ be an affine variety with an action of a reductive group $G$. Let $\lambda$ be a cocharacter of $G$. 
Let $G^\lambda$ and $G^{\lambda\geq 0}$ be the Levi and parabolic groups associated to $\lambda$. Let $Y^\lambda\subset Y$ be the closed subvariety of $\lambda$-fixed points.
Consider the attracting variety \[Y^{\lambda\geq 0}:=\{y\in Y|\,\lim_{t\to 0}\lambda(t)\cdot y\in Y^\lambda\}\subset Y.\]
Consider the attracting and fixed stacks
\begin{equation}\label{attracting}
\mathscr{Z}:=Y^\lambda/G^\lambda \xleftarrow{q}\mathscr{S}:=Y^{\lambda\geq 0}/G^{\lambda\geq 0}\xrightarrow{p}\mathscr{Y}.
\end{equation}
The map $p$ is proper. 
Kempf-Ness strata are connected components of certain attracting stacks $\mathscr{S}$, and the map $p$ restricted to a Kempf-Ness stratum is a closed immersion, see~\cite[Section~2.1]{halp}. 
The attracting stacks also appear in the definition of Hall algebras \cite{P0} (for $Y$ an affine space), where the Hall product is induced by the functor 
\begin{equation}\label{Hallproductquiver}
    \ast:=p_*q^*\colon D^b(\mathscr{Z})\to D^b(\mathscr{Y}).
\end{equation}
In this case, the map $p$ may not be a closed immersion.

Let $T \subset G$ be a maximal torus and let
$\lambda$ be a cocharacter $\lambda \colon \mathbb{C}^{\ast} \to T$. 
For a $G$-representation $Y$, 
the attracting variety
$Y^{\lambda \geq 0} \subset Y$
coincides with the sub $T$-representation
generated by weights which pair non-negatively with $\lambda$. 
We may abuse notation and denote by 
$\langle \lambda, Y^{\lambda \geq 0} \rangle:= \langle \lambda, \det Y^{\lambda \geq 0} \rangle$, 
where $\det Y^{\lambda \geq 0}$ is the sum of $T$-weights of $Y^{\lambda \geq 0}$.

\subsubsection{The definition of window categories}
Let $Y$ be an affine variety with an action of a reductive group $G$ and a linearization $\ell$. Consider the stacks 
\[j\colon\mathscr{Y}^{\ell\text{-ss}}:=Y^{\ell\text{-ss}}/G\hookrightarrow\mathscr{Y}:=Y/G.\]
We review the construction of window categories of $D^b(\mathscr{Y})$ which are equivalent to $D^b(\mathscr{Y}^{\ell\text{-ss}})$ via the restriction map, due to
Segal \cite{MR2795327}, Halpern-Leistner \cite{halp}, and Ballard--Favero--Katzarkov \cite{MR3895631}. We follow the presentation from \cite{halp}.

By also fixing a Weyl-invariant norm on the 
cocharacter lattice, 
the unstable locus $\mathscr{Y}\setminus \mathscr{Y}^{\ell\text{-ss}}$ has a stratification in Kempf-Ness strata $\mathscr{S}_i$ for $i\in I$ a finite ordered set:
\[\mathscr{Y}\setminus \mathscr{Y}^{\ell\text{-ss}}=\bigsqcup_{i\in I}\mathscr{S}_i.\]
A Kempf-Ness stratum $\mathscr{S}_i$ is the attracting stack
in $\mathscr{Y} \setminus \bigsqcup_{j<i}\mathcal{S}_j$
for
a cocharacter $\lambda_i$, with the fixed stack $\mathscr{Z}_i:=\mathscr{S}_i^{\lambda_i}$. 
Let $N_{\mathscr{S}_i/\mathscr{Y}}$ be the normal bundle of $\mathscr{S}_i$ in $\mathscr{Y}$.
 Define the width of the window categories \[\eta_i:=\left\langle \lambda_i, \det\left(N^{\vee}_{\mathscr{S}_i/\mathscr{Y}}|_{\mathscr{Z}_i}\right)
 \right\rangle.\]
  For a choice of real numbers $m_{\bullet}=(m_i)_{i\in I}\in \mathbb{R}^I$, define the category
\begin{equation}\label{def:window}
\mathbb{G}^\ell_{m_{\bullet}}:=\{\mathcal{F}\in D^b(\mathscr{Y})\text{ such that } \mathrm{wt}_{\lambda_i}(\mathcal{F}|_{\mathscr{Z}_i})\subset [
m_i, m_i+\eta_i) \text{ for all }i\in I\}.
\end{equation}
In the above, $\mathrm{wt}_{\lambda_i}(\mathcal{F}|_{\mathscr{Z}_i})$
is the set of $\lambda_i$-weights on $\mathcal{F}|_{\mathscr{Z}_i}$.
Then \cite[Theorem 2.10]{halp}
says that the restriction functor 
$j^*$ induces an equivalence of categories: 
\begin{equation}\label{jequiv}
j^*\colon \mathbb{G}^\ell_{m_{\bullet}}\xrightarrow{\sim} D^b\big(\mathscr{Y}^{\ell\text{-ss}}\big)
\end{equation}
for any choice of real numbers $m_{\bullet}=(m_i)_{i\in I}\in \mathbb{R}^I$.


\subsection{Quasi-smooth derived stacks}\label{subsec:qsmooth}
\subsubsection{Derived categories of (quasi-)coherent sheaves}
 Let $\mathfrak{M}$ be a derived stack over $\mathbb{C}$ and let $\mathcal{M}$ be its classical truncation. Let $\mathbb{L}_{\mathfrak{M}}$ be the cotangent complex 
of $\mathfrak{M}$. The stack $\mathfrak{M}$ is called \textit{quasi-smooth} if for all closed points $x\to \mathcal{M}$, the restriction $\mathbb{L}_\mathfrak{M}|_x$ has cohomological amplitude in $[-1, 1]$.
By \cite[Theorem 2.8]{BBBJ}, a stack $\mathfrak{M}$ is quasi-smooth if and only if it is a $1$-stack and any point of $\mathfrak{M}$ lies in the image of a $0$-representable smooth morphism \begin{equation}\label{alpha}
    \alpha \colon \mathscr{U}\to\mathfrak{M}
\end{equation} for a Koszul scheme $\mathscr{U}$ as in \eqref{Kscheme}.
Let $D_{\rm{qc}}(\mathscr{U})$ be the derived category of dg-modules over $\mathcal{O}_{\mathscr{U}}$ and let
$D^b(\mathscr{U}) \subset D_{\rm{qc}}(\mathscr{U})$ be the subcategory of objects with
bounded coherent cohomologies. Further, let $\Ind D^b(\mathscr{U})$ be the ind-completion of $D^b(\mathscr{U})$ \cite{MR3136100}.
For a quasi-smooth stack $\mathfrak{M}$, the dg-categories $D_{\rm{qc}}(\mathfrak{M})$, 
$D^b(\mathfrak{M})$, and $\Ind D^b(\mathfrak{M})$ are defined to be limits in the $\infty$-category of smooth morphisms \eqref{alpha}, see~\cite[Subsection 3.1.1]{T}, \cite{MR3136100}:
\begin{align*}
D_{\rm{qc}}(\mathfrak{M})=\lim_{\mathscr{U} \to \mathfrak{M}}
D_{\rm{qc}}(\mathscr{U}), \ 
    D^b(\mathfrak{M})=\lim_{\mathscr{U} \to \mathfrak{M}} D^b(\mathscr{U}), \   \Ind D^b(\mathfrak{M})=\lim_{\mathscr{U} \to \mathfrak{M}} \Ind D^b(\mathscr{U}).
\end{align*}
The category $\Ind D^b(\mathfrak{M})$ is a module over $D_{\rm{qc}}(\mathfrak{M})$
via the tensor product. 
For $\mathcal{E}_1, \mathcal{E}_2 \in \Ind D^b(\mathfrak{M})$, there exist an
internal homomorphism, see~\cite[Remark~3.4.5]{MR3037900}
\begin{align*}
    \mathcal{H}om(\mathcal{E}_1, \mathcal{E}_2) \in D_{\rm{qc}}(\mathfrak{M}),
\end{align*}
such that for any $\mathcal{A} \in D_{\rm{qc}}(\mathfrak{M})$ we have 
\begin{align*}
    \Hom_{D_{\rm{qc}}(\mathfrak{M})}(\mathcal{A}, \mathcal{H}om(\mathcal{E}_1, \mathcal{E}_2))
    \cong \Hom_{\Ind D^b(\mathfrak{M})}(\mathcal{A} \otimes \mathcal{E}_1, \mathcal{E}_2). 
\end{align*}
If $\mathfrak{M}$ is QCA (quasi-compact and with affine automorphism groups)~\cite[Definition 1.1.8]{MR3037900}, then 
$\Ind D^b(\mathfrak{M})$
is compactly generated with compact objects
$D^b(\mathfrak{M})$, see~\cite[Theorem~3.3.5]{MR3037900}.

\subsubsection{\'Etale and formal local structures along good moduli 
spaces}\label{subsection:etale}
Let $\mathfrak{M}$ be a quasi-smooth stack over $\mathbb{C}$ and let $\mathcal{M}$ be its classical truncation.
Suppose that $\mathcal{M}$ admits a good moduli space map
\[\pi\colon \mathcal{M} \to M,\] 
see~\cite{MR3237451} for the notion of a good moduli space.
In particular, $M$ is an algebraic space and $\pi$ is a quasi-compact morphism.
For each point in $M$, there are an \'{e}tale
  neighborhood $U \to M$
  and Cartesian squares: 
  \begin{align}\label{dia:fnbd}
	\xymatrix{
		\mathfrak{M}_U \ar[d] \ar@{}[rd]|\square
		& \ar@<-0.3ex>@{_{(}->}[l] \mathcal{M}_U
		\ar[r] \ar[d] \ar@{}[rd]|\square & U \ar[d] \\
		\mathfrak{M}  & \ar@<-0.3ex>@{_{(}->}[l] \mathcal{M} \ar[r] & M,
	}
\end{align}
where each vertical arrow is \'{e}tale
and $\mathfrak{M}_U$ is equivalent to a Koszul stack 
$\mathscr{P}=s^{-1}(0)/G$ as in (\ref{koszuldef}), see~\cite[Subsection~3.1.4]{T}, \cite[Theorem 4.2.3]{halpK32}, \cite{AHR}.
Similarly, for each closed point $y \in M$, there exist 
Cartesian squares, see~\cite[Subsection~3.1.4]{T}:
  \begin{align}\label{dia:fnbd2}
	\xymatrix{
		\widehat{\mathfrak{M}}_y \ar[d] \ar@{}[rd]|\square
		& \ar@<-0.3ex>@{_{(}->}[l] \widehat{\mathcal{M}}_y
		\ar[r] \ar[d] \ar@{}[rd]|\square & \Spec \widehat{\mathcal{O}}_{M, y} \ar[d] \\
		\mathfrak{M}  & \ar@<-0.3ex>@{_{(}->}[l] \mathcal{M} \ar[r] & M.
	}
\end{align}

\subsubsection{\texorpdfstring{$(-1)$}--shifted cotangent stacks}\label{subsec:shifted}
Let $\mathfrak{M}$ be a quasi-smooth stack. Let $\mathbb{T}_\mathfrak{M}$ be the tangent complex of $\mathfrak{M}$, which is the dual complex to the cotangent complex $\mathbb{L}_\mathfrak{M}$.
We denote by $\Omega_{\mathfrak{M}}[-1]$ the \textit{$(-1)$-shifted cotangent stack} of $\mathfrak{M}$:
\begin{align*}
\Omega_\mathfrak{M}[-1]:=\mathrm{Spec}_\mathfrak{M}\left(\mathrm{Sym}(\mathbb{T}_\mathfrak{M}[1])\right).
\end{align*}
Consider the projection map 
\begin{equation}\label{p0}
p_0\colon \mathcal{N}:=\Omega_\mathfrak{M}[-1]^{\rm{cl}}\to \mathfrak{M}.
\end{equation}
For a Koszul stack $\mathscr{P}$ as in \eqref{koszuldef}, recall the function $f$ from \eqref{defreg} and consider the critical locus $\mathrm{Crit}(f)\subset \mathscr{V}^{\vee}$. 
In this case, the map $p_0$ is the natural projection
\begin{equation}\label{p1}
p_0\colon \mathrm{Crit}(f)=\Omega_\mathscr{P}[-1]^{\rm{cl}}\to \mathscr{P}.
\end{equation}
For an object $\mathcal{F} \in D^b(\mathfrak{M})$, 
Arinkin--Gaitsgory \cite{AG} defined the notion of singular support  
denoted by 
\begin{align*}
    \mathrm{Supp}^{\rm{sg}}(\mathcal{F}) \subset \mathcal{N}. 
\end{align*}
The definition is compatible with maps $\alpha$ as in \eqref{alpha}, see \cite[Section~7]{AG}. 
Consider the group $\mathbb{C}^*$ scaling the fibers of the map $p_0$. 
A closed substack $\mathscr{Z}$ of $\mathcal{N}$ is called \textit{conical} if it is closed under the action of $\mathbb{C}^*$. 
The singular support $\mathrm{Supp}^{\rm{sg}}(\mathcal{F})$ of $\mathcal{F}$ is a 
conical subset $\mathscr{Z}$ of $\mathcal{N}$. 
For a given conical closed substack $\mathscr{Z} \subset \mathcal{N}$, we denote by 
$\mathcal{C}_{\mathscr{Z}} \subset D^b(\mathfrak{M})$
the subcategory of objects whose singular supports are 
contained in $\mathscr{Z}$. 

Consider a Koszul stack $\mathscr{P}$ as in \eqref{koszuldef} and recall the Koszul equivalence $\Theta$ from \eqref{equiv:Phi}. Under $\Theta$, the singular support of $\mathcal{F}\in D^b(\mathscr{P})$ corresponds to the support $\mathscr{Z}$ of the matrix factorization $\Theta(\mathcal{F})$, namely the maximal closed substack $\mathscr{Z}\subset \mathrm{Crit}(f)$ such that $\mathcal{F}|_{\mathscr{V}^{\vee}\setminus \mathscr{Z}}=0$ in $\mathrm{MF}^{\mathrm{gr}}(\mathscr{V}^{\vee}\setminus \mathscr{Z}, f)$, see \cite[Subsection 2.3.9]{T}.

\subsection{The window theorem for quasi-smooth stacks}\label{windowfirst}
We review
the theory of window categories for singular support quotients of 
quasi-smooth stacks \cite[Chapter 5]{T}, which itself is 
inspired by 
Halpern-Leistner's theory of window categories for $0$-shifted symplectic 
derived stacks~\cite{halpK32}. We continue with the notation from the previous subsection.

Let $\mathfrak{M}$ be a quasi-smooth stack and assume
throughout this subsection that its classical truncation $\mathcal{M}$ admits a good moduli space 
$\mathcal{M} \to M.$
Let $\ell$ be a line bundle on $\mathcal{M}$
and let $b \in H^4(\mathcal{M}, \mathbb{Q})$ be a positive definite class, see~\cite[Definition~3.7.6]{Halpinstab}. 
We also use the same symbols $(\ell, b)$ for 
$p_0^{\ast}\ell \in \mathrm{Pic}(\mathcal{N})$ and 
$p_0^{\ast}b \in H^4(\mathcal{N}, \mathbb{Q})$. 
Then there is a $\Theta$-stratification 
with respect to $(\ell, b)$:
\begin{align}\label{theta:N}
    \mathcal{N}=\mathscr{S}_1 \sqcup \cdots \sqcup \mathscr{S}_N\sqcup \mathcal{N}^{\ell\text{-ss}}
\end{align}
with centers $\mathscr{Z}_i\subset\mathscr{S}_i$, see~\cite[Theorem 5.2.3, Proposition 5.3.3]{Halpinstab}.  
In the above situation, an analogue of the window theorem is 
proved in~\cite[Theorem 1.1]{Totheta}, \cite[Theorem 5.3.13]{T}
(which generalizes~\cite[Theorem~3.3.1]{halpK32} in the case that 
$\mathfrak{M}$ is $0$-shifted symplectic): 
\begin{thm}\emph{(\cite{T, Totheta})}\label{thm:window:M}
In addition to the above, 
suppose that $\mathcal{M} \to M$ satisfies the formal neighborhood theorem, see below. 
Then for each $m_{\bullet}=(m_i)_{i=1}^N\in \mathbb{R}^N$, 
 there is a subcategory $\mathbb{W}(\mathfrak{M})^\ell_{m_{\bullet}} \subset D^b(\mathfrak{M})$
 such that the composition 
 \begin{align*}
     \mathbb{W}(\mathfrak{M})^\ell_{m_{\bullet}} \subset D^b(\mathfrak{M}) \twoheadrightarrow
     D^b(\mathfrak{M})/\mathcal{C}_{\mathscr{Z}}
 \end{align*}
 is an equivalence, where $\mathscr{Z}:=\mathcal{N} \setminus \mathcal{N}^{\ell\text{-ss}}$. 
\end{thm}
\begin{remark} When $\mathcal{N}$ is a (global) quotient stack
$\mathcal{N}=Y/G$
for a reductive algebraic group $G$, 
a 
$\Theta$-stratification \eqref{theta:N} is the same as a Kempf-Ness stratification \cite[Example 0.0.5]{Halpinstab}. 
The class $b$ is then constructed as the pull-back of the class in $H^4(BG, \mathbb{Q})$ corresponding to the chosen positive definite form \cite[Example 5.3.4]{Halpinstab}.
\end{remark}

\begin{remark}\label{rmk:0shift}
Suppose that $\mathbb{L}_{\mathfrak{M}}$ is self-dual, e.g. 
$\mathfrak{M}$ is 0-shifted symplectic. 
In this case, we have 
$\mathcal{N}^{\ell\text{-ss}}=\Omega_{\mathfrak{M}^{\ell\text{-ss}}}[-1]^{\rm{cl}}$, 
which easily follows from~\cite[Lemma~4.3.22]{halpK32}. Then we have the equivalence, see~\cite[Lemma~3.2.9]{T}:
\begin{align*}
  D^b(\mathfrak{M})/\mathcal{C}_{\mathscr{Z}}  \stackrel{\sim}{\to} 
  D^b(\mathfrak{M}^{\ell\text{-ss}}). 
\end{align*}
    
\end{remark}
We now explain the meaning of ``the formal neighborhood theorem" in the statement of Theorem \ref{thm:window:M}, see \cite[Definition 5.2.3]{T}. 
For a closed point $y \in M$, denote also by $y \in \mathcal{M}$
the unique closed point in the fiber of $\mathcal{M} \to M$
at $y$.
Set $G_y:=\mathrm{Aut}(y)$,
which is a reductive algebraic group. 
Let $\widehat{\mathcal{M}}_y$ be the formal fiber
along with $\mathcal{M} \to M$ at $y$. Let $\widehat{\mathcal{H}}^0(\mathbb{T}_{\mathcal{M}}|_{y})$ be the formal fiber at the origin of $\mathcal{H}^0(\mathbb{T}_{\mathcal{M}}|_{y})\to \mathcal{H}^0(\mathbb{T}_{\mathcal{M}}|_{y})\sslash G_y$ at the origin, and define $\widehat{\mathcal{H}}^0(\mathbb{T}_{\mathfrak{M}}|_{y})$ similarly, see also the convention from Subsection \ref{notation}. 
Then the formal 
neighborhood theorem says that 
there is a $G_y$-equivariant morphism 
\begin{align*}
\kappa_y \colon \widehat{\mathcal{H}}^0(\mathbb{T}_{\mathcal{M}}|_{y}) 
\to \mathcal{H}^1(\mathbb{T}_{\mathcal{M}}|_{y}) 
\end{align*}
such that, by setting $\mathcal{U}_y$
to be the classical zero locus of $\kappa_y$,
there is an isomorphism  
$\widehat{\mathcal{M}}_y \cong \mathcal{U}_y/G_y$. 
Let $\mathfrak{U}_y$ be the derived zero locus of 
$\kappa_y$. 
Then, by replacing $\kappa_y$ if necessary, 
$\widehat{\mathfrak{M}}_y$ is equivalent to 
$\mathfrak{U}_y/G$, see~\cite[Lemma~5.2.5]{T}. 

Below we give a formal local 
description of $\mathbb{W}(\mathfrak{M})_{m_{\bullet}}^\ell$. 
Consider the pair of a smooth
stack and a regular function $(\mathscr{X}_y, f_y)$:
\begin{align*}
    \mathscr{X}_y:=\left(\widehat{\mathcal{H}}^0(\mathbb{T}_{\mathfrak{M}}|_{y})
    \times \mathcal{H}^1(\mathbb{T}_{\mathfrak{M}}|_{y})^{\vee}\right)/G_y
    \stackrel{f_y}{\to} \mathbb{C},
\end{align*}
where $f_y(u, v)=\langle \kappa_y(u), v \rangle$. 
From (\ref{p1}), the critical locus of $f_y$
is isomorphic to the classical 
truncation of the $(-1)$-shifted cotangent stack over $\widehat{\mathfrak{M}}_y$, 
so it is isomorphic to the formal fiber 
$\widehat{\mathcal{N}}_y$
of $\mathcal{N} \to \mathcal{M} \to M$
at $y$. 
The pull-back of the $\Theta$-stratification 
(\ref{theta:N}) to $\widehat{\mathcal{N}}_y$
gives a Kempf-Ness stratification 
\begin{align*}
    \widehat{\mathcal{N}}_y=\widehat{\mathscr{S}}_{1, y} \sqcup 
    \cdots \sqcup \widehat{\mathscr{S}}_{N, y} \sqcup \widehat{\mathcal{N}}_y^{\ell\text{-ss}}
\end{align*}
with centers $\widehat{\mathscr{Z}}_{i, y}\subset \widehat{\mathscr{S}}_{i, y}$
and one parameter subgroups $\lambda_i \colon \mathbb{C}^{\ast} \to G_y$. 
By the Koszul equivalence, 
see Theorem~\ref{thm:Kduality}, 
there is an equivalence:
\begin{align*}
    \Theta_y \colon D^b(\widehat{\mathfrak{M}}_y) \stackrel{\sim}{\to}
    \mathrm{MF}^{\text{gr}}(\mathscr{X}_y, f_y). 
\end{align*}
Then the subcategory $\mathbb{W}(\mathfrak{M})_{m_{\bullet}}^\ell$ in Theorem~\ref{thm:window:M} is characterized 
as follows: 
an object $\mathcal{E} \in D^b(\mathfrak{M})$ is an object of 
$\mathbb{W}(\mathfrak{M})^\ell_{m_{\bullet}}$
if and only if, for any closed point $y \in M$, we have 
\begin{align}\label{PhiEy}
    \Theta_{y}(\mathcal{E}|_{\widehat{\mathfrak{M}}_y}) \in \mathrm{MF}^{\text{gr}}(\mathbb{G}^\ell_{m_{\bullet}'}, f_y), \ 
  m_i'=m_i-\left\langle \lambda_i, \det\left(\mathcal{H}^1(\mathbb{T}_{\mathfrak{M}}|_{y})^{\lambda_i>0}\right) \right\rangle.  
\end{align}
The category $\mathbb{G}^{\ell}_{m'_{\bullet}}$ is the window category \eqref{def:window} for the weights $(m'_i)_{i=1}^N$ and the 
line bundle $\ell$.
The difference between $m_i$ and $m_i'$ is due to the discrepancy of categorical Hall products on 
$\mathfrak{M}_y$ and $\mathscr{X}_y$, see \cite[Proposition 3.1]{P2}.

\subsection{Intrinsic window subcategory}
We continue to consider a quasi-smooth derived stack $\mathfrak{M}$ whose classical truncation $\mathcal{M}$ admits a good moduli 
space $\mathcal{M} \to M$. 
We say that $\mathfrak{M}$ is \textit{symmetric} if for any closed point 
$y \in \mathfrak{M}$, the $G_y:=\mathrm{Aut}(y)$-representation 
\begin{align*}
	\mathcal{H}^0(\mathbb{T}_{\mathfrak{M}}|_{y}) \oplus 
	\mathcal{H}^1(\mathbb{T}_{\mathfrak{M}}|_{y})^{\vee}
\end{align*}
is a self-dual $G_y$-representation. 
In this subsection, we assume that $\mathfrak{M}$ is symmetric. 
Let $\delta \in \mathrm{Pic}(\mathfrak{M})_{\mathbb{R}}$. 
We now define a different kind of window categories, called \textit{intrinsic window subcategories}
$\mathbb{W}(\mathfrak{M})_{\delta}^{\rm{int}} \subset D^b(\mathfrak{M})$, 
see~\cite[Definition~5.2.12, 5.3.12]{T}. These categories are the quasi-smooth version of ``magic window categories" from \cite{SVdB, hls}.

First, assume that 
$\mathfrak{M}$ is a Koszul stack 
associated with $(Y, V, s, G)$ as in 
(\ref{koszuldef})
\begin{align}\label{present:M}
	\mathfrak{M}=s^{-1}(0)/G.
\end{align}
Consider the quotient stack $\mathscr{Y}=Y/G$, the closed immersion
$j \colon \mathfrak{M} \hookrightarrow \mathscr{Y}$, and let $\mathscr{V} \to \mathscr{Y}$ be
the total space of $V/G \to Y/G$. 
In this case, we define 
$\mathbb{W}(\mathfrak{M})_{\delta}^{\rm{int}} \subset D^b(\mathfrak{M})$ to be 
consisting of 
$\mathcal{E} \in D^b(\mathfrak{M})$ such that 
for any map $\nu \colon B\mathbb{C}^{\ast} \to \mathfrak{M}$ we have 
\begin{align*}
	\mathrm{wt}(\nu^{\ast}j^{\ast}j_{\ast}\mathcal{E}) 
	\subset \left[\frac{1}{2}\mathrm{wt}\left(\det \nu^{\ast}(\mathbb{L}_{\mathscr{V}}|_{\mathscr{Y}})^{\nu<0}
	\right), 
	\frac{1}{2}\mathrm{wt}\left(\det \nu^{\ast}(\mathbb{L}_{\mathscr{V}}|_{\mathscr{Y}})^{\nu>0}\right)
	\right]
	+\mathrm{wt}(\nu^{\ast}\delta). 
\end{align*}
The above subcategory $\mathbb{W}(\mathfrak{M})_{\delta}^{\rm{int}}$
is intrinsic to $\mathfrak{M}$, that is, independent of a 
choice of a presentation $\mathfrak{M}$ as (\ref{present:M})
for $(Y, V, s, G)$,  
see~\cite[Lemma~5.3.14]{T}.

In general, 
the intrinsic window subcategory is defined as follows
(which generalizes the magic window category in~\cite[Definition~4.3.5]{halpK32}
considered when $\mathbb{L}_{\mathfrak{M}}$ is self-dual):

\begin{defn}(\cite[Definition~5.3.12]{T})\label{def:intwind}
	We define the subcategory 
	\begin{align*}
		\mathbb{W}(\mathfrak{M})_{\delta}^{\rm{int}} \subset D^b(\mathfrak{M})
	\end{align*}
	to be consisting of objects $\mathcal{E}$ such that,
	for any étale morphism 
	$\iota_U\colon U \to M$ such that $\mathfrak{M}_U$ is of the form $s^{-1}(0)/G$
	as in (\ref{present:M}) and $\iota_U$ induces an étale morphism $\iota_U \colon \mathfrak{M}_U \to \mathfrak{M}$, 
	we have 
	$\iota_U^{\ast}\mathcal{E} \in \mathbb{W}(\mathfrak{M}_U)_{\iota_U^{\ast}\delta}^{\rm{int}}
	\subset D^b(\mathfrak{M}_U)$. 
\end{defn}

\section{Quasi-BPS categories for doubled quivers}\label{sec:qbps:double}
In this section, we review the results in~\cite{PTquiver} about quasi-BPS categories 
of doubled quivers, focusing on the example of doubled quivers of $g$-loop quivers for $g\geq 1$. 
These results are the local analogues of Theorems \ref{intro:thm1} and \ref{intro:thm4}. 
We also discuss 
similar results for formal fibers along good moduli space morphisms. 

\subsection{Moduli stacks of representations of quivers}
\subsubsection{Moduli stacks}\label{subsec311}
Let $Q=(I, E)$ be a quiver. 
For a dimension vector $\bm{d}=(d^{(a)})_{a \in I} \in \mathbb{N}^{I}\subset \mathbb{Z}^I$, 
we denote by 
\begin{align}\label{stack:X}
\mathscr{X}(\bm{d})=R_Q(\bm{d})/G(\bm{d})
\end{align}
the moduli stack of $Q$-representations of dimension $\bm{d}$. 
Here, the affine space $R_Q(\bm{d})$ and the reductive group $G(\bm{d})$ are 
defined by 
\begin{align*}
    R_Q(\bm{d})=\bigoplus_{(a \to b) \in E}
    \Hom(V^{(a)}, V^{(b)}), \ 
    G(\bm{d})=\prod_{a \in I}GL(V^{(a)}),
\end{align*}
where $V^{(a)}$ is a $\mathbb{C}$-vector space of dimension $d^{(a)}$. 
We denote by $\mathfrak{g}(\bm{d})$ the Lie algebra of 
$G(\bm{d})$. 

\subsubsection{Doubled quivers}\label{subsec312}
 Let $Q^{\circ}=(I, E^{\circ})$ be a quiver. Let $E^{\circ \ast}$ be the set of edges $e^{\ast}=(b \to a)$
    for each $e=(a \to b)$ in $E^{\circ}$. 
    Consider the \textit{doubled quiver} of $Q^\circ$: 
    \begin{align*}
    Q^{\circ, d}=(I, E^{\circ, d}), \ 
    E^{\circ, d}=E^{\circ} \sqcup E^{\circ \ast}.
    \end{align*}
    Let $\mathscr{I}$ be the quadratic relation $\sum_{e \in E^{\circ}} [e, e^{\ast}]\in \mathbb{C}[Q^{\circ, d}]$. 
   For a dimension vector $\bm{d}=(d^{(a)})_{a \in I}$, 
the relation $\mathscr{I}$
induces a moment map:  
\begin{align}\label{moment}
\mu \colon 
    R_{Q^{\circ, d}}(\bm{d})=T^*R_{Q^\circ}(\bm{d}) \to \mathfrak{g}(\bm{d}). 
\end{align}
The derived zero locus 
\begin{align}\label{P:dzero}
   \mathscr{P}(\bm{d}):=\mu^{-1}(0)/G(\bm{d}) \stackrel{j}{\hookrightarrow}
    \mathscr{Y}(\bm{d}):=R_{Q^{\circ, d}}(\bm{d})/G(\bm{d})
\end{align}
is the derived moduli stack of $(Q^{\circ, d}, \mathscr{I})$-representations of dimension vector 
$\bm{d}$. Note that a $(Q^{\circ, d}, \mathscr{I})$-representation is the same as a representation of the preprojective algebra $\pi_{Q^\circ}:=\mathbb{C}[Q^{\circ, d}]/(\mathscr{I})$ of $Q^\circ$, and we will use these two names interchangeably.

\subsubsection{Tripled quivers}\label{subsec:triple}
Consider a quiver $Q^{\circ}=(I, E^{\circ})$.
For $a\in I$, let $\omega_a$ be a loop at $a$. 
\textit{The tripled quiver} of $Q^\circ$ is: 
\begin{align*}
    Q=(I, E), \ E=E^{\circ, d}\sqcup \{\omega_a\}_{a \in I}. 
\end{align*}
The tripled potential $W$ of $Q$ is: 
    \begin{align*}
        W=\left(\sum_{a \in I}\omega_a \right) \left(
        \sum_{e \in E^{\circ}}[e, e^{\ast}] \right). 
    \end{align*}
    Consider the stack (\ref{stack:X}) of representations of dimension $d$ for the tripled quiver $Q$:
    \[\mathscr{X}(\bm{d}):=R_{Q}(\bm{d})/G(\bm{d}):=R_{Q^{\circ, d}}(\bm{d}) \oplus \mathfrak{g}(\bm{d})/G(d).\]
The potential $W$ induces the regular function: 
\begin{align}\label{moduli:triple}
  \Tr W\colon  \mathscr{X}(\bm{d})=R_{Q}(\bm{d})/G(\bm{d}) \to \mathbb{C}.
\end{align}
We have the Koszul duality equivalence, 
see Theorem~\ref{thm:Kduality}:
\begin{align}\label{Koszul:theta}
    \Theta \colon D^b(\mathscr{P}(\bm{d}))
    \stackrel{\sim}{\to} \mathrm{MF}^{\rm{gr}}(\mathscr{X}(\bm{d}), \Tr W). 
\end{align}


\subsection{The weight lattice}\label{subsec:qBPS}
Let $Q=(I, E)$ be a quiver. 
For a dimension vector $\bm{d} \in \mathbb{N}^I$, 
let $T(\bm{d}) \subset G(\bm{d})$ be the maximal torus
and let $M(\bm{d})$ be the character lattice for $T(\bm{d})$:
\begin{align*}
    M(\bm{d})=\bigoplus_{a\in I} \bigoplus_{1\leq i\leq d^{(a)}}
    \mathbb{Z} \beta_i^{(a)}. 
\end{align*}
Here $\beta_1^{(a)}, \ldots, \beta_{d^{(a)}}^{(a)}$ are the weights of 
the standard representation of $GL(V^{(a)})$ for $a\in I$. 
In the case that $I$ consists of one element, 
we omit the subscript $(a)$. 
We denote by $\rho \in M(\bm{d})_{\mathbb{Q}}$ half of the sum of the
positive roots of $\mathfrak{g}(\bm{d})$. 
Let $W$ be the Weyl group of $G(\bm{d})$ and
let $M(\bm{d})_{\mathbb{R}}^W \subset M(\bm{d})_{\mathbb{R}}$ be the Weyl invariant 
subspace. There is a decomposition: 
\[M(\bm{d})_{\mathbb{R}}^W=\bigoplus_{a \in I} \mathbb{R}\sigma^{(a)},\]
where $\sigma^{(a)}:=\sum_{i=1}^{d^{(a)}}\beta_i^{(a)}$. 
There is a natural pairing: 
\begin{align*}
    \langle -, -\rangle \colon 
    M(\bm{d})_{\mathbb{R}}^W \times \mathbb{R}^I \to \mathbb{R}, \ 
    \langle \sigma^{(a)}, e^{(b)} \rangle=\delta^{ab}. 
\end{align*}
We denote by $\iota \colon M(\bm{d})_{\mathbb{R}} \to \mathbb{R}$ the linear 
map sending $\beta_i^{(a)}$ to $1$, 
and its kernel by $M(\bm{d})_{0, \mathbb{R}}$. 
An element $\ell \in M(\bm{d})_{0, \mathbb{R}}^W$ is written as 
\begin{align*}
    \ell=\sum_{a\in I}\ell^{(a)}\sigma^{(a)}, \ 
    \langle \ell, \bm{d} \rangle=
    \sum_{a}\ell^{(a)}d^{(a)}=0
\end{align*}
i.e. $\ell$ is an $\mathbb{R}$-character of $G(\bm{d})$
which is trivial on the diagonal torus 
$\mathbb{C}^{\ast} \subset G(\bm{d})$. 
Denote by $\underline{\bm{d}}:=\sum_{a\in I}d^{(a)}$ the total dimension.
Define the following Weyl invariant weight: 
\begin{align*}
\tau_{\bm{d}} :=\frac{1}{\underline{\bm{d}}} \cdot \sum_{a\in I, 1\leq i\leq d^{(a)}}\beta_i^{(a)}. 
\end{align*}
Define the polytope:
\begin{equation}\label{def:polytope}
\textbf{W}(\bm{d}):=\frac{1}{2}\mathrm{sum}[0, \beta]\subset M(\bm{d})_\mathbb{R},
\end{equation}
where the Minkowski sum is after all $T(\bm{d})$-weights $\beta$ of $R(\bm{d})$. 

\begin{defn}\label{def:generic}
A weight $\ell \in M(\bm{d})_{0, \mathbb{R}}^W$ is 
\textit{generic} if the following conditions hold:
\begin{itemize}
    \item if $H \subset M(\bm{d})_{0, \mathbb{R}}$ is a hyperplane parallel to a face in $\mathbf{W}(\bm{d})$ which contains $\ell$, then  $M(\bm{d})_{0, \mathbb{R}}^W \subset H$,
\item for any decomposition  $\bm{d}=\bm{d}_1+\bm{d}_2$ such that $\bm{d}_1, \bm{d}_2 \in \mathbb{N}^I$ are not proportional 
to $\bm{d}$, we have that $\langle \ell, \bm{d}_i\rangle \neq 0$
for $i\in\{1,2\}$.
\end{itemize}
    \end{defn}
Note that the set of generic weights is a dense open subset in $M(\textbf{d})^W_{0,\mathbb{R}}$.


\subsection{Quasi-BPS categories for stacks of representations of preprojective algebras}\label{subsec:double}
   Let $Q^{\circ}=(I, E^{\circ})$ be a quiver and let
   $\mathscr{P}(d)$ be the derived moduli stack of 
representations of its preprojective algebra (\ref{P:dzero}).    
For $\delta \in M(\bm{d})_{\mathbb{R}}^W$, define the quasi-BPS category
to be the intrinsic window subcategory in Definition~\ref{def:intwind}:
\begin{align}\label{def:Tdelta}
    \mathbb{T}(\bm{d})_\delta:=\mathbb{W}(\mathscr{P}(\bm{d}))_{\delta}^{\rm{int}} \subset D^b(\mathscr{P}(\bm{d})), \ 
    \mathbb{T}(\bm{d})_w :=\mathbb{T}(\bm{d}; w \tau_{\bm{d}}). 
\end{align} 
An alternative description is as follows, where recall the map $j\colon \mathscr{P}(\bm{d})\hookrightarrow\mathscr{Y}(\bm{d})$ and choose a dominant chamber $M(\bm{d})^+\subset M(\bm{d})$, for example the one in \cite[Subsection 2.2.2]{PTquiver}: 
\begin{lemma}\emph{(\cite[Corollary~3.20]{PTquiver})}\label{lem:compareT}
The subcategory (\ref{def:Tdelta}) consists of objects $\mathcal{E} \in D^b(\mathscr{P}(\bm{d}))$
such that $j_{\ast}\mathcal{E}$ is classically generated by the vector bundle
$\mathcal{O}_{\mathscr{Y}(d)} \otimes \Gamma_{G(\bm{d})}(\chi)$, where $\chi$ is a dominant weight
such that 
 \begin{align}\label{chi:rho}
        \chi+\rho-\delta \in \mathbf{W}(\bm{d}).
    \end{align}
    Here, $\Gamma_{G(\bm{d})}(\chi)$ is the irreducible representation of $G(\bm{d})$ with highest 
    weight $\chi$, and $\mathbf{W}(\bm{d})$ is the polytope \eqref{def:polytope} for the tripled quiver $Q$ of $Q^\circ$.
\end{lemma}

For $\ell \in M(\bm{d})_{0, \mathbb{R}}^W$, let 
$\mathscr{P}(\bm{d})^{\ell\text{-ss}} \subset \mathscr{P}(\bm{d})$ be the 
open substack of $\ell$-semistable locus. 
The quasi-BPS category for $\ell$-semistable locus 
is defined to be 
\begin{align*}
    \mathbb{T}^{\ell}(\bm{d})_\delta :=\mathbb{W}(\mathscr{P}(\bm{d})^{\ell\text{-ss}})_{\delta}^{\rm{int}}
    \subset D^b(\mathscr{P}(\bm{d})^{\ell\text{-ss}}), \ 
    \mathbb{T}^\ell(\bm{d})_{w} :=\mathbb{T}^\ell(\bm{d}; w\tau_{\bm d}). 
\end{align*}
Consider the restriction functor 
\begin{align}\label{rest:P}
\mathrm{res} \colon D^b(\mathscr{P}(\bm{d})) \twoheadrightarrow 
D^b(\mathscr{P}(\bm{d})^{\ell\text{-ss}}). 
\end{align}
We recall a wall-crossing equivalence proved in~\cite{PTquiver}: 
\begin{thm}\emph{(\cite[Corollary~3.19, Remark~3.12]{PTquiver})}\label{prop:eqS}
    For generic $\ell_+, \ell_- \in M(\bm{d})_{0, \mathbb{R}}^W$, 
    let $\delta'=\varepsilon_{+} \cdot \ell_{+} +\varepsilon_{-} \cdot \ell_{-}$
    for general $0<\varepsilon_{\pm} \ll 1$.
    Let $\delta\in M(\bm{d})^W_\mathbb{R}$
    and let $\delta''=\delta+\delta'$. Then the restriction functor (\ref{rest:P}) 
    induces equivalences:
    \begin{align*}
        \mathbb{T}(\bm{d})_{\delta''} \stackrel{\sim}{\to}  \mathbb{T}^{\ell_{\pm}}(\bm{d})_{ \delta''}. 
    \end{align*}
    In particular, there is an equivalence 
    $\mathbb{T}^{\ell_{+}}(\bm{d})_{\delta''} \simeq \mathbb{T}^{\ell_{-}}(\bm{d})_{ \delta''}$.
\end{thm}
\subsection{Semiorthogonal decompositions for preprojective algebras of quivers with one vertex}\label{subsec:onevert}
In the remaining of this section, we focus on the case 
of the $g$-loop quiver $Q^{\circ}=Q_g$
with loops $x_1, \ldots, x_g$. 
In this case, we write the dimension vector by $\bm{d}=d \in \mathbb{N}$. 
The doubled quiver is 
$Q^{\circ, d}=Q_{2g}$ with
loops $x_1, \ldots, x_g, y_1, \ldots, y_g$
and the 
relation $\mathscr{I}$ is given by 
$\sum_{i=1}^g[x_i, y_i]\in \mathbb{C}[Q_{2g}]$. 
The map (\ref{moment}) in this case is 
\begin{align*}
\mu \colon \mathfrak{gl}(d)^{\oplus 2g} \to \mathfrak{gl}(d), \ 
(x_1, \ldots, x_g, y_1, \ldots, y_g) \mapsto \sum_{i=1}^g [x_i, y_i]. 
\end{align*}
Then the derived stack in (\ref{P:dzero}) is 
\begin{align*}\mathscr{P}(d)=\mu^{-1}(0)/GL(d)
\hookrightarrow \mathscr{Y}(d)=\mathfrak{gl}(d)^{\oplus 2g}/GL(d).
\end{align*}

For a partition $d=d_1+\cdots+d_k$, let 
$\mathscr{P}(d_1, \ldots, d_k)$
be the derived moduli stack of 
filtrations
\begin{align*}
0=R_0\subset 
R_1 \subset \cdots \subset R_k
 \end{align*}
of $(Q^{\circ, d}, \mathscr{I})$-representations 
such that $R_i/R_{i-1}$ has dimension $d_i$. 
Explicitly, let $\lambda \colon \mathbb{C}^{\ast} \to T(d)$
be an antidominant cocharacter corresponding to the decomposition 
$d=d_1+\cdots+d_k$
and set 
\begin{align*}
    \mu^{\geq 0} \colon 
    \left(\mathfrak{gl}(d)^{\oplus 2g}\right)^{\lambda \geq 0} \to 
    \mathfrak{gl}(d)^{\lambda \geq 0}
\end{align*}
to be
the restriction 
of $\mu$. 
Then 
\begin{align*}
    \mathscr{P}(d_1, \ldots, d_k)=\left(\mu^{\geq 0}\right)^{-1}(0)/GL(d)^{\lambda \geq 0}. 
\end{align*}
Consider the evaluation morphisms 
\begin{align*}
    \times_{i=1}^k 
    \mathscr{P}(d_i) 
    \stackrel{q}{\leftarrow} \mathscr{P}(d_1, \ldots, d_k)
    \stackrel{p}{\to} \mathscr{P}(d). 
\end{align*}
The map $q$ is quasi-smooth and the map $p$ is proper. Consider
the categorical Hall product 
for the preprojective algebra of $Q^\circ$ \cite{PoSa, VaVa}: 
\begin{align}\label{chall:P}
    p_{\ast}q^{\ast} \colon 
    \boxtimes_{i=1}^k
    D^b(\mathscr{P}(d_i))
    \to D^b(\mathscr{P}(d)). 
\end{align}
We recall a result from~\cite{PTquiver}:
\begin{thm}\emph{(\cite[Theorem~4.20, Example~4.21]{PTquiver})}\label{cor:sodT}
There is a semiorthogonal decomposition 
\begin{align}\label{sod:triple}
D^b(\mathscr{P}(d))=
\left\langle  \boxtimes_{i=1}^k 
\mathbb{T}(d_i)_{w_i+(g-1)d_i(\sum_{i>j}d_j-\sum_{i<j}d_j)}\right\rangle
\end{align}
where the right hand side is after all partitions 
$(d_i)_{i=1}^k$ of $d$ and weights
$(w_i)_{i=1}^k \in \mathbb{Z}^k$ such that 
\begin{equation}\label{ineq:slopes}
\frac{w_1}{d_1}<\cdots<\frac{w_k}{d_k}.
\end{equation}
The fully-faithful functor 
\begin{align*}
 \boxtimes_{i=1}^k 
\mathbb{T}(d_i)_{w_i+(g-1)d_i(\sum_{i>j}d_j-\sum_{i<j}d_j)}
\to D^b(\mathscr{P}(d))
\end{align*}
is given by the categorical Hall product (\ref{chall:P}). The order is as in \cite[Subsection 3.4]{PTzero}, \cite[Subsection 4.6]{PTquiver}.
\end{thm}
\begin{remark}\label{rmk:order}
The semiorthogonal decomposition (\ref{sod:triple}) is obtained 
from that of $(2g+1)$-loop quiver and applying Koszul equivalence. 
The order of summands in (\ref{sod:triple}) is not immediate to state, and 
is the same one for the $(2g+1)$-loop quiver 
explained in~\cite[Subsection~4.6]{PTquiver}, see also~\cite[Subsection~3.4]{PTzero} for the case $g=1$. 
\end{remark}
\subsection{Semiorthogonal decompositions on formal fibers}\label{subsub:formal}
We have the following diagram: 
\begin{align}\label{moduli:M}
    \xymatrix{
    \mathscr{P}(d)^{\rm{cl}} \inclusion \ar[d]^{\pi_{P,d}} &
\mathscr{P}(d) \inclusion^{j} &
\mathscr{Y}(d) \ar[d]^{\pi_{Y,d}} \\
P(d) \ar@<-0.3ex>@{^{(}->}[rr] & & 
Y(d). 
    }
\end{align}
Here, the vertical arrows are good moduli space morphisms
and horizontal arrows are closed immersions. 
Consider a closed point $p\in P(d)$
which corresponds to a semisimple $(Q^{\circ, d}, \mathscr{I})$-representation 
\begin{align}\label{Rp}
 R_p=\bigoplus_{i=1}^m W^{(i)} \otimes R^{(i)}, 
\end{align}
where $R^{(i)}$ is a simple $(Q^{\circ, d}, \mathscr{I})$-representation 
of dimension $r^{(i)}$ and 
$W^{(i)}$ is a finite dimensional $\mathbb{C}$-vector space. 
We denote by $\widehat{\mathscr{Y}}(d)_p$
the formal fiber of the right vertical 
arrow in (\ref{moduli:M}) at $p$. 
By the \'{e}tale slice theorem, we have 
\begin{align*}
   \widehat{\mathscr{Y}}(d)_p=\widehat{\Ext}_{Q^{\circ, d}}^1(F_p, F_p)/G_p,
\end{align*}
where $G_p=\mathrm{Aut}(R_p)=\prod_{i=1}^m GL(W^{(i)})$, and see Subsection \ref{notation} for the notation. 
We denote by  
\begin{align}\label{mapjp}
    j_p \colon \widehat{\mathscr{P}}(d)_p
    \hookrightarrow \widehat{\mathscr{Y}}(d)_p
\end{align}
the natural inclusion of the derived zero locus of $\mu$ restricted to 
$\widehat{\mathscr{Y}}(d)_p$. 
\begin{remark}\label{rmk:ext-quiver}
    Let $\kappa$ be the morphism 
    \begin{align*}
        \kappa \colon \Ext^1_{(Q^{\circ, d}, \mathscr{I})}(R_p, R_p) \to 
        \Ext^2_{(Q^{\circ, d}, \mathscr{I})}(R_p, R_p)
    \end{align*}
    given by $x \mapsto [x, x]$. By the formality of polystable 
    objects in 
    CY2 category, see~\cite[Corollary 4.9]{DavPurity}, 
    the derived stack $\widehat{\mathscr{P}}(d)_p$
    is equivalent to the formal fiber of $\kappa^{-1}(0)/G_p$ 
    at $0 \in \kappa^{-1}(0)^{\mathrm{cl}}\ssslash G_p$. 
   \end{remark}
We define 
\begin{align}\label{qbps:that}
  \mathbb{T}_p(d)_w :=\mathbb{W}(\widehat{\mathscr{P}}(d)_p)_{w\tau_d}^{\rm{int}}
  \subset D^b(\widehat{\mathscr{P}}(d)_p). 
\end{align}
There is a description of $\mathbb{T}_p(d)_w$
similar to Lemma~\ref{lem:compareT}, see Subsection~\ref{subsec:prop}. 
Consider a partition $d=d_1+\cdots+d_k$.  
We have the commutative diagram:
\begin{equation}\label{com:hall2}
    \begin{tikzcd}
        \times_{i=1}^k
\mathscr{P}(d_i)&  \mathscr{P}(d_1, \ldots, d_k)\arrow[l, "q"']\arrow[r, "p"]& \mathscr{P}(d)\\
\times_{i=1}^k
\mathscr{P}(d_i)^{\mathrm{cl}}\arrow[u, hook]\arrow[d, "\times_{i=1}^k \pi_{P,d_i}"]&\mathscr{P}(d_1, \ldots, d_k)^{\mathrm{cl}}\arrow[l, "q"']\arrow[r, "p"] \arrow[u, hook] & \mathscr{P}(d)^{\mathrm{cl}}\arrow[u, hook]\arrow[d, "\pi_{P,d}"]\\
\times_{i=1}^k P(d_i)\arrow[rr, "\oplus"]& & P(d),
    \end{tikzcd}
\end{equation}
where $\times_{i=1}^k \pi_{P,d_i}$ and $\pi_{P,d}$ are good moduli space maps.
The base change of the categorical Hall product 
gives the functor 
\begin{align}\label{bchange}
    \bigoplus_{p_1+\cdots+p_k=p}
    \boxtimes_{i=1}^k D^b(\widehat{\mathscr{P}}(d_i)_{p_i})
    \to D^b(\widehat{\mathscr{P}}(d)_{p}),
\end{align}
where the sum on the left hand side consists of the fiber of the 
 bottom horizontal arrow $\oplus$ in (\ref{com:hall2}), 
 which is a finite 
 map. Also see~\cite[(6.12), Lemma~6.4]{Totheta} for the existence of 
 base change diagram of (\ref{com:hall2}) extended to derived stacks. 
 
 The following proposition is a formal fiber 
 version of Theorem~\ref{cor:sodT}. 
 The proof is technical 
 and will be postponed to Subsection~\ref{subsec:prop}. 
\begin{prop}\label{prop:sod2}
    There is a semiorthogonal decomposition 
    \begin{align*}
    D^b(\widehat{\mathscr{P}}(d)_p)
    =\left\langle \bigoplus_{p_1+\cdots+p_k=p}
    \boxtimes_{i=1}^k \mathbb{T}_{p_i}(d_i)_{w_i+(g-1)d_i(\sum_{i>j}d_j-\sum_{i<j}d_j)}
    \right\rangle. 
    \end{align*}
    The right hand side is after all partitions $(d_i)_{i=1}^k$ of $d$,
    all
    points $(p_1,\ldots, p_k)$ in the fiber over $p$ of the addition map $\oplus\colon \times_{i=1}^k P(d_i)\to P(d)$, and all weights $(w_i)_{i=1}^k\in \mathbb{Z}^k$ such that
    \[\frac{w_1}{d_1}<\cdots<\frac{w_k}{d_k}.\] 
    The order of the semiorthogonal 
    decomposition is the same as the order of (\ref{sod:triple}). 
The fully-faithful functor 
\begin{align*}
    \bigoplus_{p_1+\cdots+p_k=p}
    \boxtimes_{i=1}^k \mathbb{T}_{p_i}(d_i)_{w_i+(g-1)d_i(\sum_{i>j}d_j-\sum_{i<j}d_j)}
    \to  D^b(\widehat{\mathscr{P}}(d)_p)
\end{align*}
is given by the base change of the categorical Hall product (\ref{bchange}). 
\end{prop}

During the proof of Proposition~\ref{prop:sod2}, we will also obtain the following: 
\begin{cor}\label{cor:gen}
The map $\iota_p \colon \widehat{\mathscr{P}}(d)_p \to \mathscr{P}(d)$
induces the functor 
\begin{align*}
    \iota_p^{\ast} \colon 
    \mathbb{T}(d)_w \to \mathbb{T}_p(d)_w
\end{align*}
and its image classically generates 
$\mathbb{T}_p(d)_w$. 
\end{cor}

\subsection{Reduced quasi-BPS categories}
We continue the discussion from the previous subsection. 
Let 
$\mathfrak{gl}(d)_0 \subset \mathfrak{gl}(d)$ be the traceless Lie subalgebra, and 
let $\mu_0$ be the map 
\begin{align}\label{mu0:trace}
    \mu_0 \colon \mathfrak{gl}(d)^{\oplus 2g} \to \mathfrak{gl}(d)_0, \ 
    (x_1, \ldots, x_g, y_1, \ldots, y_g) \mapsto \sum_{i=1}^g [x_i, y_i]. 
    \end{align}
    Define the reduced stack: 
     \begin{align*}
         \mathscr{P}(d)^{\rm{red}} :=\mu_0^{-1}(0)/GL(d). 
     \end{align*}
          We define the reduced quasi-BPS category to be 
    \begin{align}\label{def:Tdwred}
      \mathbb{T}(d)_w^{\rm{red}}:=\mathbb{W}(\mathscr{P}(d)^{\rm{red}})^{\rm{int}}_{w\tau_d} \subset D^b(\mathscr{P}(d)^{\rm{red}}). 
      \end{align}
      There is a description similar to Lemma~\ref{lem:compareT}
      using the embedding 
$\mathscr{P}(d)^{\rm{red}} \hookrightarrow \mathscr{Y}(d)$. 
Denote by $\mathfrak{gl}(d)_{\rm{nil}} \subset \mathfrak{gl}(d)_0$ the subset of nilpotent 
        elements. 
 The categorical support lemma in~\cite{PTquiver} is the following: 
     \begin{lemma}\emph{(\cite[Corollary~5.5]{PTquiver})}\label{cor:support}
        For coprime $(d, w)\in\mathbb{N}\times\mathbb{Z}$, any object 
        $\mathcal{E} \in \mathbb{T}(d)_w^{\rm{red}}$ satisfies: 
        \begin{align*}
            \mathrm{Supp}^{\rm{sg}}(\mathcal{E})
            \subset (\mathfrak{gl}(d)^{\oplus 2g} \oplus \mathfrak{gl}(d)_{\rm{nil}})/GL(d). 
        \end{align*}
     \end{lemma}

For $g\geq 2$, the derived stack $\mathscr{P}(d)^{\rm{red}}$ is 
classical by~\cite[Proposition~3.6]{KaLeSo}, 
in particular there is a 
good moduli space 
morphism 
\begin{align*}
\pi_P \colon \mathscr{P}(d)^{\rm{red}}\to P(d)=\mu^{-1}(0)^{\mathrm{red}}\ssslash G(d).
\end{align*}
It follows that the Hom-space between any two objects in $D^b(\mathscr{P}(d)^{\rm{red}})$ is a module 
over $\mathcal{O}_{P(d)}$.      
     The categorical support lemma is the main ingredient in the proof of the following:
\begin{prop}\label{lem:bound}\emph{(\cite[Proposition~5.9]{PTquiver})}
For coprime $(d, w)\in\mathbb{N}\times\mathbb{Z}$ and objects $\mathcal{E}_i \in \mathbb{T}(d)_w^{\rm{red}}$
for $i=1, 2$, 
the $\mathcal{O}_{P(d)}$-module \[\bigoplus_{i\in \mathbb{Z}}\Hom^i_{\mathscr{P}(d)^{\mathrm{red}}}(\mathcal{E}_1, \mathcal{E}_2)\]
is finitely generated. In particular, 
we have $\Hom^i_{\mathscr{P}(d)^{\mathrm{red}}}(\mathcal{E}_1, \mathcal{E}_2)=0$ for $\lvert i \rvert \gg 0$.     
\end{prop}

\subsection{Relative Serre functor on reduced quasi-BPS categories}
We continue the discussion from the previous subsection. 
We have that $\mathbb{T}:=\mathbb{T}(d)^{\rm{red}}_w$ is a subcategory of $D^b(\mathscr{P}^{\mathrm{red}})$, which is a 
module over $\rm{Perf}(\mathscr{P}(d)^{\rm{red}})$.
Thus there is an associated internal homomorphism, 
see Subsection~\ref{subsec:qsmooth}:
\begin{align*}
 \mathcal{H}om_{\mathbb{T}}(\mathcal{E}_1, \mathcal{E}_2)
\in D_{\rm{qc}}(\mathscr{P}(d)^{\rm{red}})
\end{align*} for $\mathcal{E}_1, \mathcal{E}_2 \in \mathbb{T}$.
Proposition~\ref{lem:bound} implies that 
$\pi_{\ast} \mathcal{H}om_{\mathbb{T}}(\mathcal{E}_1, \mathcal{E}_2)$ is an object of $D^b(P(d))$. 

\begin{thm}\emph{(\cite[Theorem~5.10]{PTquiver})}\label{thm:Serre}
For coprime $(d, w)\in\mathbb{N}\times\mathbb{Z}$
and $\mathcal{E}_1, \mathcal{E}_2 \in \mathbb{T}$, 
there is an isomorphism: 
\begin{align}\label{isom:Serre}
    \Hom_{P(d)}(\pi_{\ast}\mathcal{H}om_{\mathbb{T}}(\mathcal{E}_1, \mathcal{E}_2), \mathcal{O}_{P(d)})
    \cong \Hom_{\mathbb{T}}(\mathcal{E}_2, \mathcal{E}_1). 
\end{align}
\end{thm}
For $\mathcal{E}_1=\mathcal{E}_2=\mathcal{E}$, 
the identity $\id \colon \mathcal{E} \to \mathcal{E}$
corresponds, under \eqref{isom:Serre}, to the morphism 
\begin{align*}
    \mathrm{tr}_{\mathcal{E}} \colon 
    \pi_{\ast}\mathcal{H}om(\mathcal{E}, \mathcal{E}) 
    \to \mathcal{O}_{P(d)}.
\end{align*}
From the construction in~\cite{PTquiver}, 
the above morphism coincides with the 
trace map determined by 
$(GL(d), \mathfrak{gl}(d)^{\oplus 2g}, \mathfrak{gl}(d)_0, \mu_0)$, 
see Subsection~\ref{subsec:trace} for the construction of the trace map, 
especially (\ref{const:tre}).

\section{Quasi-BPS categories for K3 surfaces}
In this section, we introduce (non-reduced and reduced)
quasi-BPS categories for K3 surfaces. In Theorem \ref{thm:walleq}, we prove the wall-crossing equivalence for quasi-BPS categories. We state a categorical version of the $\chi$-independence phenomenon, see Conjecture \ref{conj:HK}, which we prove for $g=0$ and for $g=1$ and $(d,w)=(2,1)$.

\subsection{Generalities on K3 surfaces}\label{subsec41}
Let $S$ be a smooth projective K3 surface, i.e. $K_S$ is trivial 
and $H^1(\mathcal{O}_S)=0$. 
Let $K(S)$ be the Grothendieck group of $S$.
Denote by $\chi(-, -)$ the Euler pairing 
\begin{align*}
\chi(E, F)=\sum_{j}(-1)^j \mathrm{ext}^j(E, F).
\end{align*} 
Let $N(S)$ be the numerical 
Grothendieck group:
\begin{align*}
 N(S):=K(S)/\equiv,
\end{align*}
where $E_1 \equiv E_2$ in $K(S)$ if 
$\chi(E_1, F)=\chi(E_2, F)$ for any $F \in K(S)$.
There is an isomorphism 
by taking the Mukai vector: 
\begin{align}\label{isom:Mvector}
    v(-)=\ch(-)\sqrt{\mathrm{td}}_S \colon N(S) \stackrel{\cong}{\to} 
    \mathbb{Z} \oplus \mathrm{NS}(S) \oplus \mathbb{Z}. 
\end{align}
Write a vector $v\in N(S)$ as $v=(r, \beta, \chi)\in \mathbb{Z}\oplus \mathrm{NS}(S)\oplus \mathbb{Z}$ via the above isomorphism.
There is a symmetric bilinear pairing on $N(S)$
defined by $\langle E_1, E_2 \rangle=-\chi(E_1, E_2)$. 
Under the isomorphism (\ref{isom:Mvector}), we have 
\begin{align*}
    \langle E_1, E_2 \rangle=\beta_1 \beta_2-r_1 \chi_2-r_2 \chi_1,
\end{align*}
where $v(E_i)=(r_i, \beta_i, \chi_i)$. 

We say $v\in N(S)$ is \textit{primitive} if it cannot be written as $v=dv_0$ for an integer $d\geq 2$ and $v_0\in N(S)$.
Let $v\in N(S)$ and $w\in \mathbb{Z}$. Write $v=dv_0$ for $d\in \mathbb{Z}$ and $v_0$ primitive. We define $\gcd(v, w):=\gcd(d,w)$.
Below we identify $N(S)$ with $\mathbb{Z} \oplus \mathrm{NS}(S) \oplus \mathbb{Z}$ via 
the isomorphism (\ref{isom:Mvector}), and write an element $v \in N(S)$ 
as $v=(r, \beta, \chi)$. 

\subsection{Bridgeland stability conditions on K3 surfaces}
For a K3 surface $S$, 
we denote by 
\begin{align*}
\mathrm{Stab}(S)
\end{align*}
the (main connected component of)
the space of Bridgeland stability 
conditions~\cite{Brs1, Brs2} on $D^b(S)$. 
A point $\sigma \in \mathrm{Stab}(S)$ consists of 
a pair 
\begin{align*}
    \sigma=(Z, \mathcal{A}), \ 
    Z \colon N(S) \to \mathbb{C}, \ 
    \mathcal{A} \subset D^b(S),
\end{align*}
where
$Z$ is a group homomorphism (called \textit{central charge}) and 
$\mathcal{A}$ is the heart of a bounded t-structure 
satisfying some axioms, see~\cite{Brs1}.
One of the axioms is the following positivity 
property 
\begin{align*}
    Z(E) \in \{ z \in \mathbb{C} : 
    \mathrm{Im}(z)>0 \mbox{ or } z \in \mathbb{R}_{<0}\}
\end{align*}
for any $0\neq E \in \mathcal{A}$. 
An object $E \in \mathcal{A}$ is called \textit{$Z$-(semi)stable}
if for any subobject $0\neq F \subsetneq E$
we have $\arg Z(F)<(\leq) \arg Z(E)$ in $(0, \pi]$. 
An object $E \in D^b(X)$ is called \textit{$\sigma$-(semi)stable}
if $E[a] \in \mathcal{A}$ is $Z$-semistable for some $a \in \mathbb{Z}$. 

For each $B+iH \in \mathrm{NS}(S)_{\mathbb{C}}$ such 
that $H$ is ample with $H^2>2$, 
there is an associated stability condition 
\begin{align*}
\sigma_{B, H}=(Z_{B, H}, \mathcal{A}_{B, H}) \in \mathrm{Stab}(S),  
\end{align*}
where $\mathcal{A}_{B, H} \subset D^b(S)$ is 
the heart of a bounded t-structure obtained 
by a tilting of $\mathrm{Coh}(S)$ and 
$Z_{B, H}$ is given by 
\begin{align*}
    Z_{B, H}(E)=-\int_S e^{-B-iH} v(E) \in \mathbb{C}. 
\end{align*}
We refer to~\cite[Section~6]{Brs2} for the construction 
of the above stability conditions. 
A stability condition $\sigma_{B, mH}$
for $m\gg 0$ is said 
to be in a \textit{neighborhood of the large volume limit}. 
Recall the following proposition about semistable
objects at the large volume limit:
\begin{prop}\label{prop:LV}
\emph{(\cite[Proposition~14.2]{Brs3}, \cite[Proposition~6.4, Lemma~6.5]{Tst3})}
If $v=(r, \beta, \chi)$ such that $r\geq 0$ and $H \cdot \beta>0$, 
or $r=H \cdot \beta=0$ and $\chi>0$, 
then an object $E\in D^b(S)$ of Mukai vector $v$ is $\sigma_{0, mH}$-semistable 
for $m\gg 0$ 
if and only if $E[2a]$ is an $H$-Gieseker semistable 
sheaf for some $a\in\mathbb{Z}$.
\end{prop}

\subsection{Moduli stacks of semistable objects on K3 surfaces}\label{subsection:moduliK3}
For each $\sigma \in \mathrm{Stab}(S)$ and $v \in N(S)$, 
we denote by 
\begin{align*}
    \mathfrak{M}_S^{\sigma}(v)
    \end{align*}
the derived moduli stack of $\sigma$-semistable 
objects $E \in \mathcal{A} \cup \mathcal{A}[1]$ with numerical class $v$.
We denote by 
$\mathbb{F}$ the universal object
\begin{align*}
    \mathbb{F} \in D^b(S \times \mathfrak{M}_S^{\sigma}(v)). 
\end{align*}
We also consider the reduced version 
of the stack $\mathfrak{M}_S^{\sigma}(v)$. 
Let $v=(r, \beta, \chi)$. 
Let $\mathcal{P}ic^{\beta}(S)$ be the derived moduli 
stack of line bundles on $S$ with first Chern class 
$\beta$. Then  $\mathcal{P}ic^{\beta}(S)=\Spec\mathbb{C}[\varepsilon]/\mathbb{C}^{\ast}$, where $\varepsilon$ is of degree $-1$.
We consider the 
 determinant morphism 
\begin{align*}
	\det \colon \mathfrak{M}_S^{\sigma}(v) \to \mathcal{P}ic^{\beta}(S)=\Spec\mathbb{C}[\varepsilon]/\mathbb{C}^{\ast}.
	\end{align*}
Define the reduced stack:
\begin{align}\label{defn:reduced}
	\mathfrak{M}_S^{\sigma}(v)^{\rm{red}}:=
	\mathfrak{M}_S^{\sigma}(v) \times_{\mathcal{P}ic^{\beta}(S)} B\mathbb{C}^{\ast}. 
	\end{align}
 The obstruction 
 space of the 
 reduced stack $\mathfrak{M}_S^{\sigma}(v)^{\rm{red}}$
 at $F$ 
is the kernel of the 
trace map: 
\begin{align*}
    \Ext_S^2(F, F)_0:=\Ker\left(\Ext_S^2(F, F) \stackrel{\mathrm{tr}}{\twoheadrightarrow}
    H^2(\mathcal{O}_S)=\mathbb{C} \right). 
\end{align*}
Note that $\mathfrak{M}^\sigma_S(v)^{\mathrm{red}}$ may still not be a classical stack.
 There are decompositions: 
\begin{align*}
    D^b(\mathfrak{M}_S^{\sigma}(v))=\bigoplus_{w \in \mathbb{Z}}D^b(\mathfrak{M}_S^{\sigma}(v))_w, \ 
     D^b(\mathfrak{M}_S^{\sigma}(v)^{\rm{red}})=\bigoplus_{w \in \mathbb{Z}}D^b(\mathfrak{M}_S^{\sigma}(v)^{\rm{red}})_w,
\end{align*}
where each summand contains complexes $F$ of weight 
$w$ with respect to the 
scaling automorphisms $\mathbb{C}^{\ast} \subset \mathrm{Aut}(F)$, see~\cite[Subsection~3.2.4]{T}.

We denote by $\mathcal{M}_S^{\sigma}(v)$ the
classical truncation of $\mathfrak{M}_S^{\sigma}(v)$. It admits a good moduli space
(cf.~\cite{MR3237451}, \cite[Example 7.26]{AHLH}):
\begin{align}\label{gmoduli}
\pi \colon 
    \mathcal{M}_S^{\sigma}(v) \to M_S^{\sigma}(v),
\end{align}
where $M_S^{\sigma}(v)$ is a proper algebraic space. 
A closed point $y \in M_S^{\sigma}(v)$ corresponds 
to a $\sigma$-polystable object
\begin{align}\label{pstable:S}
    F=\bigoplus_{i=1}^m V^{(i)} \otimes F^{(i)},
\end{align}
where $F^{(1)}, \ldots, F^{(m)}$ 
are mutually non-isomorphic 
$\sigma$-stable objects such that $\arg Z(F^{(i)})=\arg Z(F)$, 
and $V^{(i)}$ is a finite dimensional vector space with dimension $d^{(i)}$ for $1\leq i\leq m$.

Let 
$G_y:=\mathrm{Aut}(F)=\prod_{i=1}^m GL(V^{(i)})$
and let
$\widehat{\Ext}_S^1(F, F)$ be the
formal fiber at the origin of the morphism 
\begin{align*}
    \Ext_S^1(F, F) \to \Ext_S^1(F, F)\ssslash G_y. 
\end{align*}
By the formality of the dg-algebra 
$\mathrm{RHom}(F, F)$, see~\cite[Corollary 4.9]{DavPurity}, 
there are equivalences 
\begin{align}\label{formal:equiv}
    \widehat{\mathfrak{M}}_S^{\sigma}(v)_y
    \simeq \widehat{\kappa}^{-1}(0)/G_y, \ 
     \widehat{\mathfrak{M}}_S^{\sigma}(v)_y^{\rm{red}}
    \simeq \widehat{\kappa}_0^{-1}(0)/G_y,
\end{align}
where 
$\kappa$, $\kappa_0$ are the maps
\begin{align*}
    \kappa\colon \Ext_S^1(F, F) \to \Ext_S^2(F, F), \
    \kappa_0 \colon \Ext_S^1(F, F) \to \Ext_S^2(F, F)_0
\end{align*}
given by $x \mapsto [x, x]$, and $\widehat{\kappa}$, 
$\widehat{\kappa}_0$ are their restrictions to 
$\widehat{\Ext}_S^1(F, F)$. 
\begin{remark}\label{rmk:Equiver}
The stack $\kappa^{-1}(0)/G_y$ 
is described in terms of the Ext-quiver of $F$ as follows. 
    Let $Q^{\circ, d}_y$ be the quiver with vertices $\{1, \ldots, m\}$
    and the number of edges from $i$ to $j$ is 
    $\dim \Ext_{S}^1(F^{(i)}, F^{(j)})$ for any $1\leq i, j\leq m$.  
    By Serre duality, $Q^{\circ, d}_y$ is symmetric. Moreover
    the number of loops at each vertex is even, 
    so $Q^{\circ, d}_y$ is the doubled quiver of some quiver $Q_y^{\circ}$. 
    The derived stack $\kappa^{-1}(0)/G_y$ is identified with the 
    derived moduli stack of 
    representations of the preprojective algebra of $Q^\circ_y$ (alternatively, of
    $Q^{\circ, d}_y$-representations with quadratic
    relation $\mathscr{I}_y$)
    as in Subsection~\ref{subsec:double},
    and dimension vector $(d^{(i)})_{i=1}^m$ where $d^{(i)}=\dim V^{(i)}$. 
\end{remark}

There is a wall-chamber structure on $\mathrm{Stab}(S)$ 
 such that $\mathcal{M}_S^{\sigma}(v)$ is constant 
 if $\sigma$ lies in a chamber, but may change when 
 $\sigma$ crosses a wall. 
 Locally, a wall is defined by the equation 
 \begin{align*}
     \frac{Z(v_1)}{Z(v_2)} \in \mathbb{R}_{>0}, \ 
v=v_1+v_2
\end{align*}
such that $v_1$ and $v_2$ are not 
 proportional, see~\cite[Proposition~9.3]{Brs2}. 
 
 A stability condition $\sigma\in \mathrm{Stab}(S)$ is \textit{generic} 
 if $\sigma$ is not on a wall. 
If $\sigma$ is generic, then 
for a polystable object (\ref{pstable:S})
each numerical class of $F^{(i)}$ is proportional to 
$v$. 
Let $v=dv_0$ for a primitive $v_0$. Then 
we have 
\begin{align*}[F^{(i)}]=r^{(i)}v_0, \ d^{(1)}r^{(1)}+\cdots+d^{(m)}r^{(m)}=d.
\end{align*}
The good moduli space $M_S^{\sigma}(v)$ has a 
stratification indexed by data $(d^{(i)}, r^{(i)})_{i=1}^m$, 
and the deepest stratum
corresponds to $m=1$, $\dim V^{(1)}=d$ and 
$d^{(1)}=1$. 

Let $v=dv_0$ for a primitive $v_0$ with 
$\langle v_0, v_0\rangle=2g-2$, and 
$Q^{\circ, d}=Q_{2g}$ be the quiver with one vertex 
and $2g$-loops with relation $\mathscr{I}$ as in Subsection~\ref{subsec:double}. 
Recall the stacks:
\begin{align*}
    \mathscr{P}(d)=\mu^{-1}(0)/GL(d), \ 
    \mathscr{P}(d)^{\rm{red}}=\mu_0^{-1}(0)/GL(d)
\end{align*}
where $\mu \colon \mathfrak{gl}(d)^{\oplus 2g} \to 
\mathfrak{gl}(d)$ and $\mu_0 \colon \mathfrak{gl}(d)^{\oplus 2g} \to \mathfrak{gl}(d)_0$ are moment maps. 
\begin{lemma}\label{lem:py}
    For any closed point $y \in M_S^{\sigma}(v)$, 
    there is a point $p \in P(d)$
    which is sufficiently close to $0\in P(d)$ such that 
    we have equivalences
    \begin{align}\label{formal:equiv3}
    \widehat{\mathfrak{M}}_S^{\sigma}(v)_y
    \simeq \widehat{\mathscr{P}}(d)_p, \ \widehat{\mathfrak{M}}_S^{\sigma}(v)_y^{\rm{red}}
    \simeq 
    \widehat{\mathscr{P}}(d)_p^{\rm{red}}.
\end{align}
If $y$ lies in the deepest stratum, we can take 
$p=0$. 
\end{lemma}
\begin{proof}
    Let $y$ corresponds to a direct sum (\ref{pstable:S})
    such that $[F^{(i)}]=r^{(i)} v_0$, and 
    let $R^{(i)}$ be a simple $(Q^{\circ, d}, \mathscr{I})$-representation 
    with dimension $r^{(i)}$. 
    Such $R^{(i)}$ exists by a straightforward dimension count argument, for example see
    the proof of~\cite[Lemma~5.7 (i)]{PTquiver}. 
    Let 
    \begin{align*}
R=\bigoplus_{i=1}^m V^{(i)} \otimes R^{(i)}
\end{align*}
be a semisimple $(Q^{\circ, d}, \mathscr{I})$-representation
    and let $p\in P(d)$ be the corresponding point. 
    Note that 
    \begin{align*}
    \chi(R^{(i)}, R^{(j)})=\chi(F^{(i)}, F^{(j)})=r^{(i)}r^{(j)}(2-2g).
    \end{align*}
    By the CY2 property of $(Q^{\circ, d}, \mathscr{I})$-representations (cf. see \cite{KellerVandenBergh}, \cite[Proposition 7.1]{DavPurity}), 
    and the fact that $\hom(R^{(i)}, R^{(j)})=\hom(F^{(i)}, F^{(j)})=\delta_{ij}$, 
    we have an isomorphism 
    \begin{align*}
    \Ext^{\ast}(R^{(i)}, R^{(j)})\cong \Ext^{\ast}(F^{(i)}, F^{(j)}).
    \end{align*}
    Then by the formality of 
    polystable objects 
    CY2 categories~\cite[Corollary~4.9]{DavPurity}, 
    there is an isomorphism of dg-algebras 
    $\mathrm{RHom}(R, R) \cong \mathrm{RHom}(F, F)$. 
    Therefore we have equivalences (\ref{formal:equiv3}), see Remark~\ref{rmk:ext-quiver}. 
    
    There is an action of $\mathbb{C}^{\ast}$ on the moduli of
    $Q^{\circ, d}$-representations which scales the linear maps corresponding to each edge of $Q^{\circ, d}$, which induces an 
    action on $P(d)$. 
    The above $\mathbb{C}^{\ast}$-action preserves
    the type of the semi-simplification, 
    and any point $p \in P(d)$ 
    satisfies $\lim_{t\to 0} (t \cdot p)=0$.
    Therefore we can take $p$ to be sufficiently close 
    to $0$. By the above construction, we can take $p=0$
    if $y$ lies in the deepest stratum. 
\end{proof}
Combining Lemma \ref{lem:py} with~\cite{KaLeSo}, we have the following: 
\begin{lemma}\label{lem:classical}
Suppose that $g\geq 2$. Then for a generic $\sigma$, 
the derived stack $\mathfrak{M}_S^{\sigma}(v)^{\rm{red}}$ is classical, i.e. 
the natural morphism 
$\mathcal{M}_S^{\sigma}(v) \to \mathfrak{M}_S^{\sigma}(v)^{\rm{red}}$ is an equivalence.     
\end{lemma}
\begin{proof}
For $g\geq 2$, the derived stack $\mathscr{P}(d)^{\rm{red}}$ is classical
by~\cite[Proposition~3.6]{KaLeSo}.
Therefore the conclusion holds by Lemma~\ref{lem:py}. 
\end{proof}

\subsection{Quasi-BPS categories for K3 surfaces}
Let $v\in N(S)$ and $w\in\mathbb{Z}$.
Take $a \in K(S)_{\mathbb{R}}$ such that 
$\chi(a\otimes v)=w \in \mathbb{Z}$.
We define the $\mathbb{R}$-line bundle $\delta$ 
on 
$\mathfrak{M}_S^{\sigma}(v)$ to be 
\begin{align}\label{def:delta}
   \delta=\det p_{\mathfrak{M}_{\ast}}(a \boxtimes 
   \mathbb{F}), 
\end{align}
where $p_{\mathfrak{M}} \colon S \times 
\mathfrak{M}_S^{\sigma}(v) \to \mathfrak{M}_S^{\sigma}(v)$ is the projection. 
Note that the object $p_{\mathfrak{M}\ast}(A \boxtimes \mathbb{F})$ is a 
perfect complex on $\mathfrak{M}_S^{\sigma}(v)$ for any $A \in D^b(S)$, so 
the $\mathbb{R}$-line bundle (\ref{def:delta}) is well-defined. 
The pull-back of $\delta$ to $\mathfrak{M}_S^{\sigma}(v)^{\rm{red}}$ is 
also denoted by $\delta$. 
We 
define the (non-reduced or reduced) 
quasi-BPS categories to be the following intrinsic window categories from Definition \ref{def:intwind}:
\begin{align}\label{qbps:T}	&\mathbb{T}_S^{\sigma}(v)_ \delta:=\mathbb{W}(\mathfrak{M}_S^{\sigma}(v))^{\rm{int}}_{\delta}\subset 	D^b(\mathfrak{M}_S^{\sigma}(v))_w, \\ \notag
	&\mathbb{T}_S^{\sigma}(v)_ \delta^{\rm{red}}:=\mathbb{W}(\mathfrak{M}_S^{\sigma}(v)^{\rm{red}})^{\rm{int}}_{\delta}\subset 
 D^b(\mathfrak{M}_S^{\sigma}(v)^{\rm{red}})_w. 
	\end{align}

\begin{remark}\label{rmk:qbps}
For each $y\in M_S^{\sigma}(v)$, 
 let $\widehat{\mathfrak{M}}_S^{\sigma}(v)_y$ be the formal 
 fiber at $y$ and $\delta_y$ the pull-back of 
 $\delta$ to it. 
 The quasi-BPS category for the formal 
 fiber is defined in a similar way: 
 \begin{align*}\mathbb{T}^{\sigma}_{S, y}(v)_{\delta_y}
 :=\mathbb{W}(\widehat{\mathfrak{M}}_S^{\sigma}(v)_y)_{\delta_y}^{\rm{int}}
 \subset D^b(\widehat{\mathfrak{M}}_S^{\sigma}(v)_y). 
 \end{align*}
 By the definition of $\mathbb{T}_S^{\sigma}(v)_{\delta}$, an object 
 $\mathcal{E} \in D^b(\mathfrak{M}_S^{\sigma}(v))$ is an 
 object in $\mathbb{T}_S^{\sigma}(v)_{\delta}$ 
 if and only if 
 its restriction to any formal fiber is 
 an object in $\mathbb{T}_{S, y}^{\sigma}(v)_{\delta_y}$. 
 There is an analogous statement for the reduced version. 
 \end{remark}
\begin{lemma}\label{lem:sigmagen}
If $\sigma\in \mathrm{Stab}(S)$ is generic, then $\mathbb{T}_S^{\sigma}(v)_{\delta}$ 
and $\mathbb{T}_S^{\sigma}(v)_{\delta}^{\rm{red}}$ are independent of $a\in K(S)_{\mathbb{R}}$ satisfying 
$\chi(a \otimes v)=w$. 
\end{lemma}
\begin{proof}
Let $b \in K(S)_{\mathbb{R}}$ such that 
$\chi(b \otimes v)=0$ and set 
$a'=a+b$. Let $\delta'$ be the $\mathbb{R}$-line 
bundle defined as in (\ref{def:delta}) for $a'$. 
    By Remark~\ref{rmk:qbps}, 
    it is enough to show that $\delta_y=\delta'_y$ for 
    any closed point $y \in M_S^{\sigma}(v)$. 
    Let $y$ be a point which corresponds to the polystable object $F$ as in (\ref{pstable:S}). 
    By the decomposition (\ref{pstable:S}), 
    we have $\mathrm{Aut}(F)=\prod_{i=1}^m GL(V^{(i)})$
    and $\delta_y$ is the character of $\mathrm{Aut}(F)$ 
    given by 
    \begin{align}\label{deltay}
        \delta_y=
        \det\left(\sum_{i=1}^m V^{(i)} \otimes \chi(a \otimes F^{(i)})\right)
        =\bigotimes_{i=1}^m (\det V^{(i)})^{\chi(a\otimes F^{(i)})}. 
    \end{align}
    As $\sigma$ is generic, 
    the numerical class of $F^{(i)}$ is proportional to 
    $v$. Therefore $\chi(b \otimes v)=0$ implies 
    $\chi(b \otimes F^{(i)})=0$ for $1\leq i\leq m$, hence 
    $\delta_y=\delta'_y$. 
\end{proof}

By the above lemma, the following definition 
makes sense. 
\begin{defn}\label{def:qbps0}
For $v \in N(S)$, 
let $\sigma \in \mathrm{Stab}(S)$ be generic. 
For $w \in \mathbb{Z}$, define 
\begin{align*}
    \mathbb{T}_S^{\sigma}(v)_{w} :=
    \mathbb{T}_S^{\sigma}(v)_\delta, \ 
   \mathbb{T}_S^{\sigma}(v)_{w}^{\rm{red}} :=
    \mathbb{T}_S^{\sigma}(v)_\delta^{\rm{red}}. 
\end{align*}
Here $\delta$ is defined as in \eqref{def:delta} for any $a \in K(S)_{\mathbb{R}}$
such that $\chi(a \otimes v)=w$. 
\end{defn}

The first main result of this section is the following wall-crossing equivalence of 
quasi-BPS categories. 

\begin{thm}\label{thm:walleq}
Let $\sigma_1, \sigma_2 \in \mathrm{Stab}(S)$
be generic stability conditions. 
Then there exist equivalences
\begin{align}\label{equiv:T}
\mathbb{T}_S^{\sigma_1}(v)_{w} \stackrel{\sim}{\to} 
\mathbb{T}_S^{\sigma_2}(v)_{w}, \ 
\mathbb{T}_S^{\sigma_1}(v)_{w}^{\rm{red}}\stackrel{\sim}{\to} 
\mathbb{T}_S^{\sigma_2}(v)_{w}^{\rm{red}}. 
\end{align}
    \end{thm}
    \begin{proof}
    We only prove the first equivalence, the second one follows by the same argument.
We reduce the proof of the equivalence to a local statement as in Theorem~\ref{prop:eqS}.
    
Consider a stability $\sigma=(Z, \mathcal{A}) \in \mathrm{Stab}(S)$ lying on a wall 
and consider stability conditions $\sigma_{\pm}=(Z_{\pm}, \mathcal{A}_{\pm}) \in \mathrm{Stab}(S)$
lying in adjacent chambers. 
Let $b \in K(S)_{\mathbb{R}}$ be an  
element satisfying 
$\chi(b \otimes v)=0$ 
and let $\delta' \in \mathrm{Pic}(\mathfrak{M}_S^{\sigma}(v))_{\mathbb{R}}$
be defined as in (\ref{def:delta}) using $b$. Let 
$\delta \in \mathrm{Pic}(\mathfrak{M}_S^{\sigma}(v))_{\mathbb{R}}$
be 
as in Definition \ref{def:qbps0}, 
and set 
$\delta''=\delta+\delta'$.  
It is enough to show that, 
there exists $b$ as above such that 
the restriction 
functors for 
the open immersions $
\mathfrak{M}^{\sigma_{\pm}}_S(v) \subset 
\mathfrak{M}^{\sigma}_S(v)$
restrict to the equivalence 
\begin{align}\label{induce:T}
    \mathbb{T}_S^{\sigma}(v)_{\delta''} \stackrel{\sim}{\to} 
    \mathbb{T}_S^{\sigma_{\pm}}(v)_{\delta''}.
    \end{align}

The open substacks $\mathfrak{M}_S^{\sigma_{\pm}}(v) \subset \mathfrak{M}_S^{\sigma}(v)$ 
are semistable loci with respect to line bundle $\ell_{\pm}$ on 
$\mathfrak{M}_S^{\sigma}(v)$ and they are parts of  
$\Theta$-stratifications, see~\cite[Proposition~4.4.5]{halpK32}. 
The line bundles 
$\ell_{\pm}$ are constructed as follows. 
We may assume that $Z(v)=Z_{\pm}(v)=\sqrt{-1}$, 
and write 
$Z_{\pm}(-)=\chi(\omega_{\pm}\otimes -)$ for 
$\omega_{\pm} \in K(S)_{\mathbb{C}}$. 
Then set $b_{\pm} \in K(S)_{\mathbb{R}}$ 
to be the real parts of $\omega_{\pm}$, 
which satisfy $\chi(b_{\pm} \otimes v)=0$. 
The line bundles $\ell_{\pm}$ are 
defined by 
$\ell_{\pm}=\det p_{\mathfrak{M}\ast}(b_{\pm} \boxtimes \mathbb{F})$, 
see~\cite[Theorem~6.4.11]{Halpinstab}. 
Then we set $b=\varepsilon_{+} b_{+} +\varepsilon_{-} b_{-}$
for general elements
$0< \varepsilon_{\pm} \ll 1$. 

Since $\mathfrak{M}_S^{\sigma}(v)$ is 0-shifted 
symplectic, from Theorem~\ref{thm:window:M} and Remark~\ref{rmk:0shift} (see also~\cite[Theorem~3.3.1]{halpK32}), 
there exist subcategories
$\mathbb{W}(\mathfrak{M}_S^{\sigma}(v))^{\ell_{\pm}}_{m_{\bullet\pm}} 
    \subset D^b(\mathfrak{M}_S^{\sigma}(v))$
    which induce equivalences: 
\begin{align}\label{window:k3}
    \mathbb{W}(\mathfrak{M}_S^{\sigma}(v))^{\ell_{\pm}}_{m_{\bullet\pm}} 
    \stackrel{\sim}{\to}
    D^b(\mathfrak{M}_S^{\sigma_{\pm}}(v)).
\end{align}
Moreover, there exist choices of $m_{\bullet\pm}$ such that 
$\mathbb{T}_S^{\sigma}(v)_{\delta''} \subset 
\mathbb{W}(\mathfrak{M}_S^{\sigma}(v))^{\ell_{\pm}}_{m_{\bullet\pm}}$, 
see~\cite[Lemma~4.3.10]{halpK32} or~\cite[Proposition~6.15]{Totheta}
for a choice of $m_{\bullet}$. 
Therefore, by Remark~\ref{rmk:qbps}, it is enough to show 
that, for each closed point $y\in M_S^{\sigma}(v)$,
we have the equivalences 
\begin{align}\label{equiv:local}
    \mathbb{T}_{S, y}^{\sigma}(v)_{\delta_y''} \stackrel{\sim}{\to} 
    \mathbb{T}_{S, y}^{\sigma_{\pm}}(v)_{\delta_y''}.
\end{align}
Here, on the right hand side we consider the intrinsic window 
subcategories for the formal fibers 
of 
the morphisms 
$\mathcal{M}_S^{\sigma_{\pm}}(v) \subset \mathcal{M}_S^{\sigma}(v) \to M_S^{\sigma}(v)$ at $y$. 

Let $y$ corresponds to the polystable object (\ref{pstable:S}) 
and set $\bm{d}=(\dim V^{(i)})_{i=1}^m$. 
Let $(Q^{\circ, d}_y, \mathscr{I}_y)$ be the Ext-quiver at $y$ with relation $\mathscr{I}_y$, 
see~Remark~\ref{rmk:Equiver}. 
The quiver with relation $(Q^{\circ, d}_y, \mathscr{I}_y)$
is the
double of some quiver $Q_y^{\circ}$. 
Let $\mathscr{P}_y(\bm{d})$ 
be the derived stack of $(Q^{\circ, d}_y, \mathscr{I}_y)$-representations with dimension vector $\bm{d}$, 
see (\ref{P:dzero}), 
and $P_y(\bm{d})$ the good 
moduli space of its classical truncation. 
By the equivalence (\ref{formal:equiv}), 
there is an equivalence 
\begin{align}\label{equiv:formal}
\widehat{\mathfrak{M}}_S^{\sigma}(v)_y \simeq 
\widehat{\mathscr{P}}_y(\bm{d}).
\end{align}
Here the right hand side is the formal fiber of 
$\mathscr{P}(\bm{d})$ at $0 \in P(\bm{d})$. 
The line bundles 
$\ell_{\pm}$
restricted to $\widehat{\mathfrak{M}}_S^{\sigma}(v)$
correspond to 
generic elements $\ell_{\pm} \in M(\bm{d})_{\mathbb{R}}^W$, 
where $M(\bm{d})_{\mathbb{R}}$ is the 
character lattice of the maximal torus 
of $G_y:=\mathrm{Aut}(y)=\prod_{i=1}^m GL(V^{(i)})$. 
Moreover the 
$\sigma_{\pm}$-semistable loci 
in the left hand side of (\ref{equiv:formal})
correspond to $\ell_{\pm}$-semistable $(Q^{\circ, d}_y, \mathscr{I}_y)$-representations. 
Therefore the equivalences (\ref{equiv:local}) follow from 
the formal fiber version of 
the equivalences
\[\mathbb{T}(\bm{d})_{\delta_y''} \stackrel{\sim}{\to} 
\mathbb{T}^{l_{\pm}}(\bm{d})_{\delta_y''}\]
in Theorem~\ref{prop:eqS}, 
whose proof is identical to loc. cit. 
\end{proof}

By Lemma~\ref{lem:sigmagen} and Theorem~\ref{thm:walleq}, 
the following definition makes sense: 
\begin{defn}\label{defn:BPS}
For $v \in N(S)$ and $w \in \mathbb{Z}$, define 
\begin{align*}
    \mathbb{T}_S(v)_{w}:=
    \mathbb{T}_S^{\sigma}(v)_{w}, \ 
     \mathbb{T}_S(v)_{w}^{\rm{red}}:=
    \mathbb{T}_S^{\sigma}(v)_{w}^{\rm{red}},
\end{align*}
where $\sigma \in \mathrm{Stab}(S)$ is a generic stability condition. 
\end{defn}
\begin{remark}
    The category $\mathbb{T}_S(v)_w$ is defined 
    as an abstract pre-triangulated dg-category. 
    If we take a generic $\sigma \in \mathrm{Stab}(S)$, 
    it is realized as a subcategory 
    of $D^b(\mathfrak{M}_S^{\sigma}(v))$ 
    by the identification $\mathbb{T}_S(v)_w=\mathbb{T}_S^{\sigma}(v)_w
    \subset D^b(\mathfrak{M}_S^{\sigma}(v))$.
    \end{remark}
\begin{remark}\label{rmk:primitive}
Suppose that $g\geq 2$ and take a generic $\sigma \in \mathrm{Stab}(S)$. 
Then we have $\mathfrak{M}_S^{\sigma}(v)^{\rm{red}}=\mathcal{M}_S^{\sigma}(v)$. 
Let $\mathcal{M}_S^{\sigma\text{-st}}(v) \subset \mathcal{M}_S^{\sigma}(v)$
be the open substack of $\sigma$-stable objects. Then the good moduli 
space morphism
$\mathcal{M}_S^{\sigma\text{-st}}(v) \to M_S^{\sigma\text{-st}}(v)$
is a $\mathbb{C}^{\ast}$-gerbe classified by $\alpha \in \mathrm{Br}(M_S^{\sigma\text{-st}}(v))$
which gives the 
obstruction of the existence of a universal object in $S \times M_S^{\sigma\text{-st}}(v)$. 
We then have that
\begin{align}\label{decom:prim0}
        \mathbb{T}_S(v)_w^{\rm{red}}|_{M_S^{\sigma\text{-st}}(v)}=D^b(M_S^{\sigma}(v), \alpha^w),
\end{align}
where the right hand is the derived category of $\alpha^w$-twisted 
coherent sheaves on $M_S^{\sigma\text{-st}}(v)$, see~\cite{MR2700538, MR2309155}, and the left hand side is the subcategory of $D^b(\mathcal{M}^{\sigma\text{-st}}(v))$ classically generated by the restriction of objects in $\mathbb{T}_S(v)_w^{\rm{red}}$. 

If $v$ is primitive, then 
$M_S^{\sigma\text{-st}}(v)=M_S^{\sigma}(v)$ and 
it is a non-singular holomorphic symplectic variety 
deformation equivalent to the Hilbert scheme of points $S^{[n]}$, where $n=\langle v, v\rangle/2+1$. 
By (\ref{decom:prim0}), we have 
\begin{align}\label{decom:prim}
        \mathbb{T}_S(v)_w^{\rm{red}}=D^b(M_S^{\sigma}(v), \alpha^w).
\end{align}
\end{remark}

\begin{remark}\label{rmk:abelian}
We can also define quasi-BPS categories for other Calabi-Yau surfaces, 
i.e. abelian surfaces, similarly to Definition~\ref{defn:BPS}. 
When $S$ is an abelian surface, 
the derived Picard stack is $\mathcal{P}ic^{\beta}(S)=\widehat{S} \times \Spec \mathbb{C}[\varepsilon]/\mathbb{C}^{\ast}$, 
where $\widehat{S}$ is the dual abelian surface, and 
we define the reduced stack (\ref{defn:reduced}) 
by $\mathfrak{M}_S^{\sigma}(v) \times_{\mathcal{P}ic^{\beta}(S)}\mathcal{P}ic^{\beta}(S)^{\rm{cl}}$. 
The results in this paper also hold for abelian surfaces.     
\end{remark}

Let $v=dv_0$ for a primitive $v_0$. 
We expect that, if $\gcd(d, w)=1$, the category 
$\mathbb{T}_S(v)_w^{\rm{red}}$
is a ``non-commutative hyperkähler manifold'', 
so that it shares several properties 
with $D^b(M)$ 
for a smooth projective hyperkähler variety of $K3^{[n]}$-type
for $n=\langle v, v \rangle/2+1$. 
More precisely, we may expect the following, 
which we view as a categorical $\chi$-independence 
phenomenon. 
\begin{conj}\label{conj:HK}
Let $v=dv_0$ for $d\geq 1$ and $v_0$ a primitive vector with $\langle v_0, v_0\rangle=2g-2$.
	Suppose that $g\geq 0$. 
 For $\gcd(d, w)=1$,
	the category $\mathbb{T}_S(v)_{w}^{\rm{red}}$ 
 is deformation equivalent to $D^b(S^{[n]})$ for $n=\langle v, v\rangle/2+1$. 
	\end{conj}

\subsection{Quasi-BPS categories for Gieseker semistable sheaves}
In Definition~\ref{defn:BPS}, we defined quasi-BPS categories 
for Bridgeland semistable objects. 
By applying the categorical wall-crossing 
equivalence in Theorem~\ref{thm:walleq}, we 
can relate the categories in Definition~\ref{defn:BPS} 
with those under Hodge isometries, and 
with those for moduli stacks of Gieseker 
semistable sheaves. 

Let $G$ be the group of Hodge isometries of 
the Mukai lattice $H^{\ast}(S, \mathbb{Z})$, 
preserving the orientation of the positive 
definite four dimensional plane of $H^{\ast}(S, \mathbb{R})$. 
Note that it acts on the algebraic part 
$\mathbb{Z} \oplus \mathrm{NS}(S) \oplus \mathbb{Z}$. 
The following is a categorical analogue of 
derived invariance property of counting invariants
for K3 surfaces~\cite{Tst3, TodK3}. 
\begin{cor}\label{cor:equiv}
For any $\gamma \in G$, there is an equivalence 
\begin{align*}
    \mathbb{T}_S(v)_w \simeq \mathbb{T}_S(\gamma v)_w. 
\end{align*}
\end{cor}
\begin{proof}
Let $\mathrm{Aut}_{\circ}(D^b(S))$ be the 
group of autoequivalences $\Phi$ of $D^b(S)$
whose action $\Phi_{\ast}$ on the space
of stability conditions preserves the component 
$\mathrm{Stab}(S)$. 
It also acts on $H^{\ast}(S, \mathbb{Z})$, and denote the action by $\Phi_{\ast} \colon H^{\ast}(S, \mathbb{Z}) \to H^{\ast}(S, \mathbb{Z})$. 
Then we have the surjective group homomorphism, see~\cite[Proposition~7.9]{HH}, 
\cite[Corollary~4.10]{HMS2}:
\begin{align}\label{aut:surj}
    \mathrm{Aut}_{\circ}(D^b(S)) \to G, \ 
    \Phi \mapsto \Phi_{\ast}
\end{align}

For $\Phi \in \mathrm{Aut}_{\circ}(D^b(S))$, 
there is an equivalence 
    of derived stacks
$\phi \colon \mathfrak{M}_S^{\sigma}(v) \simeq 
\mathfrak{M}_S^{\Phi_{\ast}\sigma}(\Phi_{\ast}\sigma)$
given by $F \mapsto \Phi(F)$. 
The above equivalence 
induces an equivalence 
\begin{align*}
\phi_{\ast} \colon 
\mathbb{T}_S^{\sigma}(v)_{\delta} \simeq 
\mathbb{T}_S^{\Phi_{\ast}\sigma}(\Phi_{\ast}v)_{\delta'}
\end{align*}
where $\delta'$ is determined by $a'=\Phi_{\ast}^{-1}a \in K(S)_{\mathbb{R}}$
which satisfies $\chi(a'\otimes v')=w$
for $v'=\Phi_{\ast}v$. By Theorem~\ref{thm:walleq}
and the surjectivity of (\ref{aut:surj}), 
we obtain the corollary. 
\end{proof}

Let $H$ be an ample divisor on $S$. 
We denote by $\mathfrak{M}_S^H(v)$ the derived moduli 
stack of $H$-Gieseker semistable sheaves on $S$ with 
Mukai vector $v$, 
and by $\mathfrak{M}_S^H(v)^{\rm{red}}$ its reduced stack. 
For $a \in K(S)_{\mathbb{R}}$, we define the
$\mathbb{R}$-line
bundle $\delta$ on $\mathfrak{M}_S^H(v)$, $\mathfrak{M}_S^H(v)^{\rm{red}}$ similarly to 
(\ref{def:delta}). 
Then we define 
\begin{align*}
&\mathbb{T}_S^H(v)_{\delta}:=\mathbb{W}(\mathfrak{M}_S^H(v))_{\delta}^{\rm{int}}
\subset D^b(\mathfrak{M}_S^H(v)), \\
&\mathbb{T}_S^H(v)_{\delta}^{\rm{red}}:=\mathbb{W}(\mathfrak{M}_S^H(v)^{\rm{red}})_{\delta}^{\rm{int}}
\subset D^b(\mathfrak{M}_S^H(v)^{\rm{red}}). 
\end{align*}

Below we consider $H$ generic with respect to $v$, 
so that all Jordan-H\"{o}lder factors of 
objects in $\mathfrak{M}_S^H(v)$ 
have numerical class proportional to $v$. 
The following is a corollary of the wall-crossing 
equivalence in Theorem~\ref{thm:walleq}. 
\begin{cor}\label{cor:lv}
For $v \in N(S)_{\mathbb{R}}$ and generic 
$\sigma \in \mathrm{Stab}(S)$, 
there is $\varepsilon \in \{0, 1\}$ 
and $m\gg 0$ such that, by setting 
$v'=(-1)^{\varepsilon}v(mH)$ we have 
equivalences
\begin{align*}
    \mathbb{T}_S(v)_{\delta} \simeq 
    \mathbb{T}_S^H(v')_{\delta'}, \ 
  \mathbb{T}_S(v)_{\delta}^{\rm{red}} \simeq 
    \mathbb{T}_S^H(v')_{\delta'}^{\rm{red}}.   
\end{align*}
Here, $\delta'$ is a line bundle 
on $\mathfrak{M}_S^H(v)$ determined by 
$a'=(-1)^{\varepsilon}a(-mH) \in K(S)_{\mathbb{R}}$ with $\chi(a \otimes v)=\chi(a' \otimes v')=w$. Then:
\begin{align*}
    \mathbb{T}_S(v)_{w} \simeq 
    \mathbb{T}_S^H(v')_{w}, \ 
  \mathbb{T}_S(v)_{w}^{\rm{red}} \simeq 
    \mathbb{T}_S^H(v')_{w}^{\rm{red}}.   
\end{align*}
\end{cor}
\begin{proof}
We take the autoequivalence $\Phi$ of $D^b(S)$ 
to be either $\Phi=\otimes \mathcal{O}(mH)$
or $\otimes \mathcal{O}(mH)[1]$ for $m\gg 0$, such that  
the vector $\Phi_{\ast}v=(r, \beta, \chi)$ either has $r \geq 0$ and $H \cdot \beta>0$, 
or $r=H \cdot \beta=0$, $\chi>0$. 
Then applying Corollary~\ref{cor:equiv}, 
Theorem~\ref{thm:walleq}, and Proposition~\ref{prop:LV} 
we obtain the conclusion. 
\end{proof}

We next mention the natural periodicity and symmetry of quasi-BPS categories:
\begin{lemma}\label{lem:period}
Let $m:=\gcd\{\chi(a \otimes v) : a \in K(S)\}$. 
We have equivalences 
\begin{align*}
    \mathbb{T}_S(v)_w \simeq \mathbb{T}_S(v)_{w+m}, \
    \mathbb{T}_S(v)_w \simeq \mathbb{T}_S(v)_{-w}^{\rm{op}}.    
\end{align*}
The similar equivalences also hold for $\mathbb{T}_S(v)_w^{\rm{red}}$. 
    \end{lemma}
\begin{proof}
    For $a \in K(S)$, let $\delta$ be the line bundle (\ref{def:delta}).
    Then tensoring by $\delta$ induces equivalences:
    \begin{align*}
        \mathbb{T}_S^{\sigma}(v)_w \simeq \mathbb{T}^{\sigma}_S(v)_{w+\chi(a\otimes v)}, \  
    \mathbb{T}^{\sigma}_S(v)_w^{\rm{red}} \simeq \mathbb{T}^{\sigma}_S(v)_{w+\chi(a\otimes v)}^{\rm{red}}.    
    \end{align*}
    Therefore we obtain the first equivalence. 
    The second equivalence is given by the restriction of 
    $\mathcal{H}om(-, \mathcal{O}_{\mathfrak{M}_S^{\sigma}(v)})$ to 
    $\mathbb{T}_S^{\sigma}(v)$. 
\end{proof}
 \subsection{Conjecture~\ref{conj:HK} for \texorpdfstring{$g=0,1$}{g=0,1}}
Write the Mukai vector as $v=dv_0$, where $d\in\mathbb{Z}_{\geq 1}$ and for $v_0$ a primitive Mukai vector
with $\langle v_0, v_0 \rangle=2g-2$
and $g\geq 0$. 

The following
proposition proves Conjecture~\ref{conj:HK} when $g=0$.
\begin{prop}\label{prop:g=0}
Suppose that $g=0$ and $\gcd(d, w)=1$. Then we have 
\begin{align*}
  \mathbb{T}_S(dv_0)_w^{\rm{red}}=\begin{cases}
      D^b(\Spec \mathbb{C}), & d=1, \\
      0, & d>1.
  \end{cases}  
\end{align*}
\end{prop}
\begin{proof}
 By Corollary~\ref{cor:lv}, we can assume that 
 $\mathbb{T}_S(v)_w=\mathbb{T}_S^H(v)_{\delta}$
 in $D^b(\mathfrak{M}_S^H(v))$ for $H$ an ample divisor on $S$. 
 It is well-known that 
 $\mathfrak{M}_S^H(v)$ consists of a single 
 point $F^{\oplus d}$
 for a spherical stable sheaf $F$, so 
 we have 
 \begin{align*}
     \mathfrak{M}_S^H(v)^{\rm{red}}=
     \Spec \mathbb{C}[\mathfrak{gl}(d)_0^{\vee}[1]]/GL(d). 
 \end{align*}
 By the definition of $\mathbb{T}_S(v)_w^{\rm{red}}$, 
 it consists of objects such that 
 for the inclusion \[j \colon \mathfrak{M}_S^H(v)^{\rm{red}} \hookrightarrow BGL(d),\]
 the object $j_{\ast}\mathcal{E}$ is generated 
 by $\Gamma_{GL(d)}(\chi)$ for a dominant weight $\chi$ 
 such that 
 \begin{align*}
     \chi+\rho \in \frac{1}{2} \mathrm{sum}[0, \beta_i-\beta_j]+\frac{w}{d}\sum_{i=1}^d \beta_i, 
 \end{align*} where the Minkowski sum is after all $1\leq i, j\leq d$.
 By~\cite[Lemma~3.2]{Toquot2}, such a weight exists if and only if $d|w$, and thus only if $d=1$ because $\gcd(d,w)=1$.
 Therefore, together with (\ref{decom:prim}) in the primitive case, the proposition follows. 
\end{proof}

We next discuss the case of $g=1$. 
Then $v=dv_0$, where $v_0$ is primitive with $\langle v_0, v_0 \rangle=0$. 
For a generic $\sigma$, set  
\begin{align*}S':=M_S^{\sigma}(v_0)
\end{align*}
which is well-known to be a K3 surface~\cite{Mu2, BaMa2}. 
We have the good moduli space morphism 
$\mathcal{M}_S^{\sigma}(v_0) \to S'$ which is a $\mathbb{C}^{\ast}$-gerbe
classified by some $\alpha \in \mathrm{Br}(S')$. 
There is an equivalence 
\begin{align}\label{equiv:twist}
D^b(S', \alpha) \stackrel{\sim}{\to} D^b(S)
\end{align}
given by the Fourier-Mukai transform with 
kernel the universal $(1\boxtimes \alpha)$-twisted 
sheaf on $S \times S'$, see~\cite{MR1902629}. 
There is also an isomorphism 
given by the direct sum map 
\begin{equation}\label{directsum}
\mathrm{Sym}^d(S') \stackrel{\cong}{\to} M_S^{\sigma}(dv_0).
\end{equation}

Let $\mathcal{M}_S^{\sigma}(v_0, \ldots, v_0)$
be the classical moduli stack of filtrations 
of semistable objects on $S$:
\begin{align*}
    0=F_0 \subset F_1 \subset \cdots \subset F_d
\end{align*}
such that $F_i/F_{i-1}$ is a $\sigma$-semistable object with numerical class $v_0$.  
We define $\mathscr{Z}_S$ and $\widetilde{\mathcal{M}}_S^{\sigma}(v_0)$ by the following diagram, where the two squares are Cartesian in the classical sense:  
\begin{align}\label{dia:ZS}
    \xymatrix{
    \mathscr{Z}_S  \ar@/^18pt/[rr]^-{p_S}
    \ar@<-0.3ex>@{^{(}->}[r]\ar[d]_-{q_S} \ar@{}[rd]|\square & \mathcal{M}_S^{\sigma}(v_0, \ldots, v_0) \ar[r] \ar[d] & 
    \mathfrak{M}_S^{\sigma}(dv_0)^{\rm{red}} \\
    \widetilde{\mathcal{M}}_S^{\sigma}(v_0) \ar@{}[rd]|\square\ar[d] \ar@<-0.3ex>@{^{(}->}[r] & \mathcal{M}_S^{\sigma}(v_0)^{\times d} \ar[d] &  \\
    S' \ar@<-0.3ex>@{^{(}->}[r]_-{\Delta} & (S')^{\times d}.
    & 
    }
\end{align}
Let $T(d)=(\mathbb{C}^{\ast})^{\times d}$. 
The map $\widetilde{\mathcal{M}}_S^{\sigma}(v_0) \to 
S'$ is a $T(d)$-gerbe, 
so we have the decomposition into $T(d)$-weights
\begin{align}\label{decom:Td}
    D^b(\widetilde{\mathcal{M}}_S^{\sigma}(v_0))
    =\bigoplus_{(w_1, \ldots, w_d)\in \mathbb{Z}^d}
    D^b(\widetilde{\mathcal{M}}_S^{\sigma}(v_0))_{(w_1, \ldots, w_d)}
\end{align}
where the summand corresponding to $(w_1, \ldots, w_d)$
is equivalent to $D^b(S', \alpha^{w_1+\cdots+w_d})$. 
For $1\leq i\leq d$, define $m_i$ by the formula
\begin{align}\label{def:mi}
	m_i :=\left\lceil \frac{wi}{d} \right\rceil -\left\lceil \frac{w(i-1)}{d} \right\rceil
	+\delta_{id} -\delta_{i1} \in \mathbb{Z}.
	\end{align}
Define the functor 
\begin{align*}
    \Phi_{d,w}\colon D^b(S', \alpha^w)\to D^b(\mathfrak{M}_{S}^{\sigma}(v)^{\rm{red}})_{w},\
    \mathcal{F}\mapsto p_{S*}\left(q_S^{\ast}\circ i_{(m_1, \ldots, m_d)}\mathcal{F}\right),
\end{align*}
where $i_{(m_1, \ldots, m_d)}$ is the inclusion of
$D^b(S', \alpha^w)$ into the weight $(m_1, \ldots, m_d)$-part 
of (\ref{decom:Td}). 
 When $v_0=[\mathcal{O}_x]$ for a point $x \in S$, 
 it is proved in~\cite[Proposition~4.7]{PT2} that the image of the functor $\Phi_{d, w}$ 
 lies in $\mathbb{T}_S^{\sigma}(v)_w$. We now state a stronger form of Conjecture \ref{conj:HK} for $g=1$.
\begin{conj}\label{conj:K3}
Let $v=d v_0$ such that $d \in \mathbb{Z}_{\geq 1}$ and 
$v_0$ is primitive with $\langle v_0, v_0 \rangle=0$. 
If $\gcd(d, w)=1$, then the functor $\Phi_{d, w}$ restricts to the equivalence 
\begin{align}\label{def:Phi}
    \Phi_{d, w} \colon D^b(S', \alpha^w) \stackrel{\sim}{\to} \mathbb{T}_S^{\sigma}(dv_0)_{w}^{\rm{red}}. 
\end{align}
In particular for $w=1$, 
the category 
$\mathbb{T}_S^{\sigma}(dv_0)_1$
   is equivalent to $D^b(S)$. 
\end{conj}

In~\cite{PTzero, PT1},
we addressed a similar conjecture for $\mathbb{C}^2$
which we recall here. 
Let $\mathscr{C}(d)^{\rm{red}}$ be the reduced derived 
moduli stack of zero-dimensional sheaves on $\mathbb{C}^2$ with 
length $d$. 
It is the quotient stack 
\begin{align*}
    \mathscr{C}(d)^{\rm{red}}=\mu_0^{-1}(0)/GL(d),
\end{align*}
where $\mu_0 \colon \mathfrak{gl}(d)^{\oplus 2} \to \mathfrak{gl}(d)_0$
is the commuting map, so the map (\ref{mu0:trace}) for $g=1$. 
Let $\mathscr{C}(1, \ldots, 1)$ be the classical moduli stack of filtrations of 
zero-dimensional sheaves on $\mathbb{C}^2$:
\begin{align*}
    0=Q_0 \subset Q_1 \subset \cdots \subset Q_d
\end{align*}
such that $Q_i/Q_{i-1}$ is isomorphic to $\mathcal{O}_{x_i}$ for some $x_i \in \mathbb{C}^2$. 
Similarly to (\ref{dia:ZS}), we have the 
following diagram 
\begin{align}\label{dia:ZS2}
    \xymatrix{
    \mathscr{Z}_{\mathbb{C}^2}  \ar@/^18pt/[rr]^-{p_{\mathbb{C}^2}}
    \ar@<-0.3ex>@{^{(}->}[r]\ar[d]_-{q_{\mathbb{C}^2}} \ar@{}[rd]|\square & \mathscr{C}(1, \ldots, 1) \ar[r] \ar[d] & 
    \mathscr{C}(d)^{\rm{red}} \\
    \mathbb{C}^2/T(d) \ar@<-0.3ex>@{^{(}->}[r]_-{\Delta} & (\mathbb{C}^2)^{\times d}/T(d). & 
    }
\end{align}
The functor 
\begin{align}\label{funct:C2}
    \Phi_{d, w}^{\mathbb{C}^2} \colon 
    D^b(\mathbb{C}^2) \to D^b(\mathscr{C}(d)^{\rm{red}})
\end{align}
is defined similarly to (\ref{def:Phi}). 
\begin{conj}\emph{(\cite{PTzero, PT1, PaTobps})}\label{conj:C2}
   If $\gcd(d, w)=1$, the functor (\ref{funct:C2})
   restricts to the equivalence 
   \begin{align}\label{equiv:conjC2}
   \Phi_{d, w}^{\mathbb{C}^2} \colon 
       D^b(\mathbb{C}^2) \stackrel{\sim}{\to} \mathbb{T}(d)_w^{\rm{red}}. 
   \end{align}
   Here the right hand side is defined in (\ref{def:Tdwred}) for $g=1$. 
\end{conj}
We have the following proposition: 
\begin{prop}\label{prop:conj}
    Conjecture~\ref{conj:C2} implies Conjecture~\ref{conj:K3}. 
\end{prop}
\begin{proof}
      Consider the composition 
    \begin{align}\label{compose}
    D^b(\mathfrak{M}_S^{\sigma}(dv_0)^{\rm{red}}) \stackrel{p_S^{!}}{\to}\Ind D^b(\mathscr{Z}_S)
    \stackrel{q_{S\ast}}{\to}\Ind D^b(\widetilde{\mathcal{M}}_S^{\sigma}(v_0)) \stackrel{\mathrm{pr}}{\to}\Ind D^b(S', \alpha^w),
    \end{align}
    where $\mathrm{pr}$ is the projection onto the weight $(m_1, \ldots, m_d)$-component. 
 We claim that, assuming Conjecture~\ref{conj:C2}, the above functor restricts 
    to the functor 
    \begin{align}\label{Phiright}
        \Phi_{d, w}^{R} \colon \mathbb{T}^\sigma_S(dv_0)_w^{\rm{red}} \to D^b(S', \alpha^w),
    \end{align}
    which is a right adjoint of $\Phi_{d, w}$. 
        Let 
        \begin{align*}\mathcal{M}_S^{\sigma}(dv_0) \to M_S^{\sigma}(dv_0) \stackrel{\cong}{\leftarrow} 
        \mathrm{Sym}^d(S')
        \end{align*}
    be the good moduli space morphism, see \eqref{directsum}. 
    
  For a point $p \in S'$, 
  the diagram (\ref{dia:ZS}) pulled back over the formal completion 
  $\Spec \widehat{\mathcal{O}}_{\mathrm{Sym}^d(S'), d[p]} \to \mathrm{Sym}^d(S')$
  is isomorphic to the diagram (\ref{dia:ZS2})
  pulled back 
via $\Spec \widehat{\mathcal{O}}_{\mathrm{Sym}^d(\mathbb{C}^2), d[0]} \to 
\mathrm{Sym}^d(\mathbb{C}^2)$. 
The ind-completion of the equivalence (\ref{equiv:conjC2}) gives an equivalence 
\[\Ind D^b(\mathbb{C}^2) \stackrel{\sim}{\to} \Ind \mathbb{T}(d)_w^{\rm{red}},\]
whose inverse is 
\begin{align}\label{functorprop418}
    \mathrm{pr} \circ q_{\mathbb{C}^2 \ast} p_{\mathbb{C}^2}^! \colon 
    \Ind\mathbb{T}(d)_w^{\rm{red}} \to \Ind D^b(\mathbb{C}^2). 
\end{align}
In the above, $\mathrm{pr}$ is again the projection functor onto the weight $(m_1,\ldots, m_d)$-component. 
By the equivalence (\ref{equiv:conjC2}), the functor \eqref{functorprop418}
restricts to the functor $\mathbb{T}(d)_w^{\rm{red}}\to D^b(\mathbb{C}^2)$.
Therefore the functor (\ref{compose})
restricts to the functor (\ref{Phiright}), giving 
a right adjoint of $\Phi_{d, w}$. 

  We have the natural transformations 
    $\id \to \Phi_{d, w}^R \circ \Phi_{d, w}$, 
    $\Phi_{d, w}\circ \Phi_{d, w}^R \to \id$ by adjunction, which are isomorphisms 
    formally locally on $\mathrm{Sym}^d(S')$. Hence they are isomorphisms 
    and thus $\Phi_{d, w}$ is an equivalence. 
  \end{proof}
  \begin{remark}\label{rmk:21}
In~\cite{PaTobps}, we prove Conjecture~\ref{conj:C2}
for $(d, w)=(2, 1)$. 
By Proposition \ref{prop:conj}, it implies that 
Conjecture~\ref{conj:K3} is true for $(d, w)=(2, 1)$. 
\end{remark}

\section{Semiorthogonal decompositions into quasi-BPS categories}
In this section, 
we prove a categorical version of the PBW theorem for cohomological Hall algebras of K3 surfaces \cite{DHSM}, see Theorem \ref{thm:sodK3}. 
We first prove
Theorem \ref{thm:sodK3} 
assuming Proposition~\ref{prop:sod2}, which states that there is a semiorthogonal decomposition formally locally on the good moduli space. 
We then prove Proposition~\ref{prop:sod2}. 


\subsection{Semiorthogonal decomposition}
Let $S$ be a K3 surface. 
We take $v \in N(S)$ and write $v=dv_0$ for 
$d \in \mathbb{Z}_{\geq 1}$ and primitive $v_0$. 
For a partition $d=d_1+\cdots+d_k$, let 
$\mathfrak{M}_S^{\sigma}(d_1 v_0, \ldots, d_k v_0)$
be the derived moduli stack of filtrations 
\begin{align*}
    0=F_0 \subset F_1 \subset \cdots \subset F_k
\end{align*}
such that $F_i/F_{i-1}$ is $\sigma$-semistable with 
numerical class $d_i v_0$. 
Consider the natural morphisms
\begin{align}\label{PortaSala}
    \times_{i=1}^k \mathfrak{M}_S^{\sigma}(d_i v_0)
    \stackrel{q}{\leftarrow} \mathfrak{M}_S^{\sigma}(d_1 v_0, \ldots, d_k v_0) \stackrel{p}{\to} \mathfrak{M}_S^{\sigma}(d v_0),
\end{align}
where $q$ is quasi-smooth and $p$ is proper. 
The above morphisms 
induce the categorical Hall product, see~\cite{PoSa}:
\begin{align}\label{hall:k3}
    p_{\ast}q^{\ast} \colon 
    \boxtimes_{i=1}^k D^b(\mathfrak{M}_S^{\sigma}(d_i v_0))
    \to D^b(\mathfrak{M}_S^{\sigma}(d v_0)). 
\end{align}
We next discuss a semiorthogonal decomposition of $D^b(\mathfrak{M}_S^{\sigma}(v))$ using categorical Hall products of quasi-BPS categories, which we view as a categorical version of the PBW theorem for cohomological Hall algebras~\cite[Theorem C]{DM}, \cite[Corollary 1.6]{DHSM}. When $v_0$ is the class of a point, the statement was proved in \cite{P2}.
 
\begin{thm}\label{thm:sodK3}
Assume $v=dv_0$ for $d\in\mathbb{Z}_{\geq 1}$ and for a primitive Mukai vector $v_0$. For a generic stability condition $\sigma$, there 
is a semiorthogonal decomposition 
\begin{align}\label{sod:main}
 D^b(\mathfrak{M}_S^{\sigma}(v))
    =\left\langle 
    \boxtimes_{i=1}^k \mathbb{T}_{S}^{\sigma}(d_i v_0)_{w_i+(g-1)d_i(\sum_{i>j}d_j-\sum_{i<j}d_j)}
    \right\rangle. 
    \end{align}
    The right hand side is after all partitions $(d_i)_{i=1}^k$ of $d$ and all weights $(w_i)_{i=1}^k\in\mathbb{Z}^k$ such that 
    \[\frac{w_1}{d_1}<\cdots<\frac{w_k}{d_k}.\] 
    Each semiorthogonal summand is given by the 
    restriction of the categorical Hall product (\ref{hall:k3}), 
    and the order of the semiorthogonal 
    decomposition is the same as that of (\ref{sod:triple}).
\end{thm}

\begin{proof}
We first explain that the semiorthogonal decomposition \eqref{sod:main} holds formally locally over the good moduli space $M^\sigma_S(v)$. 
For each $y \in M_S^{\sigma}(v)$, recall 
the equivalence 
\begin{equation}\label{msigmap}
\widehat{\mathfrak{M}}_S^{\sigma}(v)_y \simeq 
\widehat{\mathscr{P}}(d)_p\end{equation}
from Lemma~\ref{lem:py}. 
For a $\mathbb{R}$-line bundle $\delta$ on $\mathfrak{M}_S^{\sigma}(v)$ 
as in (\ref{def:delta}), its restriction 
$\delta_y$ 
to $\widehat{\mathfrak{M}}_S^{\sigma}(v)_y$ 
corresponds to $w\tau_d$ under the above equivalence
by the computation (\ref{deltay}). 
Therefore, the category 
$\mathbb{T}_{S, y}^{\sigma}(v)_{\delta_y}$ from Remark~\ref{rmk:qbps}
is equivalent to the category $\mathbb{T}_p(d)_w$ from (\ref{qbps:that}) under the equivalence \eqref{msigmap}, 
as both of them are intrinsic window subcategories
of equivalent derived stacks. 
Therefore the statement holds formally locally 
at any point $y \in M_S^{\sigma}(v)$
by Proposition~\ref{prop:sod2}. 
We set
\begin{align}\label{part:A}
    A=(d_i, w_i')_{i=1}^k, \ 
    w_i':=w_i+(g-1)d_i\left(\sum_{i>j}d_j-\sum_{i<j}d_j\right).
    \end{align}
Every functor 
\begin{align}\label{upA}
\Upsilon_{A} \colon 
 \boxtimes_{i=1}^k \mathbb{T}_{S}^{\sigma}(d_i v_0)_{w_i'}
\to D^b(\mathfrak{M}_S^{\sigma}(v))
\end{align}
is globally defined via categorical Hall product, 
hence a standard argument
reduces the existence of the desired SOD to the formal 
local statement as in~\cite[Section~4.2]{P2}, 
also see~\cite{PT2, T, Totheta, T3} for the similar  
arguments on reduction to formal fibers. 

We give more details on the proof. 
We prove the semiorthogonal 
decomposition (\ref{sod:main}) 
by induction on $d$. 
The case of $d=1$ is obvious, so we assume that $d\geq 2$. 
We first show
that, for $w_1/d_1<\cdots<w_k/d_k$,
the functor (\ref{upA})  is fully-faithful.
   By the induction hypothesis, the inclusion 
   \begin{align*}
       \boxtimes_{i=1}^k \mathbb{T}_S^{\sigma}(d_i v_0)_{w_i'} 
       \hookrightarrow \boxtimes_{i=1}^k D^b(\mathfrak{M}_S^{\sigma}(d_i v_0))_{w_i'}
   \end{align*}
   admits a right adjoint. 
   The categorical Hall product restricted to the fixed $(\mathbb{C}^{\ast})^k$-weights $(w_i)_{i=1}^k$:
   \begin{align*}
       \boxtimes_{i=1}^k D^b(\mathfrak{M}_S^{\sigma}(d_i v_0))_{w_i'}
       \to D^b(\mathfrak{M}_S^{\sigma}(v))
   \end{align*}
   also admits a right adjoint, see the proof of~\cite[Lemma~6.7]{Totheta} or~\cite[Theorem~1.1]{P2}. 
   Therefore the functor (\ref{upA}) admits a right adjoint $\Upsilon_A^R$. 
   To show that \eqref{upA} is fully-faithful, it is enough to show that the natural transform
   \begin{align}\label{isom:upA}
     \id \to \Upsilon_A^R \circ \Upsilon_A
     \end{align}
     is an isomorphism. 
   This is a local question for $M_S^{\sigma}(v)$, 
   i.e. it is enough to show that (\ref{isom:upA}) is an isomorphism 
   after restricting to $\widehat{\mathfrak{M}}_S^{\sigma}(v)_y$
   for any $y\in M_S^{\sigma}(v)$. Since 
   $\Upsilon_A$ and $\Upsilon_A^R$ are compatible with 
   pull-backs to $\widehat{\mathfrak{M}}_S^{\sigma}(v)_y$, 
   the isomorphism (\ref{isom:upA}) on 
   $\widehat{\mathfrak{M}}_S^{\sigma}(v)_y$
   follows from  
   Lemma~\ref{lem:py} and 
   Proposition~\ref{prop:sod2}. 

We next show that there is a semiorthogonal decomposition of the form 
\begin{align}\label{sod:upw}
    D^b(\mathfrak{M}_S^{\sigma}(v))_w=\langle \{\mathrm{Im}\Upsilon_A\}_{A \in \Gamma}, \mathbb{W} \rangle,
\end{align}
where $\Gamma$ is the set of partitions
$A=(d_i, w'_i)_{i=1}^k$ of $(d, w)$ as in (\ref{part:A})
such that $k\geq 2$ and $w_1/d_1<\cdots<w_k/d_k$. 
    For $A>B$, we have 
    $\Hom(\mathrm{Im}\Upsilon_A, \mathrm{Im}\Upsilon_B)=0$. 
    Indeed it is enough to show that $\Upsilon_A^R \circ \Upsilon_B=0$, 
    which is a property local on $M_S^{\sigma}(v)$. Hence similarly to showing \eqref{isom:upA} is an isomorphism, 
    the desired vanishing follows from 
    Proposition~\ref{prop:sod2}.
We next show that the functor (\ref{upA}) admits a left adjoint $\Upsilon_A^L$. 
Let $\mathbb{D}_{\mathfrak{M}}$ be the dualizing functor 
\begin{align*}
    \mathbb{D}_{\mathfrak{M}} \colon D^b(\mathfrak{M}_S^{\sigma}(v)) \stackrel{\sim}{\to} 
    D^b(\mathfrak{M}_S^{\sigma}(v))^{\rm{op}}. 
\end{align*}
The above functor 
restricts to the equivalence $\mathbb{D}_{\mathbb{T}(d)} \colon \mathbb{T}_S^{\sigma}(v)_{\delta}
\stackrel{\sim}{\to} \mathbb{T}_S^{\sigma}(v)_{-\delta}^{\rm{op}}$. 
For a partition $A$ in (\ref{part:A}), we set 
$A^{\vee}=(d_i, -w_i')_{i=1}^k$. 
Then the functor 
\begin{align*}
    \Upsilon_A^L:= \left(\boxtimes_{i=1}^k \mathbb{D}_{\mathbb{T}(d_i)}\right)
    \circ (\Upsilon_{A^{\vee}}^R)^{\rm{op}} \circ 
    \mathbb{D}_{\mathfrak{M}} \colon 
    D^b(\mathfrak{M}_S^{\sigma}(v)) \to 
    \boxtimes_{i=1}^k \mathbb{T}_S^{\sigma}(d_i v_0)_{w_i'}
    \end{align*}
gives a left adjoint of $\Upsilon_A$. 
Therefore we obtain the semiorthogonal decomposition of the form (\ref{sod:upw}).     

It is enough to show that $\mathbb{W}=\mathbb{T}_S^{\sigma}(v)_w$ in the 
    semiorthogonal decomposition (\ref{sod:upw}).
The inclusion 
$\mathbb{T}_S^{\sigma}(v)_w \subset \mathbb{W}$ follows from a formal local argument as above.
It thus suffices to show that 
$\mathbb{W} \subset \mathbb{T}_S^{\sigma}(v)_w$. 
The subcategory $\mathbb{W}$ consists of $\mathcal{E} \in D^b(\mathfrak{M}_S^{\sigma}(v))_w$
such that $\Upsilon_A^L(\mathcal{E})=0$ for all $A \in \Gamma$. 
This is a local property on $M_S^{\sigma}(v)$. 
The functor $\Upsilon_A^L$ is compatible with pull-back to $\widehat{\mathfrak{M}}_S^{\sigma}(v)_y$. 
Thus, for any $\mathcal{E} \in \mathbb{W}$, 
we have $\mathcal{E}|_{\widehat{\mathfrak{M}}^{\sigma}_S(v)_y} \in \mathbb{T}_{S, y}^{\sigma}(v)_{\delta_y}$
by Lemma~\ref{lem:py} and Proposition~\ref{prop:sod2}. 
Therefore, from Remark~\ref{rmk:qbps}, we conclude that 
$\mathcal{E} \in \mathbb{T}_S^{\sigma}(v)_w$. 
\end{proof}

The reduced version of the semiorthogonal decomposition is as 
follows: 
\begin{thm}\label{thm:sodK32}
   Assume $v=dv_0$ for $d\in\mathbb{Z}_{\geq 1}$ and for a primitive Mukai vector $v_0$. For a generic stability condition $\sigma$, there 
is a semiorthogonal decomposition 
\begin{align*}
 &D^b(\mathfrak{M}_S^{\sigma}(v)^{\rm{red}})
    =\\
    &\left\langle 
    \boxtimes_{i=1}^{k-1} \mathbb{T}_{S}^{\sigma}(d_i v_0)_{w_i+(g-1)d_i(\sum_{i>j}d_j-\sum_{i<j}d_j)}
    \boxtimes \mathbb{T}_S^{\sigma}(d_k v_0)_{w_k+(g-1)d_k(\sum_{k>j}d_j)}^{\rm{red}}
    \right\rangle. 
    \end{align*}
    The right hand side is after 
    $d_1+\cdots+d_k=d$ such that $w_1/d_1<\cdots<w_k/d_k$.  
\end{thm}
\begin{proof}
    Let $v_0=(r, \beta, \chi)$. We have the commutative diagram 
    \begin{align*}
    \xymatrix{
\times_{i=1}^k \mathfrak{M}_S^{\sigma}(d_i v_0) \ar[d]_-{\times_{i=1}^k \det} & \ar[l]_-{q}
\mathfrak{M}_S^{\sigma}(d_1 v_0, \ldots, d_k v_0) \ar[r]^-{p} & 
\mathfrak{M}_S^{\sigma}(dv_0) \ar[d]_-{\det} \\
 \times_{i=1}^k \mathcal{P}ic^{d_i \beta}(S)  \ar[rr] \ar@{}[rrd]|\square &  & 
 \mathcal{P}ic^{d\beta}(S) \\
 \times_{i=1}^{k-1}\mathcal{P}ic^{d_i \beta}(S) \times B\mathbb{C}^{\ast} \ar[rr] \ar[u] & & 
 B\mathbb{C}^{\ast}, \ar[u] 
    }
    \end{align*}
    where the middle horizontal arrow is $(L_1, \ldots, L_k) \mapsto L_1 \otimes \cdots \otimes L_k$. 
    By base change, the categorical Hall product induces the functor 
    \begin{align*}
        \boxtimes_{i=1}^{k-1}D^b(\mathfrak{M}_S^{\sigma}(d_i v_0))
        \boxtimes D^b(\mathfrak{M}_S^{\sigma}(d_k v_0)^{\rm{red}}) \to 
        D^b(\mathfrak{M}_S^{\sigma}(dv_0)^{\rm{red}}).
    \end{align*}
    The rest of the argument is the same as in Theorem~\ref{thm:sodK3}. 
\end{proof}

\subsection{Generation from ambient spaces}
The rest of this section is devoted to 
the proof of Proposition~\ref{prop:sod2}.
In this subsection, we prove technical 
preliminary results about 
generation of dg-categories 
from ambient spaces 
and the restriction of semiorthogonal 
decompositions to formal fibers. 

Let $\mathcal{U}$ be a reduced $\mathbb{C}$-scheme of finite 
type 
with an action of a reductive algebraic group $G$. 
Let $\mathcal{U}/G \to T$ be a morphism to an affine 
scheme $T$ of finite type. For a closed point $y \in T$, 
we denote by $\widehat{\mathcal{U}}_y/G$ the 
formal fiber at $y$. 
We denote by $\iota_y$ the induced map 
$\iota_y \colon \widehat{\mathcal{U}}_y/G \to \mathcal{U}/G$. 
Recall the definition of classical generation from Subsection~\ref{notation2}. 

\begin{lemma}\label{lem:gen1}
	The image of the pull-back functor 
	\begin{align}\label{iotay}
		\iota_y^{\ast} \colon D^b(\mathcal{U}/G) \to 
		D^b(\widehat{\mathcal{U}}_y/G)
		\end{align}
	classically generates $D^b(\widehat{\mathcal{U}}_y/G)$. 
	\end{lemma}
\begin{proof}
	It is enough to show that 
	$\Ind D^b(\widehat{\mathcal{U}}_y/G)$ is generated by the image of 
	\begin{align}\label{iota:y}
		\iota_y^{\ast} \colon \Ind D^b(\mathcal{U}/G) \to 
		\Ind D^b(\widehat{\mathcal{U}}_y/G). 
		\end{align}
	Indeed, suppose that $\Ind D^b(\widehat{\mathcal{U}}_y/G)$ is 
 generated by the image of (\ref{iota:y}). Let
	$\mathcal{C}_y \subset D^b(\widehat{\mathcal{U}}_y/G)$
	be the subcategory classically 
	generated by the image of (\ref{iotay}).
	Then we have 
	$\Ind \mathcal{C}_y \stackrel{\sim}{\to} \Ind D^b(\widehat{\mathcal{U}}_y/G)$, 
	hence $\mathcal{C}_y=D^b(\widehat{\mathcal{U}}_y/G)$ as both of them 
	are the subcategories of compact objects in 
	$\Ind \mathcal{C}_y$ and $\Ind D^b(\widehat{\mathcal{U}}_y/G)$, respectively.  
	
	Let $\mathscr{Z} \subset \mathcal{U}$ be a $G$-invariant 
	closed subset, and define 
	$\mathcal{U}^{\circ}=\mathcal{U} \setminus \mathscr{Z}$.
	Let 
	$i \colon \mathscr{Z} \hookrightarrow \mathcal{U}$ be the closed 
	immersion and 
	$j \colon \mathcal{U}^{\circ} \hookrightarrow \mathcal{U}$
	be the open immersion.    
	For any $\mathcal{E} \in \Ind D^b(\mathcal{U}/G)$, 
	we have the distinguished triangle 
	\begin{align*}
		R\Gamma_{\mathscr{Z}}(\mathcal{E}) \to \mathcal{E} 
		\to j_{\ast}j^{\ast}\mathcal{E}\to R\Gamma_{\mathscr{Z}}(\mathcal{E})[1],
		\end{align*}
	where $R\Gamma_{\mathscr{Z}}(\mathcal{E})$ is an object in 
	\begin{align*}
	    \Ind D^b_{\mathscr{Z}}(\mathcal{U}/G)=
     \mathrm{Ker}\left(j^{\ast} \colon \Ind D^b(\mathcal{U}/G) \to \Ind D^b(\mathcal{U}^{\circ}/G) \right)
     \end{align*}
     and 
	$j_{\ast}j^{\ast}\mathcal{E}$ is an object 
	in $j_{\ast}\Ind D^b(\mathcal{U}^{\circ}/G)$. 
	Note that
 by~\cite[Proposition~6.1.3]{MR3701352}, the category
 $\Ind D_{\mathscr{Z}}^b(\mathcal{U}/G)$ is generated by 
	the image of 
	\begin{align*}
		i_{\ast} \colon \Ind D^b(\mathscr{Z}/G) \to \Ind D^b_{\mathscr{Z}}(\mathcal{U}/G). 
		\end{align*}
	We have the Cartesian diagrams 
	\begin{align*}
	\xymatrix{
\widehat{\mathcal{U}}_y^{\circ} \inclusion^-{\widehat{j}} \ar[d]_{\iota_y^{\circ}} & \widehat{\mathcal{U}}_y
\ar[d]_-{\iota_y} & \widehat{\mathscr{Z}}_y \ar[l]_-{\widehat{i}}	\ar[d]_{\overline{\iota}_y} \\
\mathcal{U}^{\circ} \inclusion^-{j} & \mathcal{U} & \mathscr{Z} \ar[l]_-{i}.
}	
		\end{align*}
There are base change isomorphisms, see~\cite[Corollary~3.7.14]{MR3037900}:
\begin{align*}
  \iota_y^{\ast}j_{\ast} \cong \widehat{j}_{\ast}\iota_y^{\circ \ast}&\colon  
\Ind D^b(\mathcal{U}^{\circ}/G) \to \Ind D^b(\widehat{\mathcal{U}}_y/G), \\
\iota_y^{\ast}i_{\ast} \cong \widehat{i}_{\ast}\overline{\iota}_y^{\ast}&\colon 
\Ind D^b(\mathscr{Z}/G) \to \Ind D^b(\widehat{\mathcal{U}}_y/G),
\end{align*}
we can replace $\mathcal{U}$ with 
$\mathcal{U}^{\circ} \sqcup \mathscr{Z}$. 	
Then, by taking a stratification of $\mathcal{U}$
and repeating the above argument, we can assume that 
$\mathcal{U}$ is smooth. 
Then 
\begin{align*}
	\Ind D^b(\mathcal{U}/G)=D_{\rm{qc}}(\mathcal{U}/G)=\Ind \rm{Perf}(\mathcal{U}/G)
	\end{align*}
and it is a standard fact that the image of 
$\rm{Perf}(\mathcal{U}/G) \to \rm{Perf}(\widehat{\mathcal{U}}_y/G)$
classically generates $\rm{Perf}(\widehat{\mathcal{U}}_y/G)$
(see the argument of~\cite[Lemma~5.2]{MR2801403}). 
		\end{proof}

Let $Y$ be a smooth affine variety with an action of a reductive 
algebraic group $G$. Let $V \to Y$ be a $G$-equivariant vector bundle with 
a $G$-invariant section $s$. We set 
$\mathfrak{U}$ to be the derived zero locus of $s$, 
and 
$\mathcal{U} \hookrightarrow \mathfrak{U}$ 
its classical truncation. 
We have the following diagram 
\begin{align*}
	\xymatrix{
	\mathcal{U}/G \inclusion \ar[d] & \mathfrak{U}/G \inclusion \ar[rd]_-{f} & Y/G  \ar[d] \\
	\mathcal{U}\ssslash G \ar@<-0.3ex>@{^{(}->}[rr] & & Y\ssslash G.  
}
	\end{align*}
For $y \in \mathcal{U}\ssslash G$, we denote by $\widehat{Y}_y$ the formal fiber of 
$Y \to Y\ssslash G$ at $y$, and by
$\widehat{\mathfrak{U}}_y \hookrightarrow \widehat{Y}_y$ the 
derived zero locus of 
$s$ restricted to $\widehat{Y}_y$. 
Let $\iota_y \colon \widehat{\mathfrak{U}}_y/G \to \mathfrak{U}/G$
be the induced map. 
\begin{lemma}\label{lem:gen2}
	The image of the pull-back functor 
	\begin{align*}
		\iota_y^{\ast} \colon D^b(\mathfrak{U}/G) \to 
		 D^b(\widehat{\mathfrak{U}}_y/G)
		\end{align*}
	classically generates $D^b(\widehat{\mathfrak{U}}_y/G)$. 
	\end{lemma}
\begin{proof}
Since $D^b(\widehat{\mathfrak{U}}_y/G)$ is 
classically generated by the image of 
\begin{align*}
	D^b(\widehat{\mathcal{U}}_y/G) \to
	 D^b(\widehat{\mathfrak{U}}_y/G)
	\end{align*}
the claim follows from Lemma~\ref{lem:gen1}. 
	\end{proof}
\begin{lemma}\label{lem:basechange}
    Let $D^b(\mathfrak{U}/G)=\langle \mathbb{T}_i \mid i \in I \rangle$ 
    be a $(Y\sslash G)$-linear semiorthogonal decomposition. 
    Let 
    $\widehat{\mathbb{T}}_{i, y} \subset 
    D^b(\widehat{\mathfrak{U}}_y/G)$ be the subcategory classically 
    generated by the image of 
    $\iota_y^{\ast} \colon 
    \mathbb{T}_i \to D^b(\widehat{\mathfrak{U}}_y/G)$. 
    Then there is a semiorthogonal decomposition 
    \[D^b(\widehat{\mathfrak{U}}_y/G)=\langle 
    \widehat{\mathbb{T}}_{i, y} \mid i \in I\rangle.\] 
\end{lemma}
\begin{proof}
 The subcategories $\widehat{\mathbb{T}}_{i, y}$ classically
 generate $D^b(\widehat{\mathfrak{U}}_y/G)$
 by Lemma~\ref{lem:gen2}. 
 As for the semiorthogonality, take $i, j \in I$
such that $\Hom(\mathbb{T}_i, \mathbb{T}_j)=0$. 
 Then for $A \in \mathbb{T}_i$ and $B \in \mathbb{T}_j$, 
 we have 
 \begin{align}\label{HomAB}
     \Hom(\iota_y^{\ast}A, \iota_y^{\ast}B)
     =\Hom(A, B \otimes \iota_{y\ast}\mathcal{O}_{\widehat{\mathfrak{U}}_y/G}) 
     =\Hom(A, B \otimes f^{\ast}\widehat{\mathcal{O}}_{Y\ssslash G, y}).
 \end{align}
 The sheaf $\widehat{\mathcal{O}}_{Y\ssslash G, y}$
 is an object of $D_{\rm{qc}}(Y\ssslash G)=\Ind\mathrm{Perf}(Y\ssslash G)$, 
 hence $f^{\ast}\widehat{\mathcal{O}}_{Y\ssslash G, y}
 \in D_{\rm{qc}}(\mathfrak{U}/G)$, 
 and $\otimes$ is the 
 action of $D_{\rm{qc}}(\mathfrak{U}/G)$ on 
 $\Ind D^b(\mathfrak{U}/G)$, which recall is continuous (i.e. it preserves small coproducts). 
 Then $B \otimes f^{\ast}\widehat{\mathcal{O}}_{Y\ssslash G, y}$
 is an object of $\Ind \mathbb{T}_{j}$, 
 and by writing it as $\mathrm{colim}_{k \in K} B_k$
 for $B_k \in \mathbb{T}_j$, 
 we have 
 \begin{align*}
     \Hom(A, \mathrm{colim}_{k\in K}B_k)=
     \mathrm{colim}_{k\in K}\Hom(A, B_k)=0,
 \end{align*}
 where the first identity follows 
 as $A$ is compact, and the second vanishing 
 holds from $\Hom(\mathbb{T}_{i}, \mathbb{T}_j)=0$.
 \end{proof}

\subsection{Descriptions of quasi-BPS categories for doubled quivers}
In this subsection, we give an alternative description of quasi-BPS categories for doubled 
quivers, which will be used in the proof of 
Proposition~\ref{prop:sod2}. 
Below we keep the notation in Subsection~\ref{subsub:formal}. 

Let $Q^{\circ}$ be a $g$-loop quiver. 
For $d \in \mathbb{N}$, let $\X(d)$ be moduli stack of $d$-dimensional representations of the 
    tripled quiver of $Q^{\circ}$:
    \[\X(d)=\mathfrak{gl}(d)^{\oplus 2g+1}/GL(d).\]
Consider the regular function induced by the tripled potential: 
\begin{align}\label{TrW:X}
    \Tr W(x_1, \ldots, x_g, y_1, \ldots, y_g, z) =
    \Tr \sum_{i=1}^g z[x_i, y_i] \colon \X(d) \to \mathbb{C}.
    \end{align}
Let $\mathbb{C}^{\ast}$ act on $z$ with weight two. 
We define the subcategory 
\begin{align}\label{sgr:w}
\mathbb{S}^{\rm{gr}}(d)_w \subset \mathrm{MF}^{\rm{gr}}(\X(d), \Tr W)  
\end{align}
to be classically generated by matrix factorizations whose factors 
are direct sums of $\mathcal{O}_{\X(d)} \otimes \Gamma_{GL(d)}(\chi)$
such that 
\begin{align*}
    \chi+\rho \in \mathbf{W}(d)_w=\frac{1}{2} \mathrm{sum}[0, \beta]+w\tau_d=\left(\frac{2g+1}{2}\mathrm{sum}_{1\leq i,j\leq d}[0, \beta_i-\beta_j]\right)+w\tau_d
    \end{align*}
where the Minkowski sums above are after all the $T(d)$-weights of $R_{Q}(d)=\mathfrak{gl}(d)^{\oplus 2g+1}$. 
Alternatively, by~\cite[Lemma~2.9]{hls}, the subcategory (\ref{sgr:w}) consists of matrix factorizations 
with factors $\mathcal{O}_{\X(d)} \otimes \Gamma$
for a $GL(d)$-representation $\Gamma$ whose $T(d)$-weights are contained in 
\begin{align*}
    \nabla(d)_w =\left\{\chi \in M(d)_{\mathbb{R}} : 
    -\frac{1}{2}n_{\lambda} \leq \langle \lambda, \chi \rangle 
    \leq \frac{1}{2}n_{\lambda} \mbox{ for all } \lambda \colon 
    \mathbb{C}^{\ast} \to T(d)\right\}+w\tau_d. 
\end{align*}
Here, the width $n_{\lambda}$ is defined by 
\begin{align*}
    n_{\lambda}:=\left\langle \lambda, \det \left((R_Q(d)^{\vee})^{\lambda>0}\right)-\det \left((\mathfrak{gl}(d)^{\vee})^{\lambda>0}\right) \right\rangle=2g \left\langle \lambda, \det \left(\mathfrak{gl}(d)^{\lambda>0}\right)\right\rangle. 
\end{align*}
The equivalence (\ref{Koszul:theta}) restricts to the equivalence, see Lemma~\ref{lem:genJ}:
\begin{align}\label{theta:rest}
    \Theta \colon \mathbb{T}(d)_w \stackrel{\sim}{\to} \mathbb{S}^{\rm{gr}}(d)_w. 
\end{align}

We next give another descriptions of 
the subcategory (\ref{qbps:that}) based on Lemma~\ref{lem:compareT}. 
As in Subsections~\ref{subsec:onevert}, \ref{subsub:formal}, 
let $\mathscr{P}(d)$ be the 
derived moduli stack of $d$-dimensional 
representations of the quiver $Q^{\circ,d}$ with relation $\mathscr{I}$. 
There is a good moduli space map \[\pi_{P,d}\colon \mathscr{P}(d)^{\rm{cl}} \to P(d).\]
Let $p \in P(d)$ be a closed point corresponding to the 
semisimple $(Q^{\circ,d}, \mathscr{I})$-representation $R_p$ as in (\ref{Rp}):
\begin{equation}\label{Rp2}
R_p=\bigoplus_{i=1}^m W^{(i)} \otimes R^{(i)}, \end{equation} 
where $R^{(i)}$ is a simple representation 
of dimension $r^{(i)}$ and 
$W^{(i)}$ is a finite dimensional $\mathbb{C}$-vector space. 
Recall that $G_p=\prod_{i=1}^m GL(W^{(i)})$ and 
let $T_p \subset G_p$ be a maximal torus. 
Note that we have an isomorphism of $G_p$-representations:
\begin{align}\label{ext1}
\Ext_{Q^{\circ, d}}^1(R_p, R_p)
\oplus \mathfrak{gl}(d)^{\vee}
=\bigoplus_{i, j}\Hom(W^{(i)}, W^{(j)})^{\oplus (\delta_{ij}+2g r^{(i)} r^{(j)})}.
    \end{align}
 Let $M_p$ be the character lattice of $T_p$ and let
$\tau_{d, p} \in (M_p)_{\mathbb{R}}$ be the restriction of $\tau_d$
to $G_p \subset GL(d)$. 
For 
$w \in \mathbb{Z}$, we set 
\begin{align}\label{Minksum}
\mathbf{W}_p(d)_{w}=\frac{1}{2}\mathrm{sum}[0, \beta]
    +w \tau_{d, p} \subset (M_p)_{\mathbb{R}},  
\end{align}
where the Minkowski sum is after all $T_p$-weights 
$\beta$ in the representation (\ref{ext1}). 
Let $\beta_i^{(j)}$ for $1\leq i \leq \dim W^{(j)}$
be the weights of 
the standard representation of $GL(W^{(j)})$. 
Then a weight $\chi$ 
in $\textbf{W}_p(d)_w$ is written as
\begin{align}\label{wt:wpd}
\chi=\sum_{i, j, a, b}c_{ij}^{(ab)}(\beta_i^{(a)}-\beta_j^{(b)})
+\frac{w}{d}\sum_{i, a}r^{(a)}\beta_i^{(a)},
\end{align}
where the sum above is after all $1\leq a, b\leq m$, $1\leq i\leq \dim W^{(a)}$, $1\leq j\leq \dim W^{(b)}$, and
where $\lvert c_{ij}^{(ab)} \rvert \leq \delta_{ab}/2+gr^{(a)}r^{(b)}$ for all such $a,b,i,j$. 

\begin{lemma}\label{lem:anotherT}
Recall the map $j_p \colon \widehat{\mathscr{P}}(d)_p
    \hookrightarrow \widehat{\mathscr{Y}}(d)_p$ from \eqref{mapjp}.
The subcategory introduced in (\ref{qbps:that}):
\begin{align}\label{def:Tformal}
    \mathbb{T}_p(d)_{w} \subset 
    D^b(\widehat{\mathscr{P}}(d)_p)
\end{align}
coincides with the subcategory of objects $\mathcal{E}$
such that 
$j_{p\ast}\mathcal{E}$ is generated by the vector bundles
$\Gamma_{G_p}(\chi) \otimes \mathcal{O}_{\widehat{\mathscr{Y}}(d)_p}$, where 
$\chi$ is a dominant $T_p$-weight 
satisfying 
\begin{align*}\chi+\rho_p \in \mathbf{W}_p(d)_{w},
\end{align*}
where $\rho_p$ is half the sum
of positive roots of $G_p$. 
\end{lemma}
\begin{proof}
    The lemma follows similarly to Lemma~\ref{lem:compareT}, using the Koszul equivalence 
and~\cite[Corollary~3.14]{PTquiver}. 
\end{proof}
\begin{remark}\label{rmk:anotherT}
Alternatively, by Lemma~\ref{lem:anotherT} and~\cite[Lemma~2.9]{hls}, 
the subcategory (\ref{def:Tformal}) consists of objects $\mathcal{E}$ such that 
$j_{p\ast}\mathcal{E}$ is generated by vector bundles
$W \otimes \mathcal{O}_{\widehat{\mathscr{Y}}(d)_p}$ for $W$ a $G_p$-representation 
whose $T_p$-weights are contained in the set
\begin{align}\label{nablap}
   \left\{\chi \in (M_{p})_{\mathbb{R}} : 
   -\frac{1}{2}n_{\lambda, p} \leq 
   \langle \lambda, \chi \rangle \leq \frac{1}{2}
   n_{\lambda, p}\text{ for all }\lambda \colon \mathbb{C}^{\ast} \to T_p
    \right\}+w\tau_{d, p}.
\end{align}
Here, the width $n_{\lambda, p}$ is defined by 
\begin{align*}
    n_{\lambda, p}=\Big\langle \lambda, 
    \det\left(\Ext_{Q^{\circ, d}}^1(R_p, R_p)^{\vee}\oplus 
    \mathfrak{gl}(d)\right)^{\lambda>0}\Big\rangle-\Big\langle \lambda, 
    \det\left((\mathfrak{g}_p^{\vee})^{\lambda>0}\right) \Big\rangle,
\end{align*}
where $\mathfrak{g}_p$ is the Lie algebra of $G_p$. 
From (\ref{ext1}), one can easily check that 
\begin{align}
    \label{eta:equal}
n_{\lambda, p}=2g \Big\langle \lambda, \det\left(\mathfrak{gl}(d)^{\lambda>0}\right)
\Big\rangle =n_{\lambda}
\end{align}
for any cocharacter $\lambda \colon 
\mathbb{C}^{\ast} \to T_p \subset T(d)$. 
\end{remark}

\subsection{Proof of Proposition~\ref{prop:sod2}}\label{subsec:prop}
In this subsection, we prove Proposition~\ref{prop:sod2} and Corollary~\ref{cor:gen}, and thus finish the proof of Theorem \ref{thm:sodK3}. 
\begin{proof}[Proof of Proposition~\ref{prop:sod2} and Corollary~\ref{cor:gen}]
Let $\iota_p \colon \widehat{\mathscr{P}}(d)_p \to \mathscr{P}(d)$
be the natural induced map and define 
$\widehat{\mathbb{T}}_p(d)_w$ to be the 
subcategory of 
$D^b(\widehat{\mathscr{P}}(d)_p)$
classically generated by the image of 
\begin{align*}
    \iota_p^{\ast} \colon \mathbb{T}(d)_w \to 
  D^b(\widehat{\mathscr{P}}(d)_p).   
\end{align*}
    By Theorem~\ref{cor:sodT} and 
    Lemma~\ref{lem:basechange}, we have the semiorthogonal 
    decomposition 
 \begin{align}\label{sod:hat}
    D^b(\widehat{\mathscr{P}}(d)_p)
    =\left\langle \bigoplus_{p_1+\cdots+p_k=p}
    \boxtimes_{i=1}^k \widehat{\mathbb{T}}_{p_i}(d_i)_{w_i+(g-1)d_i(\sum_{i>j}d_j-\sum_{i<j}d_j)}
    \right\rangle. 
    \end{align}
    Therefore it is enough to show that \begin{equation}\label{tpdw}
\widehat{\mathbb{T}}_p(d)_w=\mathbb{T}_p(d)_w, \end{equation}
which is the claim of Corollary \ref{cor:gen}. 

Let $\widehat{\mathscr{X}}(d)_p$ be the formal fiber at $p$ of 
the composition 
\begin{align*}\mathscr{X}(d) \to \mathscr{Y}(d)
\to Y(d)=\mathfrak{gl}(d)^{\oplus 2g}\ssslash GL(d),
\end{align*}
 where the first morphism is the natural
projection. 
It is given by 
\begin{align*}
    \widehat{\mathscr{X}}(d)_p=(\widehat{\Ext}^1_{Q^{\circ, d}}(R_p, R_p)\times \mathfrak{gl}(d)^{\vee})/G_p. 
\end{align*}
We have the Koszul duality equivalence, see Theorem~\ref{thm:Kduality}
\begin{align}\label{Kdualp}
\Theta_p \colon 
    D^b(\widehat{\mathscr{P}}(d)_p)
    \stackrel{\sim}{\to} 
    \mathrm{MF}^{\rm{gr}}(\widehat{\mathscr{X}}(d)_p, \Tr W). 
\end{align}
We next define categories Koszul equivalent to the two categories in \eqref{tpdw}:
\begin{align*}
    \widehat{\mathbb{S}}^{\rm{gr}}_p(d)_w \subset 
   \mathrm{MF}^{\rm{gr}}(\widehat{\mathscr{X}}(d)_p, \Tr W), \ 
    \mathbb{S}^{\rm{gr}}_p(d)_w \subset 
   \mathrm{MF}^{\rm{gr}}(\widehat{\mathscr{X}}(d)_p, \Tr W).
\end{align*}
We define the subcategory  
$\widehat{\mathbb{S}}^{\rm{gr}}_p(d)_w$
to be classically 
generated by the image of 
\begin{align}\notag
\mathbb{S}^{\rm{gr}}(d)_w \subset 
    \mathrm{MF}^{\rm{gr}}(\mathscr{X}(d), \Tr W)
    \to \mathrm{MF}^{\rm{gr}}(\widehat{\mathscr{X}}(d)_p, \Tr W). 
\end{align}
We define the subcategory $\mathbb{S}^{\rm{gr}}_p(d)_w$  to be consisting of 
matrix factorizations 
whose factors are of the form $W \otimes \mathcal{O}_{\widehat{\mathscr{X}}(d)_p}$,
where $W$ is a $G_p$-representation whose $T_p$-weights 
are contained in (\ref{nablap}). 
By the equivalence (\ref{theta:rest}) and using Lemma~\ref{lem:genJ} and
Remark~\ref{rmk:anotherT}, 
the equivalence $\Theta_p$ restricts to equivalences 
\begin{align*}
    \Theta_p \colon 
    \widehat{\mathbb{T}}_p(d)_w \stackrel{\sim}{\to} 
    \widehat{\mathbb{S}}^{\rm{gr}}_p(d)_w, \ 
    \mathbb{T}_p(d)_w \stackrel{\sim}{\to} 
    \mathbb{S}^{\rm{gr}}_p(d)_w. 
\end{align*}
It is enough to show that 
$\widehat{\mathbb{S}}^{\rm{gr}}_p(d)_w=\mathbb{S}^{\rm{gr}}_p(d)_w$. 
By Remark~\ref{rmk:anotherT},  
it is obvious that 
$\widehat{\mathbb{T}}_p(d)_w \subset \mathbb{T}_p(d)_w$, hence 
$\widehat{\mathbb{S}}^{\rm{gr}}_p(d)_w \subset 
\mathbb{S}^{\rm{gr}}_p(d)_w$. 

By the semiorthogonal decomposition (\ref{sod:hat})
together with the equivalence (\ref{Kdualp}), 
we have the semiorthogonal decomposition 
\begin{align}\label{sod:hat2}
 \mathrm{MF}^{\rm{gr}}(\widehat{\mathscr{X}}(d)_p, \Tr W)
    =\left\langle \bigoplus_{p_1+\cdots+p_k=p}
    \boxtimes_{i=1}^k \widehat{\mathbb{S}}^{\rm{gr}}_{p_i}(d_i)_{w_i+g d_i(\sum_{i>j}d_j-\sum_{i<j}d_j)}
    \right\rangle
    \end{align}
    for $w_1/d_1<\cdots<w_k/d_k$, 
    and each summand is given by
    the categorical Hall product, see~\cite[Proposition~3.1]{P2}
or \cite[Lemma~2.4.4, 2.4.7]{T} for the compatibility of the categorical Hall products
    under Koszul duality. 
In Lemma~\ref{lem:orthoS} below, we show that the semiorthogonal summands in (\ref{sod:hat2})
except $\widehat{\mathbb{S}}^{\rm{gr}}_p(d)_w$
are right orthogonal to $\mathbb{S}^{\rm{gr}}_p(d)_w$. 
Then
by (\ref{sod:hat2}) we
have $\mathbb{S}^{\rm{gr}}_p(d)_w \subset \widehat{\mathbb{S}}^{\rm{gr}}_p(d)_w$, 
hence that $\widehat{\mathbb{S}}^{\rm{gr}}_p(d)_w=\mathbb{S}^{\rm{gr}}_p(d)_w$. 
 \end{proof}
\begin{lemma}\label{lem:orthoS}
The semiorthogonal summands in (\ref{sod:hat2}) 
with $k\geq 2$
are right orthogonal to $\mathbb{S}^{\rm{gr}}_p(d)_w$. 
\end{lemma}
\begin{proof}
The proof is analogous to that of ~\cite[Lemma~3.6]{PT1}. 
The inclusion $T_p \subset T(d)$ induces a surjection 
$M(d) \twoheadrightarrow M_p$. We will regard $T(d)$-weights 
as $T_p$-weights by the above surjection.
Let $\widehat{\textbf{W}}_p(d)_w$ be the image of 
$\textbf{W}(d)_w \subset M(d)_{\mathbb{R}} \twoheadrightarrow 
(M_{p})_{\mathbb{R}}$. 
Recall the decomposition \eqref{Rp2} and the weights $\beta^{(a)}_i$ for $1\leq a\leq m$ and $1\leq i\leq\dim W^{(a)}$.
Then a weight $\chi$ in $\widehat{\textbf{W}}_p(d)_w$ is written 
as 
\begin{align}\label{write:chi}
    \chi=\sum_{i, j, a, b}\alpha_{ij}^{(ab)}(\beta_i^{(a)}-\beta_j^{(b)})
    +\frac{w}{d}\sum_{i, a}r^{(a)}\beta_i^{(a)},
\end{align}
where the sum above is after all $1\leq a, b\leq m$, $1\leq i\leq \dim W^{(a)}$, $1\leq j\leq \dim W^{(b)}$, 
and we have that $\lvert \alpha_{ij}^{(ab)} \rvert \leq
r^{(a)}r^{(b)}(g+1/2)$.
We also note that a choice of $(p_1, \ldots, p_k)$
corresponds to decompositions for all $1\leq j\leq m$:
\begin{align*}
    W^{(j)}=W_1^{(j)} \oplus \cdots \oplus W_k^{(j)}
\end{align*}
such that $d_i^{(j)}=\dim W_i^{(j)}$ satisfies 
$d_i=d_i^{(1)}+\cdots+d_i^{(m)}$. 

Let $\lambda$ be the antidominant cocharacter
of $T_p$ which acts on the space $W_i^{(j)}$ by weight $(k+1-i)$ for $1\leq j\leq m$ and $1\leq i\leq k$, 
and write it as $\lambda=(\lambda^{(j)})_{1\leq j\leq m}$, 
where $\lambda^{(j)}$ is a cocharacter of the maximal torus of $GL(W^{(j)})$. 
We set $\mathfrak{g}^{(j)}=\mathrm{End}(W^{(j)})$. 
Consider the diagram of attracting loci 
\begin{align*}
    \widehat{\mathscr{X}}(d)_p^{\lambda}=\times_{i=1}^k
    \widehat{\mathscr{X}}(d_i)_{p_i} \stackrel{q}{\leftarrow}
    \widehat{\mathscr{X}}(d)_p^{\lambda \geq 0}
    \stackrel{p}{\to} 
    \widehat{\mathscr{X}}(d)_p. 
\end{align*}
Let $A=\Gamma_{GL(d)}(\chi) \otimes \mathcal{O}_{\widehat{\mathscr{X}}(d)_p}$
and $B=\Gamma_{GL(d)^{\lambda}}(\chi') \otimes
\mathcal{O}_{\mathscr{X}(d)_p^{\lambda}}$
such that 
\begin{align}\label{condition:chiW}
    \chi+\rho_p \in \textbf{W}_p(d)_w, \ 
    \chi'+\sum_{i=1}^k \rho_{p_i} \in 
    \bigoplus_{i=1}^k \widehat{\textbf{W}}_{p_i}(d_i)_{w_i'}\subset \bigoplus_{i=1}^k M(d_i)_\mathbb{R}=M(d)_\mathbb{R},
\end{align}
where $w=w_1+\cdots+w_k$, 
$w_1/d_1<\cdots<w_k/d_k$
and $w_i'=w_i+gd_i(\sum_{i>j}d_j-\sum_{j>i}d_j)$. 
We write 
\begin{equation}\label{psidecomp}
    \chi'=\sum_{i=1}^k (\psi_i+w'_i\tau_{d_i}),\, \psi_i\in \widehat{\textbf{W}}_{p_i}(d_i)_0.
\end{equation}
By the adjunction, we have 
\begin{align}\label{hom:AB}
    \Hom(A, p_{\ast}q^{\ast}B)=\Hom(p^{\ast}A, q^{\ast}B). 
\end{align}
Let $\chi''$ be a weight of $\Gamma_{GL(d)}(\chi)$. 
Below we show that 
\begin{equation}\label{inequality}
    \langle \lambda, \chi'' \rangle > \langle \lambda, \chi'\rangle.
\end{equation} 
Then (\ref{hom:AB}) vanishes by~\cite[Proposition~4.2]{P} and thus the lemma holds. 
Let $\mu$ be the weight:
    \begin{align}\label{wt:mu}
    \mu=-\frac{1}{2}
    \mathfrak{gl}(d)^{\lambda>0}+\frac{1}{2}\sum_{a=1}^m(\mathfrak{g}^{(a)})^{\lambda^{(a)}>0}=
    \sum_{i, j, a<b}\gamma_{ij}^{(ab)}(\beta_i^{(a)}-\beta_j^{(b)})
\end{align}
where $1\leq a<b\leq m$, $1\leq i\leq \dim W^{(a)}$, $1\leq j\leq \dim W^{(b)}$, and such that $\lvert \gamma_{ij}^{(ab)}\rvert=r^{(a)}r^{(b)}/2$. 
To show \eqref{inequality}, it is enough to show that 
\begin{align}\label{ineq:chirho}
\langle \lambda, \chi''+\rho_p+\mu-w\tau_d\rangle 
>\langle \lambda,  \chi'+\rho_p+\mu-w\tau_d \rangle.     
\end{align}
 By \eqref{psidecomp}, we write 
 \begin{align*}
    \chi'+\rho_p+\mu-w\tau_d
    =\sum_{i=1}^k \psi_i+\sum_{i=1}^k w_i \tau_{d_i}
    -\frac{2g+1}{2} \mathfrak{gl}(d)^{\lambda>0}-w\tau_d,
 \end{align*}
 where 
 $\psi_i \in \widehat{\textbf{W}}_{p_i}(d_i)_{0}$ for $1\leq i\leq k$. 
 In what follows, we write $\mathfrak{gl}(d)^{\lambda>0}$ instead of $\det\left(\mathfrak{gl}(d)^{\lambda>0}\right)$ to simplify notation.
We compute 
\begin{align*}
\left\langle \lambda, 	\chi'+\rho_p+\mu-w\tau_d \right\rangle	&=\left\langle \lambda, 	\sum_{i=1}^k \psi_i+\sum_{i=1}^k w_i \tau_{d_i}-\frac{2g+1}{2}
 \mathfrak{gl}(d)^{\lambda>0}
	-w \tau_d \right\rangle \\
	&=\sum_{i=1}^k (k+1-i)d_i\left(\frac{w_i}{d_i}-\frac{w}{d}\right)
	-\left\langle \lambda, \frac{2g+1}{2}\mathfrak{gl}(d)^{\lambda>0} \right\rangle.
	\end{align*}
	For $1\leq i\leq k$, define \[\tilde{w}_i:=d_i\left(\frac{w_i}{d_i}-\frac{w}{d}\right).\] 
	Then $\tilde{w}_1+\cdots+\tilde{w}_k=0$
	and $\tilde{w}_1+\cdots+\tilde{w}_l<0$ for $1\leq l<k$.
	Therefore 
	\begin{align*}
	   \sum_{i=1}^k (k+1-i)d_i\left(\frac{w_i}{d_i}-\frac{w}{d}\right)
	    =\sum_{l=1}^k \left(\sum_{i=1}^l \tilde{w}_i \right)<0. 	\end{align*}
	It follows that 
	\begin{align}\label{lambda1}
	    \left\langle \lambda, 	\chi'+\rho_p+\mu-w\tau_d \right\rangle <-\left\langle \lambda, \frac{2g+1}{2}\mathfrak{gl}(d)^{\lambda>0} \right\rangle.
	\end{align}
On the other hand,
by (\ref{condition:chiW}) and~\cite[Lemma~2.9]{hls},
we have 
\begin{align*}
    \langle \lambda, \chi'' -w\tau_d \rangle \geq -\frac{1}{2}n_{\lambda, p}=-g \langle \lambda, \mathfrak{gl}(d)^{\lambda>0} \rangle. 
\end{align*}
Then 
\begin{align*}
    \langle \lambda, \chi''+\rho_p+\mu-w\tau_d \rangle \geq 
    -g \langle \lambda, \mathfrak{gl}(d)^{\lambda>0}\rangle+\langle \lambda, 
    \rho_p+\mu \rangle=-\left\langle \lambda, \frac{2g+1}{2}\mathfrak{gl}(d)^{\lambda>0} \right\rangle.
\end{align*}
Therefore we have the inequality (\ref{ineq:chirho}). 
\end{proof}

\section{Smooth and properness of reduced quasi-BPS categories}
In this section, 
we show that the reduced version of quasi-BPS category 
is smooth and proper, which gives evidence 
towards Conjecture~\ref{conj:HK}. 
We first prove the strong generation 
of quasi-BPS categories. 
It relies on the strong generations of 
singular support quotients, which itself is 
of independent interest and is proved in Subsection~\ref{subsec:ssuport:gen}. 

\subsection{Strong generation of quasi-BPS categories}
In this subsection, we prove the strong generation of the quasi-BPS 
category $\mathbb{T}_S(v)_{w}$, 
see Subsection~\ref{notation} for the terminology of strong generation. 
The strategy is to show that $\mathbb{T}_S(v)_{w}$ is admissible in a singular support quotient category constructed from Joyce-Song pairs on the local Calabi-Tau threefold $X:=S\times \mathbb{C}$, which has a strong generator by Theorem \ref{thm:regular2}.

Let $S$ be a smooth projective K3 surface, let $H$ be an ample divisor on $S$, and set 
 $\mathcal{O}(n)=\mathcal{O}_S(nH)$. 
For $v \in N(S)$, let 
$\mathfrak{M}=\mathfrak{M}_S^H(v)$ be the derived 
moduli stack of $H$-Gieseker semistable sheaves $F$ 
on $S$ with numerical class $v$. We take $H$ generic 
with respect to $v$.
Let $n \gg 0$ be such that 
$H^i(F(n))=0$ for all $i>0$ and all $H$-Gieseker semistable sheaves $F$
with numerical class $v$.
Let $\mathbb{F} \in D^b(S \times \mathfrak{M})$
be the universal sheaf, and consider the 
following derived stack 
\begin{align*}
    \mathfrak{M}^{\dag}:=
    \Spec_{\mathfrak{M}} \mathrm{Sym}
    (p_{\mathfrak{M}\ast}(\mathbb{F}\boxtimes \mathcal{O}(n))^{\vee}), 
\end{align*}
where $p_{\mathfrak{M}} \colon S \times \mathfrak{M} \to \mathfrak{M}$ is the projection. 
The stack $\mathfrak{M}^{\dag}$ 
is the derived moduli stack of pairs $(F, s)$, 
where $F$ is an $H$-Gieseker semistable sheaf on $S$ with 
numerical class $v$ and $s \in H^0(F(n))$. 

We consider its $(-1)$-shifted cotangent space
\begin{align*}
\Omega_{\mathfrak{M}^{\dag}}[-1]=\Spec_{\mathfrak{M}^{\dag}}
\mathrm{Sym}(\mathbb{T}_{\mathfrak{M}^{\dag}}[1]). 
\end{align*}
Since the projection 
$\mathfrak{M}^{\dag} \to \mathfrak{M}$
is smooth, we have the isomorphism, 
see~\cite[Lemma~3.1.2]{T}: 
\begin{align*}
(\Omega_{\mathfrak{M}}[-1]\times_{\mathfrak{M}} \mathfrak{M}^{\dag})^{\rm{cl}}
\stackrel{\cong}{\to} \Omega_{\mathfrak{M}^{\dag}}[-1]^{\rm{cl}}. 
\end{align*}
Therefore, $\Omega_{\mathfrak{M}^{\dag}}[-1]$ is 
the derived moduli 
stack of pairs $(E, s)$, where 
$E$ is a compactly supported coherent 
sheaf on 
the local K3 surface 
\begin{align*}
X:=\mathrm{Tot}(\omega_S)=S \times \mathbb{C}
\stackrel{r}{\to} S
\end{align*}
such that $r_{\ast}E$ has numerical class $v$, 
and $s \in H^0(E(n))$. 
Here the pull-back of $\mathcal{O}(n)$ on $S$ to $X$ 
is also denoted by $\mathcal{O}(n)$. 
We recall the definition of Joyce-Song (JS) stable 
pairs on $X$: 
\begin{defn}(\cite[Definition~5.20]{JS})
A pair $(E, s)$ on $X=S \times \mathbb{C}$ is 
JS-stable if $E$ is a compactly supported $H$-Gieseker
semistable sheaf on $X$ and $s \in H^0(E(n))$ is a section
such that there is no non-trivial 
exact sequence of framed sheaves 
\begin{align}\label{ex:framed}
0 \to (\mathcal{O}_X \to E'(n)) \to 
(\mathcal{O}_X \stackrel{s}{\to} E(n))
\to (0 \to E''(n)) \to 0,
\end{align}
where $E'$, $E''$ are $H$-Gieseker semistable 
sheaves with the same reduced Hilbert polynomials. 
\end{defn}
We denote 
by \begin{align*}
\Omega_{\mathfrak{M}^{\dag}}^{\rm{JS}}[-1]
\subset \Omega_{\mathfrak{M}^{\dag}}[-1]
    \end{align*}
    the open substack 
    consisting of JS-stable pairs, 
    and we denote by $\mathscr{Z}^{\rm{JS}}$ its complement.
    It is well-known that $\Omega_{\mathfrak{M}^{\dag}}^{\rm{JS}}[-1]^{\rm{cl}}$ is a quasi-projective scheme, 
    which easily follows from~\cite[Theorem~5.22]{JS}
    by taking a compactification of $X$. 
We set
\begin{align*}
\ell :=\det p_{\mathfrak{M}\ast}(\mathbb{F}\boxtimes \mathcal{O}(n)) \in \mathrm{Pic}(\mathfrak{M}).
\end{align*}
Its pull-back to $\Omega_{\mathfrak{M}^{\dag}}[-1]$ 
is also denoted by $\ell$. 
We denote by $\Omega^{\ell\text{-ss}}_{\mathfrak{M}^{\dag}}[-1]$ the stack of $\ell$-semistable points in $\Omega_{\mathfrak{M}^{\dag}}[-1]^{\mathrm{cl}}$.
\begin{lemma}\label{lem:JS}
We have $\Omega^{\rm{JS}}_{\mathfrak{M}^{\dag}}[-1]=
\Omega^{\ell\mathrm{-ss}}_{\mathfrak{M}^{\dag}}[-1]$. 
    \end{lemma}
\begin{proof}
Let $\mathfrak{M}^{\rm{cl}} \to M$ be a 
good moduli space. 
    It is enough to prove the identity on 
    each fiber at a closed point $y \in M$
    for the composition of the projections
    \begin{align}\label{compose2}
    \gamma \colon 
    \Omega_{\mathfrak{M}^{\dag}}[-1]^{\rm{cl}} \to 
    \mathfrak{M}^{\dag, \rm{cl}} \to \mathfrak{M}^{\rm{cl}} \to M.         
    \end{align}
    A point $y$ corresponds to a polystable sheaf 
    $\bigoplus_{i=1}^m V^{(i)} \otimes F^{(i)}$. Let 
    $(Q^{\circ, d}_y, \mathscr{I}_y)$ be the 
    Ext-quiver of $(F^{(1)}, \ldots, F^{(m)})$
    with relation $\mathscr{I}_y$. The quiver $Q^{\circ, d}_y$ is the double of some quiver $Q_y^{\circ}$, see Remark~\ref{rmk:Equiver}. 
   Let $(Q_y, W)$
    be the tripled quiver with potential of $Q_y^{\circ}$, 
    see Subsection~\ref{subsec:triple}. 
    Let $c^{(i)}:=h^0(F^{(i)}(n))>0$
    and let $Q_y^{\dag}$ be the 
    quiver obtained by adding a vertex $\{0\}$ to $Q_y$ 
    and
    $c^{(i)}$-arrows from $0$ to $i$ for $1\leq i\leq m$. 
    Then a fiber of (\ref{compose2})
    at $y$ corresponds to nilpotent $Q_y^{\dag}$-representations with dimension vector $(1, \bm{d})$
    where $\bm{d}=(\dim V^{(i)})_{i=1}^m$ and $1$ is 
    the dimension at the vertex $\{0\}$:
    \begin{align}\label{isom:gamma}
\gamma^{-1}(y) \cong R^{\rm{nil}}_{Q_y^{\dag}}(1, \bm{d})/G(\bm{d}). 
    \end{align}
    Also the line bundle $\ell$
    restricted to $\gamma^{-1}(y)$ corresponds to the character
    \begin{align*}    
     \ell_y \colon G(\bm{d})=\prod_{i=1}^m GL(V^{(i)}) \to \mathbb{C}^{\ast}, \ (g_i)_{i=1}^m \mapsto 
    \prod_{i=1}^m (\det g_i)^{c^{(i)}}.
    \end{align*}
    By~\cite[Lemma~5.1.9, 5.1.19]{T}, the $\ell_y$-semistable 
    $Q_y^{\dag}$-representations 
    are those 
    generated by the images from the arrows 
    $0 \to i$ with $1\leq i \leq m$. 
   The above $\ell_y$-semistable locus in the right hand side of (\ref{isom:gamma})
   corresponds to 
    pairs $(E, s)$ on $X$ in $\gamma^{-1}(y)$ such that 
$r_{\ast}E$ is S-equivalent to 
$\bigoplus_{i=1}^m V^{(i)} \otimes F^{(i)}$
and there is no exact sequence of
the form (\ref{ex:framed}), i.e. it is a JS pair. Therefore we obtain the desired identity on $\gamma^{-1}(y)$. 
\end{proof}

 We set 
 \begin{align*}
 b:=\ch_2(p_{\mathfrak{M}\ast}(\mathbb{F} \boxtimes \mathcal{O}(n))) \in H^4(\mathfrak{M}, \mathbb{Q}).
 \end{align*}
Its pull-back to $\Omega_{\mathfrak{M}^{\dag}}[-1]$ is 
also denoted by $b$. 
Consider the $\Theta$-stratification with 
respect to $(\ell, b)$, see~\cite[Theorem~4.1.3]{halpK32}:
\begin{align*}
    \Omega_{\mathfrak{M}^{\dag}}[-1]=
    \mathscr{S}_1 \sqcup \cdots \sqcup \mathscr{S}_N
    \sqcup  \Omega^{\ell\text{-ss}}_{\mathfrak{M}^{\dag}}[-1].
\end{align*}
By Theorem~\ref{thm:window:M}, 
for each choice of $m_\bullet=(m_i)_{i=1}^N \in \mathbb{R}^N$, 
there is a subcategory $\mathbb{W}(\mathfrak{M}^{\dag})_{m_{\bullet}}^\ell
\subset D^b(\mathfrak{M}^{\dag})$ such that the composition 
\begin{align}\label{equiv:WDT}
\Phi \colon 
    \mathbb{W}(\mathfrak{M}^{\dag})_{m_{\bullet}}^\ell 
    \subset  D^b(\mathfrak{M}^{\dag}) \twoheadrightarrow D^b(\mathfrak{M}^{\dag})/\mathcal{C}_{\mathscr{Z}^{\rm{JS}}}
\end{align}
is an equivalence. 
Let $\eta \colon \mathfrak{M}^{\dag} \to \mathfrak{M}$
be the projection. 
We have the following lemma: 
\begin{lemma}\label{lem:eta}
Let $\delta\in\mathrm{Pic}(\mathfrak{M}^H_S(v))_\mathbb{R}$.
There exists a choice $m_\bullet$ such that the 
functor $\eta^{\ast} \colon D^b(\mathfrak{M}) \to 
D^b(\mathfrak{M}^{\dag})$ restricts to a
functor 
$\eta^{\ast} \colon \mathbb{T}_S^H(v)_\delta
    \to \mathbb{W}(\mathfrak{M}^{\dag})_{m_{\bullet}}^\ell$. 
\end{lemma}
\begin{proof}
    We use the notation in the proof of Lemma~\ref{lem:JS}. 
    For $y\in M$, 
    let 
    $\mathscr{X}_y(\bm{d})$ be the 
    moduli stack of $Q_y$-representations 
    with dimension vector $\bm{d}$
and let $\mathscr{X}_y^{\dag}(\bm{d})$ be the moduli 
stack of $Q_y^{\dag}$-representations 
with dimension vector $(1, \bm{d})$. 
Let $\widehat{\mathscr{X}}_y^{\dag}(\bm{d})$ be 
the formal fiber of the composition 
\begin{align*}
    \mathscr{X}^{\dag}_y(\bm{d}) \to \mathscr{X}_y(\bm{d})
    \to X_y(\bm{d})
\end{align*}
at the origin, where the last map is the good moduli 
space morphism. 
     Let
    \begin{align*}
        \widehat{\mathscr{X}}^{\dag}_y(\bm{d})=
        \widehat{\mathscr{S}}_1 \sqcup \cdots 
        \sqcup \widehat{\mathscr{S}}_N \sqcup 
      \widehat{\mathscr{X}}^{\dag}_y(\bm{d})^{\ell_y\text{-ss}}
    \end{align*}
    be the Kempf-Ness stratification with respect
    to $(\ell_y, b_y)$. For $1\leq i\leq N$, consider the center $\widehat{\mathscr{Z}}_i$ of $\widehat{\mathscr{S}}_i$ and its corresponding
    one parameter subgroup $\lambda_i$ for 
    the maximal torus of $G(\bm{d})$.
    Let $\widehat{\mathfrak{M}}_y^{\dag}$, 
    $\widehat{\mathfrak{M}}_y$ be the formal 
    fibers along $\mathfrak{M}^{\dag} \to M$, 
    $\mathfrak{M}\to M$ at $y$ respectively. 
    We have the commutative diagram 
    \begin{align}\label{com:koszul}
       \xymatrix{
D^b(\widehat{\mathfrak{M}}_y) \ar[r]^-{\Theta_y}_-{\sim} \ar[d]_-{\eta^{\ast}} & 
\mathrm{MF}^{\rm{gr}}(\widehat{\mathscr{X}}_y(\bm{d}), \Tr W) \ar[d]^-{\eta'^{\ast}} \\
D^b(\widehat{\mathfrak{M}}_y^{\dag}) \ar[r]^-{\Theta_y^{\dag}}_-{\sim} & 
\mathrm{MF}^{\rm{gr}}(\widehat{\mathscr{X}}_y^{\dag}(\bm{d}), \Tr W). 
    }        
    \end{align}
    Here the horizontal arrows are Koszul duality 
    equivalences in Theorem~\ref{thm:Kduality}, and the vertical arrows are 
    pull-backs along the natural projections. 
    
    By~\cite[Proposition~6.1]{Totheta}, 
    there exists a choice $m_\bullet=(m_i)_{i=1}^N \in \mathbb{R}^N$ such that 
    an object $\mathcal{E} \in D^b(\mathfrak{M}^{\dag})$
    lies in $\mathbb{W}(\mathfrak{M}^{\dag})_{m_{\bullet}}^\ell$ if and 
    only if, for any $y$ as above, 
    we have 
    \begin{align*}
        \mathrm{wt}_{\lambda_i}
        \Theta_y^{\dag}(\mathcal{E}|_{\widehat{\mathfrak{M}}_y^{\dag}})|_{\widehat{\mathscr{Z}}_i} \subset 
        \left[ -\frac{1}{2}n_i^{\dag}, \frac{1}{2}n_i^{\dag}\right)+\langle \lambda_i, \delta_y\rangle. 
    \end{align*}    
    Here, the width $n_i^{\dag}$ is defined by 
    \begin{align*}
     n_i^{\dag}:=\left\langle \lambda_i, \det\Big(\mathbb{L}^{\lambda_i>0}_{\mathscr{X}_y^{\dag}(\bm{d})}\Big|_{0}\big)\right\rangle 
     =n_i+\sum_{j=1}^m
     c^{(j)} \left\langle \lambda_i, \det\big((V_j^{\vee})^{\lambda_i>0}\big)
     \right\rangle
    \end{align*}
    and
$n_i:=\big\langle \lambda_i, \det\big(\mathbb{L}^{\lambda_i>0}_{\mathscr{X}_y(\bm{d})}|_{0}\big)\big\rangle$. 
On the other hand, by the definition of
$\mathbb{T}_S^H(v)_\delta$, for an object 
$A \in \mathbb{T}_S^H(v)_\delta$,
the $\lambda_i$-weights of 
$\Theta_y(A|_{\widehat{\mathfrak{M}}_y})|_{\widehat{\mathscr{X}}_y(\bm{d})^{\lambda_i}}$
lie in $[-n_i/2, n_i/2]+\langle \lambda_i, \delta_y\rangle$
for all $1\leq i\leq N$. 
As in~\cite[Lemma~5.1.9]{T}, 
each $\lambda_i$ has only non-positive
weights in each $V^{(j)}$ for $1\leq j\leq m$, hence 
we have $n_i^{\dag}>n_i$. 
From the diagram (\ref{com:koszul}), 
we have 
\[\Theta_y^{\dag}((\eta^{\ast}A)|_{\widehat{\mathfrak{M}}_y^{\dag}})
\cong \eta'^{\ast}\Theta_y(A|_{\mathfrak{M}_y}),\]
hence its restriction to $\widehat{\mathscr{Z}}_i$
has $\lambda_i$-weights
in $[-n_i^{\dag}/2, n_i^{\dag}/2)+\langle \lambda_i, \delta_y\rangle$. 
Therefore we have 
$\eta^{\ast}A \in \mathbb{W}(\mathfrak{M}^{\dag})_{m_{\bullet}}^\ell$. 
\end{proof}

We prove the following theorem, using the strong generation
of singular support quotients in Theorem~\ref{thm:regular2} 
which will be proved in Subsection~\ref{subsec:ssuport:gen}:
\begin{thm}\label{thm:regular}
The quasi-BPS category 
$\mathbb{T}_S(v)_w$
is regular. 
\end{thm}
\begin{proof}
 By Corollary~\ref{cor:lv}, it is enough 
 to show that $\mathbb{T}_S^H(v)_{w} \subset 
 D^b(\mathfrak{M}_S^H(v))$ is regular. 
We consider the following composition 
\begin{align*}
F \colon \mathbb{T}_S^H(v)_w \stackrel{i}{\hookrightarrow} 
D^b(\mathfrak{M})_w \stackrel{\eta^{\ast}}{\to} 
D^b(\mathfrak{M}^{\dag}) \stackrel{p}{\twoheadrightarrow} 
D^b(\mathfrak{M}^{\dag})
/\mathcal{C}_{\mathscr{Z}^{\rm{JS}}}.
\end{align*}
Let $\Phi$ be window equivalence \eqref{equiv:WDT} as in Lemma \ref{lem:eta}, and let $\Phi^{-1}$ be its inverse. Let $\Psi \colon D^b(\mathfrak{M})_w \twoheadrightarrow \mathbb{T}_S^H(v)_w$ be the projection 
with respect to the semiorthogonal decomposition 
in Theorem~\ref{thm:sodK3}.
We also define the following functor 
\begin{align*}
    G \colon     D^b(\mathfrak{M}^{\dag})/\mathcal{C}_{\mathscr{Z}^{\rm{JS}}} \stackrel{\Phi^{-1}}{\to}  \mathbb{W}(\mathfrak{M}^{\dag})_{m_{\bullet}}^\ell
    \stackrel{j}{\hookrightarrow} D^b(\mathfrak{M}^{\dag})
    \stackrel{(\eta_{\ast})_w}{\twoheadrightarrow} 
    D^b(\mathfrak{M})_w \stackrel{\Psi}{\twoheadrightarrow} 
    \mathbb{T}_S^H(v)_{w}. 
\end{align*}
Here,
$(\eta_{\ast})_w(-)$ is the weight $w$-part of 
$\eta_{\ast}(-)$, which is the projection 
onto $D^b(\mathfrak{M})_w$
with respect to the semiorthogonal decomposition 
\begin{align}\label{sod:wt}
    D^b(\mathfrak{M}^{\dag})
    =\langle \ldots, D^b(\mathfrak{M})_{-1}, D^b(\mathfrak{M})_0, D^b(\mathfrak{M})_1, \ldots \rangle. 
\end{align}
Every fully-faithful functor in (\ref{sod:wt})
is given by 
the restriction of $\eta^{\ast}$ to $D^b(\mathfrak{M})_w$. 
The above semiorthogonal decomposition exists 
since $\eta \colon \mathfrak{M}^{\dag} \to \mathfrak{M}$
is an affine space bundle 
such that the cone of $\mathcal{O}_{\mathfrak{M}} \to 
\eta_{\ast}\mathcal{O}_{\mathfrak{M}^{\dag}}$
has strictly negative $\mathbb{C}^{\ast}$-weights, 
see~\cite[Amplification~3.18]{halp}. 

Then $G \circ F \cong \id$. 
Indeed, we have 
\begin{align*}
\Psi \circ (\eta_{\ast})_w \circ \Phi^{-1} \circ p \circ\eta^{\ast} \circ i \cong 
\Psi \circ (\eta_{\ast})_w \circ \eta^{\ast} \circ i
\cong \Psi \circ i \cong \id. 
\end{align*}
For the first isomorphism, the image of $\eta^{\ast} \circ i$
lies in $\mathbb{W}(\mathfrak{M}^{\dag})_{m_{\bullet}}^\ell$ by 
Lemma~\ref{lem:eta}
and then $\Phi^{-1}\circ p$ is identity on 
$\mathbb{W}(\mathfrak{M}^{\dag})_{m_{\bullet}}^\ell$ by the definition of $\Phi$. 
The second isomorphism follows
since $(\eta_{\ast})_w \circ \eta^{\ast}\cong \id$. 
The last isomorphism also holds by the definition 
of $\Psi$. By Theorem~\ref{thm:regular2}
together with the fact that $\Omega^{\rm{JS}}_{\mathfrak{M}^{\dag}}[-1]$
is a quasi-projective scheme, 
the category $D^b(\mathfrak{M}^{\dag})/\mathcal{C}_{\mathscr{Z}^{\rm{JS}}}$ is regular, so 
it is $\langle \mathcal{E} \rangle^{\star n}$ for some 
$\mathcal{E} \in D^b(\mathfrak{M}^{\dag})/\mathcal{C}_{\mathscr{Z}^{\rm{JS}}}$ and $n\geq 1$. Then
as $\mathrm{Im}(F) \subset \langle \mathcal{E} \rangle^{\star n}$ and 
$G \circ F \cong \id$, 
we conclude that $\mathbb{T}_S^H(v)_{w}=\langle 
G(\mathcal{E}) \rangle^{\star n}$, hence 
$\mathbb{T}_S^H(v)_{w}$ is regular. 
\end{proof}

By an analogous argument using window categories of the reduced stack  $\mathfrak{M}^{\dag, \mathrm{red}}$, we obtain:
\begin{thm}\label{thm:regularred}
    The reduced quasi-BPS category $\mathbb{T}_S(v)_w^{\mathrm{red}}$ is regular.
\end{thm}

\subsection{Properness of reduced quasi-BPS categories}\label{subsec62}
Recall that we write $v=dv_0$ for $d\in\mathbb{Z}_{\geq 1}$ and for $v_0$ primitive 
with $\langle v_0, v_0 \rangle=2g-2$. 
Let $\mathfrak{M}^{\rm{red}}=\mathfrak{M}_S^{\sigma}(v)^{\rm{red}}$
for a generic $\sigma \in \mathrm{Stab}(S)$. 
We consider its $(-1)$-shifted cotangent space:
\begin{align*}
\Omega_{\mathfrak{M}^{\rm{red}}}[-1] \to \mathfrak{M}^{\rm{red}}.
\end{align*}
Its classical truncation is identified with the moduli 
stack of pairs 
\begin{align}\label{F:pairs}
    (F, \theta), \ F \in \mathcal{M}_S^{\sigma}(v), \ \theta \colon F \to F
\end{align}
such that $\mathrm{tr}(\theta)=0$, see~\cite[Lemma~3.4.1]{T} for the non-reduced 
case and the proof for the reduced case is similar. Let 
\begin{align*}
\mathcal{N}_{\rm{nil}} \subset \Omega_{\mathfrak{M}^{\rm{red}}}[-1]
\end{align*}
be the closed substack consisting of pairs (\ref{F:pairs}) such that 
$\theta$ is nilpotent. The following is the
global version of the categorical support lemma. 
\begin{thm}\label{prop:catsupp}
Let $w\in \mathbb{Z}$ be coprime with $d$ and let 
$\mathcal{E} \in \mathbb{T}_S^{\sigma}(v)_w^{\rm{red}} \subset D^b(\mathfrak{M}_S^{\sigma}(v)^{\rm{red}})$. 
Then $\mathrm{Supp}^{\rm{sg}}(\mathcal{E}) \subset \mathcal{N}_{\rm{nil}}$. 
    \end{thm}
\begin{proof}
It is enough to prove the inclusion 
$\mathrm{Supp}^{\rm{sg}}(\mathcal{E}) \subset \mathcal{N}_{\rm{nil}}$
over any point $y \in M_S^{\sigma}(v)$. 
For simplicity, we write 
$\widehat{\mathfrak{M}}^{\rm{red}}_y=\widehat{\mathfrak{M}}_S^{\sigma}(v)_y^{\rm{red}}$.
The equivalence in Lemma~\ref{lem:py} induces the 
isomorphism of classical 
truncations of $(-1)$-shifted cotangents, 
\begin{align}\label{isom:hat}
\Omega_{\widehat{\mathfrak{M}}^{\rm{red}}_y}[-1]^{\rm{cl}} \stackrel{\cong}{\to}
\Omega_{\widehat{\mathscr{P}}(d)_p^{\rm{red}}}[-1]^{\rm{cl}}.
\end{align}
The right hand side is the critical locus 
of the function 
\begin{align*}
\Tr W \colon 
    \widehat{\mathfrak{gl}}(d)^{\oplus 2g}_p
    \times \mathfrak{gl}(d)_0 \to \mathbb{C}
\end{align*}
where $\Tr W$ is the function (\ref{TrW:X}) 
associated with the tripled quiver of 
the $g$-loop quiver, see Subsection~\ref{subsec:shifted}. 
Then the isomorphism (\ref{isom:hat}) restricts to the 
isomorphism 
\begin{align*}
\mathcal{N}_{\rm{nil}} \times_{\mathfrak{M}^{\rm{red}}}
\widehat{\mathfrak{M}}^{\rm{red}}_y \stackrel{\cong}{\to} 
\mathrm{Crit}(\Tr W) \cap (\widehat{\mathfrak{gl}}(d)^{\oplus 2g}_p
    \times \mathfrak{gl}(d)_{\rm{nil}}).
\end{align*}
Therefore 
the theorem follows from Lemma~\ref{cor:support}.     
\end{proof}

Recall that a pre-triangulated 
category $\mathcal{D}$ over $\mathbb{C}$ is 
called \textit{proper} if for any 
$\mathcal{E}_1, \mathcal{E}_2 \in \mathcal{D}$, 
the vector space
$\bigoplus_{i\in \mathbb{Z}} \Hom^{\ast}(\mathcal{E}_1, \mathcal{E}_2)$ is finite dimensional. 
We also have the following global analogue of 
Proposition~\ref{lem:bound}: 

\begin{thm}\label{thm:proper}
If $(d, w)\in\mathbb{N}\times\mathbb{Z}$ are coprime and $g\geq 2$, the 
category $\mathbb{T}_S(v)_w^{\rm{red}}$
is proper. 
\end{thm}
\begin{proof}
We regard $\mathbb{T}_S(v)_w^{\rm{red}}$ as a subcategory of 
$D^b(\mathfrak{M}_S^{\sigma}(v)^{\rm{red}})$ for a generic $\sigma\in\mathrm{Stab}(S)$
via $\mathbb{T}_S(v)_w^{\rm{red}}=\mathbb{T}_S^{\sigma}(v)_{w}^{\rm{red}}$. 
For $\mathcal{E}_1, \mathcal{E}_2 \in \mathbb{T}_S^{\sigma}(v)_w^{\rm{red}}$, let  
    \begin{align*}\mathcal{H} om(\mathcal{E}_1, \mathcal{E}_2) \in 
    D_{\rm{qc}}(\mathfrak{M}_S^{\sigma}(v)^{\rm{red}})
    \end{align*}
    be the internal homomorphism, see~Subsection~\ref{subsec:qsmooth}. 
    Recall that $\mathfrak{M}_S^{\sigma}(v)^{\rm{red}}=\mathcal{M}_S^{\sigma}(v)$
    by Lemma~\ref{lem:classical}. 
    Let $\pi$ be the good moduli space morphism from (\ref{gmoduli}). 
  Then we have 
    \begin{align}\label{pi:coh}
        \pi_{\ast}\mathcal{H}om(\mathcal{E}_1, \mathcal{E}_2)
        \in D^b(M_S^{\sigma}(v)). 
    \end{align}
    Indeed, the statement \eqref{pi:coh} is local on $M_S^{\sigma}(v)$,
hence it follows from Proposition~\ref{lem:bound} and Lemma~\ref{lem:py}. 
       Then the theorem holds as 
    \begin{align*}
      \mathrm{Hom}^{\ast}(\mathcal{E}_1, \mathcal{E}_2)
    =R^{\ast}\Gamma(\pi_{\ast}\mathcal{H}om(\mathcal{E}_1, \mathcal{E}_2))
    \end{align*}
    and 
    $M_S^{\sigma}(v)$ is a proper algebraic space. 
   \end{proof}

\begin{cor}\label{cor:smooth}
If $(d, w)\in\mathbb{N}\times\mathbb{Z}$ are coprime and $g\geq 2$, 
then $\mathbb{T}_S(v)_w^{\rm{red}}$ is proper and smooth. 
\end{cor}
\begin{proof}
    By Theorem~\ref{thm:proper} and Theorem~\ref{thm:regularred}, 
    the category $\mathbb{T}_S(v)_w^{\rm{red}}$ is proper and
    regular if $\gcd(d, w)=1$. 
    Then it is also proper and smooth by~\cite[Theorem~3.18]{Orsmooth}. 
\end{proof}

\subsection{Strong generation of singular support quotients}
\label{subsec:ssuport:gen}
In this subsection, we prove Theorem~\ref{thm:regular2}
on strong generation of singular support quotients, which 
was used in the proof of Theorem~\ref{thm:regular}. 

Let $\mathfrak{M}$ be a quasi-smooth derived stack 
of finite type over $\mathbb{C}$ 
such that 
its classical truncation 
$\mathcal{M}=\mathfrak{M}^{\rm{cl}}$ admits 
a good moduli space $\mathcal{M} \to M$
which is quasi-separated. 
Note that $M$ is quasi-compact by the 
assumption on $\mathfrak{M}$. 

We denote by $\mathrm{Et}/M$ the category 
whose objects $(U, \rho)$,
where $U$ is a $\mathbb{C}$-scheme and $\rho \colon U \to M$
is an \'{e}tale morphism. 
The set of morphisms $(U', \rho') \to (U, \rho)$
consists of \'{e}tale morphisms $U' \to U$ commuting with 
$\rho$ and $\rho'$. 

For 
a closed subscheme $Z \subset U$,
an \'{e}tale morphism $f \colon U' \to U$ 
is called an \textit{\'{e}tale neighborhood} of $Z$ if 
$f^{-1}(Z) \to Z$ is an isomorphism. 

We will use the following result in the proof of Theorem \ref{thm:regular}: 
\begin{thm}
\emph{(\cite[Theorem~D]{Rydetale})}\label{thm:ryd}
 Let $\mathbf{D} \subset \mathrm{Et}/M$
 be the subcategory satisfying the following 
 conditions:
\begin{enumerate}[(i)]
\item If $(U\to M) \in \mathbf{D}$
     and $(U' \to U)$ is a morphism in $\mathrm{Et}/M$, 
     then $(U' \to M) \in \mathbf{D}$. 
\item If 
$(U' \to M) \in \mathbf{D}$ and 
$(U' \to U)$ is a morphism in $\mathrm{Et}/M$
     which is finite and surjective,
     then $(U \to M) \in \mathbf{D}$. 
     \item If $(j \colon U^{\circ} \to U)$
     and $(f \colon W \to U)$ 
     are morphisms in $\mathrm{Et}/M$
     such that $j$ is an open immersion and $f$ is 
     an \'{e}tale neighborhood of $U\setminus U^{\circ}$, and $(U^{\circ} \to M)
     \in \mathbf{D}$ and $(W \to M) \in \mathbf{D}$,
     then $(U \to M) \in \mathbf{D}$. 
     \end{enumerate}
  Then if there 
 is $(g \colon M' \to M) \in \mathbf{D}$ 
 such that $g$ is surjective, then 
$(\id \colon M \to M) \in \mathbf{D}$. 
\end{thm}

For each object $(U \to M) \in \mathrm{Et}/M$, 
let $\mathcal{M}_U \to U$ be the pull-back of 
$\mathcal{M} \to M$ by $U \to M$. 
There is a derived stack $\mathfrak{M}_U$, unique 
up to equivalence, 
such that for each morphism 
$\rho \colon U' \to U$ in $\mathrm{Et}/M$
there is an induced diagram, see Subsection~\ref{subsec:qsmooth} 
\begin{align}\label{dia:induced}
\xymatrix{
U' \ar[d]_-{\rho} & \ar[l] 
    \mathcal{M}_{U'} \ar[r] \ar[d] & \mathfrak{M}_{U'} \ar[d]^-{\rho} \\
  U & \ar[l]   \mathcal{M}_U \ar[r] & \mathfrak{M}_U. 
    }
\end{align}
For each $y\in M$, there is $\rho \colon U \to M$
in $\mathrm{Et}/M$
whose image contains $y$ such that 
$\mathfrak{M}_U$ is equivalent to 
a Koszul stack
\begin{align}\label{eq:Kstack}
\mathfrak{M}_U \simeq s^{-1}(0)/G
\end{align}
for some $(Y, V, s, G)$, 
where $Y$ is a smooth scheme with 
an action of a reductive algebraic group $G$, 
$V\to Y$ is a $G$-equivariant vector bundle 
with a $G$-invariant section $s$ and $s^{-1}(0)$ is the 
derived zero locus of $s$, see~Subsection~\ref{subsec:qsmooth}. 

For $\ell \in \mathrm{Pic}(\mathfrak{M})_{\mathbb{R}}$ and $(U\to M)\in\mathrm{Et}/M$, consider 
the $\ell$-semistable locus
\begin{align}\label{l-stable}
    \Omega_{\mathfrak{M}_U}^{\ell\text{-ss}}[-1]^{\rm{cl}} \subset \Omega_{\mathfrak{M}_U}[-1]^{\rm{cl}}.
\end{align}
We denote by $\mathscr{Z}_U$ the complement of the 
open immersion (\ref{l-stable}), which is a conical 
closed substack. 
Let $\mathcal{C}_{\mathscr{Z}_U} \subset D^b(\mathfrak{M}_U)$
be the subcategory of objects with singular supports
contained in $\mathscr{Z}_U$. 
\begin{lemma}\label{lem:compact}
Suppose that the open substack (\ref{l-stable}) is an algebraic space. 
Then for a Koszul stack as in (\ref{eq:Kstack}), 
    the category $D^b(\mathfrak{M}_U)/\mathcal{C}_{\mathscr{Z}_{U}}$ is 
    regular. In particular, there is a 
    compact generator $\mathcal{E}_U \in D^b(\mathfrak{M}_U)/\mathcal{C}_{\mathscr{Z}_U}$. 
\end{lemma}
\begin{proof}
    By the Koszul duality equivalence in Theorem~\ref{thm:Kduality}, we have 
    the equivalence 
    \begin{align}\label{equiv:KZ}
        D^b(\mathfrak{M}_U) \stackrel{\sim}{\to} 
        \mathrm{MF}^{\rm{gr}}(V^{\vee}/G, f).  
    \end{align}
    The above equivalence descends to an equivalence, see~\cite[Proposition~2.3.9]{T}
    \begin{align*}
        D^b(\mathfrak{M}_U)/\mathcal{C}_{\mathscr{Z}_U} \stackrel{\sim}{\to} 
          \mathrm{MF}^{\rm{gr}}((V^{\vee}/G) \setminus \mathscr{Z}_U, f). 
    \end{align*}
    Note that we have 
    \begin{align}\label{crit:Uw}
       (\Omega_{\mathfrak{M}}[-1]^{\rm{cl}}\setminus \mathscr{Z})\times_{\mathcal{M}}U
       =(\mathrm{Crit}(w)/G) \setminus \mathscr{Z}_U,
    \end{align}
    hence the right hand side is an algebraic space 
    by the assumption.     
Let $(V^{\vee})^{\rm{free}} \subset (V^{\vee})^{\ell\text{-ss}}$
be the $(G\times \mathbb{C}^{\ast})$-invariant open subspace of $\ell$-semistable points 
with free closed $G$-orbits. 
Then (\ref{crit:Uw}) is a closed substack of 
$Y:=(V^{\vee})^{\rm{free}}/G$. 
   Since the category of matrix factorizations depends only on an open 
    neighborhood of the critical locus, 
    there is an equivalence 
    \begin{align*}
         \mathrm{MF}^{\rm{gr}}((V^{\vee}/G) \setminus \mathscr{Z}_U, f)
         \stackrel{\sim}{\to}   \mathrm{MF}^{\rm{gr}}(Y, f). 
    \end{align*}
    Note that $Y$ is quasi-projective
    since it is an open subset of 
the quasi-projective good moduli space $(V^{\vee})^{\ell\text{-ss}}\ssslash G$. 
      The category $\mathrm{MF}^{\rm{gr}}(Y, f)$ is proven 
    to be 
    smooth in~\cite[Lemma~2.11, Remark~2.12]{FavTy}, 
    hence it is regular.    
\end{proof}

The main result of this subsection is the following strong generation of singular 
support quotient: 
\begin{thm}\label{thm:regular2}
Let $\mathfrak{M}$ be a quasi-smooth 
derived stack of finite type over $\mathbb{C}$
with a good moduli space \[\mathfrak{M}^{\rm{cl}} \to M,\] where $M$
is a quasi-separated algebraic space. 
For $\ell \in \mathrm{Pic}(\mathfrak{M})_{\mathbb{R}}$, 
suppose that $\Omega_{\mathfrak{M}}^{\ell\text{-ss}}[-1]^{\rm{cl}}$
is an algebraic space. 
Let $\mathscr{Z} = \Omega_{\mathfrak{M}}[-1]^{\rm{cl}}  \setminus \Omega_{\mathfrak{M}}^{\ell\text{-ss}}[-1]^{\rm{cl}}$. Then 
the quotient category $D^b(\mathfrak{M})/\mathcal{C}_{\mathscr{Z}}$ is
 regular. 
\end{thm}
\begin{proof}
    For $(U \to M) \in \mathrm{Et}/M$, we define \begin{align}\label{def:mathcaltu}    \mathcal{T}_U=D^b(\mathfrak{M}_U)/\mathcal{C}_{\mathscr{Z}_U}, \     \Ind\mathcal{T}_U=\Ind D^b(\mathfrak{M}_U)/ \Ind \mathcal{C}_{\mathscr{Z}_U}.
    \end{align}
By the diagram (\ref{dia:induced}),
there is an adjoint pair: 
 \begin{align*}
	\xymatrix{
		\Ind\mathcal{T}_{U} \ar@<0.5ex>[r]^{\rho^{\ast}} & \Ind\mathcal{T}_{U'} \ar@<0.5ex>[l]^{\rho_{\ast}}
	}, \rho^{\ast} \dashv \rho_{\ast}. 
\end{align*}
Then $U \mapsto \Ind\mathcal{T}_{U}$ is a $\mathrm{Et}/M$-pre-triangulated categories with adjoints, 
see~\cite[Section~5]{Hperf}. 
    
    Let $\mathbf{D}^{\rm{st}} \subset \mathrm{Et}/M$
    be the full subcategory of $(U \to M)$ such that 
    $\mathcal{T}_U$ is regular. 
    The condition $(U \to M) \in \mathbf{D}^{\rm{st}}$ is 
    equivalent to $\mathcal{T}_U=\langle \mathcal{E}_U \rangle^{\star n}$ for some
    $\mathcal{E}_U \in \mathcal{T}_U$ and $n\geq 1$. 
On the other hand, it is proved in~\cite[Proposition~3.2.7, Section~7.2]{T}
that $\Ind \mathcal{T}_U=\Ind(\mathcal{T}_U)$ with 
compact objects the idempotent closure of $\mathcal{T}_U$.
    Therefore by~\cite[Proposition~1.9]{Neeman}, 
    the condition $\mathcal{T}_U=\langle \mathcal{E}_U \rangle^{\star n}$
    is equivalent to $\Ind \mathcal{T}_U=
    \llangle \mathcal{E}_U \rrangle^{\star n}$ for 
    some $n\geq 1$. 
    By Lemma~\ref{lem:compact}, 
    there exists $(M' \to M) \in \mathbf{D}^{\rm{st}}$ which 
    is surjective. 
    By Theorem~\ref{thm:ryd}, it is enough to check the conditions 
    (i), (ii) and (iii) for the 
    subcategory $\mathbf{D}^{\rm{st}} \subset \mathrm{Et}/M$. 

    To show condition (i), consider a morphism $(\rho \colon U' \to U)$
    in $\mathrm{Et}/M$. Suppose that 
    $\Ind \mathcal{T}_U=\llangle \mathcal{E}_U \rrangle^{\star n}$. 
    For each $A \in \Ind \mathcal{T}_{U'}$, 
    there is a natural morphism $\rho^{\ast}\rho_{\ast}A \to A$
    and $\rho_{\ast}A \in \llangle \mathcal{E}_U \rrangle^{\star n}$ by the assumption. 
       Since $U' \times_U U' \to U'$ admits a section 
    given by the diagonal, we have a decomposition 
    into open and closed subsets 
    \[U' \times_U U'=U' \sqcup U''.\] 
    Then, by the base change 
    for $U' \to U \leftarrow U'$, 
    the morphism 
    $\rho^{\ast}\rho_{\ast}A \to A$ splits, 
    hence $A \in \llangle \rho^{\ast}\mathcal{E}_U \rrangle^{\star n}$.
    Therefore  
    $\Ind \mathcal{T}_{U'}=\llangle \mathcal{E}_{U'}\rrangle^{\star n}$
    for $\mathcal{E}_{U'}=\rho^{\ast}\mathcal{E}_U$ 
    and
    $(U' \to M) \in \mathbf{D}^{\rm{st}}$ holds. 

To show condition (ii), let $(\rho \colon U' \to U)$
 be a morphism in $\mathrm{Et}/M$ such that 
 $\rho$ is finite surjective. 
  Assume that $(U' \to M) \in \mathbf{D}^{\rm{st}}$, 
  so $\Ind \mathcal{T}_{U'}=\llangle \mathcal{E}_{U'} \rrangle^{\star n}$ for some $\mathcal{E}_{U'} \in \mathcal{T}_{U'}$ and $n\geq 1$. 
  For $A \in \Ind \mathcal{T}_U$, let 
  $A \to \rho_{\ast}\rho^{\ast}A=A \otimes \rho_{\ast}\mathcal{O}_{\mathfrak{M}_U}$
  be the natural morphism. 
 The induced map $\rho \colon \mathfrak{M}_{U'} \to \mathfrak{M}_U$ is also finite and surjective, 
 and $\mathcal{O}_{\mathfrak{M}_U} \to \rho_{\ast}\mathcal{O}_{\mathfrak{M}_{U'}}$ splits. 
 In fact, we have $\rho_{\ast}=\rho_{!}$ as $\rho$ is 
 finite \'{e}tale, and the natural map $\rho_{!}\mathcal{O}_{\mathfrak{M}_{U'}} \to \mathcal{O}_{\mathfrak{M}_U}$ gives a splitting. 
 Therefore $A$ is a direct summand of $\rho_{\ast}\rho^{\ast}A$. 
 As $\rho^{\ast}A \in \llangle \mathcal{E}_{U'} \rrangle^{\star n}$, we have 
 $A \in \llangle \rho_{\ast}\mathcal{E}_{U'} \rrangle^{\star n}$. 
 Since $\rho$ is finite, we have 
 $\rho_{\ast}\mathcal{E}_{U'} \in \mathcal{T}_U$. 
  Then by setting $\mathcal{E}_U=\rho_{\ast}\mathcal{E}_{U'}$, 
  we have $A \in \llangle \mathcal{E}_{U} \rrangle^{\star n}$, 
 hence $\Ind \mathcal{T}_U= \llangle \mathcal{E}_{U} \rrangle^{\star n}$
 and $(U \to M) \in \mathbf{D}^{\rm{st}}$ holds. 

To show condition (iii), let $(j \colon U_{\circ} \to U)$
and $(f \colon W \to U)$ be morphisms in $\mathrm{Et}/M$
such that $j$ is an open immersion and $f$ is an 
\'{e}tale neighborhood of $U \setminus U_{\circ}$. 
Suppose that $\Ind \mathcal{T}_{U_{\circ}}=\llangle \mathcal{E}_{U_{\circ}} \rrangle^{\star n}$
and $\Ind \mathcal{T}_{W}=\llangle \mathcal{E}_W \rrangle^{\star n}$ for some $n \geq 1$
and $\mathcal{E}_W \in \mathcal{T}_W$, $\mathcal{E}_{U_{\circ}} \in \mathcal{T}_{U_{\circ}}$. 
For an object $A \in \Ind \mathcal{T}_U$, there is 
a distinguished triangle
in $\Ind \mathcal{T}_U$, see~\cite[Lemma~5.9]{Hperf}:
\begin{align}\label{tr:A}
    A \to j_{\ast}j^{\ast}A \oplus f_{\ast}f^{\ast}A 
    \to f_{\ast}f^{\ast}j_{\ast}j^{\ast}A \to A[1].
\end{align}
We have $j_{\ast}j^{\ast}A \in \llangle j_{\ast}\mathcal{E}_{U_{\circ}}\rrangle^{\star n}$, 
$f_{\ast}f^{\ast}A \in \llangle f_{\ast}\mathcal{E}_W
\rrangle^{\star n}$
and 
$f_{\ast}f^{\ast}j_{\ast}j^{\ast}A \in \llangle f_{\ast}\mathcal{E}_W\rrangle^{\star n}$. 
By Lemma~\ref{lem:cohbou}, 
there exist $\mathcal{E}_U \in \mathcal{T}_U$
such that $j_{\ast}\mathcal{E}_{U_{\circ}}$
and $f_{\ast}\mathcal{E}_W$ are objects in 
$\llangle \mathcal{E}_U \rrangle^{\star m}$ for some $m\geq 1$.  
Then we have 
$j_{\ast}j^{\ast}A \in \llangle \mathcal{E}_U \rrangle^{\star nm}$, 
$f_{\ast}f^{\ast}A \in \llangle \mathcal{E}_U
\rrangle^{\star nm}$
and 
$f_{\ast}f^{\ast}j_{\ast}j^{\ast}A \in \llangle \mathcal{E}_U\rrangle^{\star nm}$. 
From the triangle (\ref{tr:A}), 
we conclude that $\Ind \mathcal{T}_U=\llangle 
\mathcal{E}_U \rrangle^{\star nm+1}$,
therefore $(U \to M) \in \mathbf{D}^{\rm{st}}$. 
  \end{proof}

We have used the following lemma: 
\begin{lemma}\label{lem:cohbou}
Let $f \colon U' \to U$ be a morphism in 
$\mathrm{Et}/M$. Then 
for any object $P \in D^b(\mathfrak{M}_{U'})$, there is 
$Q \in D^b(\mathfrak{M}_{U})$
and $m \geq 1$ such that
$f_{\ast} P \in \llangle Q \rrangle^{\star m}$
in $\Ind D^b(\mathfrak{M}_U)$. 
\end{lemma}
\begin{proof}
    Since $P$ is a finite extension of objects 
    from the image of the pushforward functor 
    $D^b(\mathcal{M}_{U'}) \to D^b(\mathfrak{M}_{U'})$, 
    we may assume that $P\in D^b(\mathcal{M}_{U'})$. It suffices to find $Q\in D^b(\mathcal{M}_U)$ and $m\geq 1$ such that $f_*P\in \llangle Q \rrangle^{\star m}$ in $\Ind D^b(\mathcal{M}_U)$. 
    By Nagata compactification, there is a factorization 
    \[f \colon U' \stackrel{j}{\hookrightarrow} \overline{U} \stackrel{g}{\to} U,\]
    where $j$ is an open immersion and $g$ is proper. 
    There is an object $\overline{P} \in D^b(\mathcal{M}_{\overline{U}})$
    such that $j^{\ast}\overline{P}\cong P$. 
    Then $j_{\ast}P \cong \overline{P} \otimes_{\mathcal{O}_{\overline{U}}} j_{\ast}\mathcal{O}_{U'}$, 
    where $j_{\ast}\mathcal{O}_{U'} \in D_{\rm{qc}}(\overline{U})$
    and the tensor product is given by the action of 
    $D_{\rm{qc}}(\overline{U})$ on $\Ind D^b(\mathcal{M}_{\overline{U}})$. 
    By~\cite[Theorem~6.2]{Neeman}, 
    there is $B \in \mathrm{Perf}(\overline{U})$ such that 
    $j_{\ast}\mathcal{O}_{U'} \in \llangle B \rrangle^{\star m}$ for some $m\geq 1$
    in $D_{\rm{qc}}(\overline{U})$. 
    Then $j_{\ast}P \in \llangle \overline{P} \otimes_{\mathcal{O}_{\overline{U}}}B\rrangle^{\star m}$, hence $f_{\ast}P \in \llangle Q \rrangle^{\star m}$
    for $Q=g_{\ast}(\overline{P}\otimes_{\mathcal{O}_{\overline{U}}}B)
    \in D^b(\mathcal{M}_U)$. 
\end{proof}

\section{Serre functor for reduced quasi-BPS categories}
In this section, we show that the reduced quasi-BPS 
categories have \'etale locally trivial 
Serre functor, which gives further evidence towards 
Conjecture~\ref{conj:HK}. 
\subsection{Serre functor}\label{subsec71}
Recall that we write $v=dv_0$ for $d\in\mathbb{Z}_{\geq 1}$ and a primitive Mukai vector $v_0$ 
with $\langle v_0, v_0 \rangle=2g-2$. We assume $g\geq 2$.
Consider a generic stability $\sigma \in \mathrm{Stab}(S)$. 
Recall that the derived 
stack $\mathfrak{M}_S^{\sigma}(v)^{\rm{red}}$ is 
equivalent to 
its classical truncation $\mathcal{M}=\mathcal{M}^\sigma_S(v)$ by Lemma~\ref{lem:classical}. 
Let $w \in \mathbb{Z}$ such that $\gcd(d, w)=1$, 
and consider the
quasi-BPS category \[\mathbb{T}=\mathbb{T}_S^{\sigma}(v)_w^{\rm{red}}\subset D^b(\mathcal{M}).\]

We recall some terminology from \cite{MR1996800}. Let $\mathcal{T}$ be a $\mathbb{C}$-linear pre-triangulated category.
A contravariant functor $F\colon \mathcal{T}\to \mathrm{Vect}(\mathbb{C})$ is called \textit{of finite type} if $\oplus_{i\in\mathbb{Z}}F(A[i])$ is finite dimensional for all objects $A$ of $\mathcal{T}$. The category $\mathcal{T}$ is called \textit{saturated} if every contravariant functor $H\colon\mathcal{T}\to \mathrm{Vect}(\mathbb{C})$ of finite type is representable.

By Corollary~\ref{cor:smooth} and~\cite[Theorem~1.3]{MR1996800}, the category $\mathbb{T}$ 
is saturated, and thus it 
admits a Serre functor 
\begin{align*}S_{\mathbb{T}} \colon \mathbb{T} \to \mathbb{T}
\end{align*}
i.e. 
a functor such that there are functorial isomorphisms 
for $\mathcal{E}_1, \mathcal{E}_2 \in \mathbb{T}$:
\begin{align*}
\Hom(\mathcal{E}_1, \mathcal{E}_2) \cong 
\Hom(\mathcal{E}_2, S_{\mathbb{T}}(\mathcal{E}_1))^{\vee}. 
\end{align*}

There is also a version of the Serre functor 
relative to the good moduli space 
$\pi \colon \mathcal{M} \to M$. 
For $\mathcal{E}_1, \mathcal{E}_2 \in \mathbb{T}$, 
let $\mathcal{H}om_{\mathbb{T}}(\mathcal{E}_1, \mathcal{E}_2) \in D_{\rm{qc}}(\mathcal{M})$ 
be its internal homomorphism. 
Then a functor $S_{\mathbb{T}/M}\colon \mathbb{T} \to \mathbb{T}$ is called a 
\textit{relative Serre functor} if there are functorial isomorphisms in $D^b(M)$:
\begin{align}\label{rel:S}
\mathcal{H}om_M(\pi_{\ast}\mathcal{H}om_{\mathbb{T}}(\mathcal{E}_1, \mathcal{E}_2), \mathcal{O}_M)
\cong \pi_{\ast}\mathcal{H}om_{\mathbb{T}}(\mathcal{E}_2, S_{\mathbb{T}/M}(\mathcal{E}_1)).
\end{align}
\begin{remark}\label{rmk:Gorenstein}
We note that $M$ has at worst 
Gorenstein singularities. The result is most probably well-known, but we did not find a reference. The statement follows from 
Lemma~\ref{lem:py} and~\cite[Lemma~5.7]{PTquiver}.
Thus $\mathcal{H}om(-, \mathcal{O}_M)$ is 
an equivalence 
\begin{align*}
    \mathcal{H}om(-, \mathcal{O}_M) \colon 
    D^b(M) \stackrel{\sim}{\to} D^b(M)^{\rm{op}}. 
\end{align*}
Moreover, the dualizing complex 
is $\omega_M=\mathcal{O}_M[\dim M]$, since 
the singular locus of $M$ is at least 
codimension two and there is a holomorphic 
symplectic form on the smooth part. 
\end{remark}

\begin{remark}
The category $\mathbb{T}$ is
proper over $M$, i.e. $\pi_{\ast}\mathcal{H}om_{\mathbb{T}}(\mathcal{E}_1, \mathcal{E}_2) \in D^b(M)$, 
and it is strongly generated. Thus the 
relative Serre functor also exists, and is constructed as follows. Let $\mathcal{E} \in \mathbb{T}$ be
a strong generator and consider the sheaf of dg-algebras on $M$:
\[\mathcal{A}=\pi_{\ast}\mathcal{H}om_{\mathbb{T}}(\mathcal{E}, \mathcal{E}).\]
Then $\mathbb{T}$ is equivalent to the derived 
category of 
coherent right dg-$\mathcal{A}$-modules. 
Under the above equivalence, the relative Serre 
functor is given by the 
$\mathcal{A}^{\rm{op}}\otimes_{\mathcal{O}_M}
\mathcal{A}$-module
$\mathcal{H}om_{M}(\mathcal{A}, \mathcal{O}_M)$. 
\end{remark}

The absolute and the relative Serre functors are 
related as follows: 
\begin{lemma}\label{lem:Serre}
    We have $S_{\mathbb{T}}=S_{\mathbb{T}/M}[\dim M]$. 
\end{lemma}
\begin{proof}
    By taking the global sections of (\ref{rel:S}),
    we obtain 
    \begin{align*}
        \Hom_M(\pi_{\ast}\mathcal{H}om_{\mathbb{T}}(\mathcal{E}_1, \mathcal{E}_2), \mathcal{O}_M)
        \cong \Hom(\mathcal{E}_2, S_{\mathbb{T}/M}(\mathcal{E}_1)).
    \end{align*}
    By the Serre duality for $M$ and using 
    that $\omega_M=\mathcal{O}_M[\dim M]$
    from Remark~\ref{rmk:Gorenstein}, the left hand side 
    is isomorphic to 
    \begin{align*}
        \Hom_M(\mathcal{O}_M, \pi_{\ast}\mathcal{H}om_{\mathbb{T}}(\mathcal{E}_1, \mathcal{E}_2)[\dim M])^{\vee}
        =\Hom(\mathcal{E}_1, \mathcal{E}_2[\dim M])^{\vee}.
    \end{align*}
    Then the lemma holds by the uniqueness of 
    $S_{\mathbb{T}}$. 
    \end{proof}

We believe that $S_\mathbb{T}$ is isomorphic to the shift functor $[\dim M]$, see the discussion in Subsection \ref{subsec12}, which reinforces the analogy between reduced quasi-BPS categories and hyperkähler varieties, see Conjecture~\ref{conj:HK}. 
The main result in this section is the following weaker form of this expectation, which we prove in Subsection~\ref{subsec:proof2}:

\begin{thm}\label{thm:Serre:etale}
The Serre functor $S_{\mathbb{T}}$ is isomorphic to 
the shift functor 
$[\dim M]$ \'{e}tale locally on $M$, i.e. there is an 
\'{e}tale cover $U \to M$ such that for 
each $\mathcal{E} \in \mathbb{T}$
we have $S_{\mathbb{T}}(\mathcal{E})|_U \cong 
\mathcal{E}|_U[\dim M]$.
    \end{thm}

\subsection{Construction of the trace map}\label{subsec:trace}
In this subsection, we construct a
trace map for objects with nilpotent 
singular supports in a general setting. 
The construction here is used in the proof 
of Theorem~\ref{thm:Serre:etale}. 

Let $G$ be a reductive 
algebraic group which acts on 
a smooth affine variety $Y$. 
We assume that there is a one-dimensional 
subtorus $\mathbb{C}^{\ast} \subset G$
which acts on $Y$ trivially, so 
the $G$-action on $Y$ factors through the 
action of $\mathbb{P}(G):=G/\mathbb{C}^{\ast}$. 
We say that $Y$ is \textit{unimodular} if $\det \Omega_Y$ is 
trivial as a $G$-equivariant line bundle.
We also say that 
the action of $\mathbb{P}(G)$
on $Y$ is \textit{generic} if the subset $Y^s \subset Y$
of points with closed $\mathbb{P}(G)$-orbit and trivial 
stabilizer is non-empty and 
$\mathrm{codim}(Y\setminus Y^s) \geq 2$. 
\begin{lemma}\emph{(\cite[Korollary~2]{Knop2})}\label{lem:Gorenstein}
If $Y$ is unimodular and generic, then 
$Y\ssslash G$
has only Gorenstein singularities
and its canonical module is trivial. 
\end{lemma}

Let $Y$ be unimodular and generic. By Lemma~\ref{lem:Gorenstein}, the quotient 
$Y\ssslash G$ is Gorenstein and its dualizing 
complex is 
\begin{align}\label{dualcomp}\omega_{Y\ssslash G}=\mathcal{O}_{Y\ssslash G}[\dim Y\ssslash G].
\end{align}
Let $V \to Y$ 
be a $G$-equivariant vector bundle with a $G$-invariant
regular section $s$ such that $V$ is also unimodular and generic. 
We refer to such choices of $G$, $Y$, $V$, and $s$ as \textit{a good data} $(G, Y, V, s)$.

Let $\mathcal{U}:=s^{-1}(0)$ be 
the zero locus of $s$, which is equivalent to the 
derived zero locus as we assumed that $s$ is regular. 
We have the following diagram 
\begin{align}\label{dia:cartesian}
    \xymatrix{
    \mathcal{U}/G \inclusion^-{j} \ar[d]_-{\pi_{\mathcal{U}}} &
    Y/G \ar[d]_-{\pi_Y} \ar@<-0.5ex>[r]_-{0} & V^{\vee}/G 
 \ar@<-0.5ex>[l]_-{\eta} \ar[d]_-{\pi_{V^{\vee}}} \\
 \mathcal{U}\ssslash G \inclusion^-{\overline{j}} & 
 Y\ssslash G \ar@<-0.5ex>[r]_-{\overline{0}} &\ar@<-0.5ex>[l]_-{\overline{\eta}} V^{\vee}\ssslash G. 
 }
\end{align}
Here $0 \colon Y/G \to V^{\vee}/G$ is the zero section, 
$\eta$ is the projection, and the bottom horizontal 
arrows are induced maps on good moduli spaces.

Recall the Koszul duality equivalence in Theorem~\ref{thm:Kduality}
\begin{align*}
\Theta \colon D^b(\mathcal{U}/G) \stackrel{\sim}{\to} 
\mathrm{MF}^{\rm{gr}}(V^{\vee}/G, f).     
\end{align*}
For $\mathcal{E} \in D^b(\mathcal{U}/G)$, let 
$\mathcal{P}=\Theta(\mathcal{E})$. 
Then we have the following isomorphism in 
$D_{\rm{qc}}(Y/G)$, see~\cite[Lemma~2.7]{PTquiver}:
\begin{align*}
    j_{\ast}\mathcal{H}om_{\mathcal{U}/G}(\mathcal{E}, \mathcal{E}) \stackrel{\cong}{\to} \eta_{\ast} \mathcal{H} om_{V^{\vee}/G}(\mathcal{P}, \mathcal{P}). 
\end{align*}
Here $\mathcal{H}om_{V^{\vee}/G}(\mathcal{P}, \mathcal{P})$
is the internal homomorphism of matrix factorizations, 
which is an object in $\mathrm{MF}^{\rm{gr}}(V^{\vee}/G, 0)$.
As $V^{\vee}/G$ is smooth, 
by taking a resolution of $\mathcal{P}$
by vector bundles, 
we obtain 
the natural trace map 
in $\mathrm{MF}^{\rm{gr}}(V^{\vee}/G, 0)$:
\begin{align*}
    \mathrm{tr} \colon \mathcal{H}om_{V^{\vee}/G}(\mathcal{P}, \mathcal{P})
    \to \mathcal{O}_{V^{\vee}/G}. 
\end{align*}
By taking $\pi_{V^{\vee}\ast}$, we obtain 
the morphism in $D^{\rm{gr}}(V^{\vee}\ssslash G)$:
\begin{align}\label{tr:push}
    \pi_{V^{\vee}\ast}\mathrm{tr} \colon 
    \pi_{\ast}\mathcal{H}om_{V^{\vee}/G}(\mathcal{P}, \mathcal{P})
    \to \mathcal{O}_{V^{\vee}\ssslash G}
\end{align}
Here the grading on $V^{\vee}\ssslash G$
is induced by the fiberwise weight two
$\mathbb{C}^{\ast}$-action on 
$V^{\vee}/G \to Y/G$, see Subsection~\ref{subsec:graded}
for the graded category $D^{\rm{gr}}(V^{\vee}\ssslash G)$. 

We say that $\mathcal{P}$ has \textit{nilpotent support} if:
\begin{align*}\mathrm{Supp}(\mathcal{P}) \subset \pi_{V^{\vee}}^{-1}(\mathrm{Im}(\overline{0})).
\end{align*}
We say $\mathcal{E}$ has \textit{nilpotent singular support} with respect to $(G, Y, V, s)$ if $\mathcal{P}$ has nilpotent support. 

Assume that $\mathcal{P}$ has nilpotent support.
Then 
the object $\pi_{V^{\vee}\ast}\mathcal{H}om_{V^{\vee}/G}(\mathcal{P}, \mathcal{P})$ in $D^{\rm{gr}}(V^{\vee}\ssslash G)$
has proper support over $Y\ssslash G$. 
Moreover, we have 
\begin{align}\label{equalityomegao}
 \omega_{V^{\vee}\ssslash G}=
 \mathcal{O}_{V^{\vee}\ssslash G}[\dim V^{\vee}\ssslash G](-2 \rk V)=
 \mathcal{O}_{V^{\vee}\ssslash G}[\dim Y\ssslash G-\rk V]
\end{align}
in $D^{\rm{gr}}(V^{\vee}\ssslash G)$,
where $(1)$ is the grade shift functor of $D^{\rm{gr}}(V^{\vee}\ssslash G)$
which is isomorphic to the cohomological shift 
functor $[1]$. 
Then by Lemma~\ref{lem:psupport} below, the morphism 
(\ref{tr:push}) induces the morphism 
in $D^b(Y\ssslash G)$:
\begin{align}\label{induce:ap}
    a_{\mathcal{P}} \colon 
    \eta_{\ast}\pi_{V^{\vee}\ast}\mathcal{H}om_{V^{\vee}/G}(\mathcal{P}, \mathcal{P}) \to \mathcal{O}_{Y\ssslash G}
[\rk V]. 
\end{align}
Suppose that $\mathcal{U}\ssslash G$ is Gorenstein 
with trivial canonical module 
and has dimension 
$\dim Y\ssslash G -\rk V$. 
Then $\overline{j}^! \mathcal{O}_{Y\ssslash G}=
\mathcal{O}_{\mathcal{U}\ssslash G}[-\rk V]$. 
Since there are isomorphisms:
\begin{align*}
\overline{\eta}_{\ast}\pi_{V^{\vee}\ast}\mathcal{H}om_{V^{\vee}/G}(\mathcal{P}, \mathcal{P})
&=\pi_{Y\ast}\eta_{\ast}\mathcal{H}om_{V^{\vee}/G}(\mathcal{P}, \mathcal{P}) \\
&\stackrel{\cong}{\to} \pi_{Y\ast}j_{\ast}\mathcal{H}om_{\mathcal{U}/G}(\mathcal{E}, \mathcal{E}) \\
&=\overline{j}_{\ast}\pi_{\mathcal{U}\ast}\mathcal{H}om_{\mathcal{U}/G}(\mathcal{E}, \mathcal{E}),
\end{align*}
the morphism (\ref{induce:ap}) induces the trace
morphism in $D^b(\mathcal{U}\ssslash G)$:
\begin{align}\label{const:tre}
 \mathrm{tr}_{\mathcal{E}} \colon 
 \pi_{\mathcal{U}\ast}\mathcal{H}om_{\mathcal{U}/G}(\mathcal{E}, \mathcal{E})
 \to \overline{j}^{!}\mathcal{O}_{Y\ssslash G}[\rk V]
 =\mathcal{O}_{\mathcal{U}\ssslash G}. 
\end{align}

We have used the following lemma 
in the above construction: 
\begin{lemma}\label{lem:psupport}
Let 
$X, Y$ be Noetherian $\mathbb{C}$-schemes with $\mathbb{C}^{\ast}$-actions, 
and let
$f \colon X \to Y$ be a $\mathbb{C}^{\ast}$-equivariant morphism. 
Let $\omega_X$ be a dualizing complex for $X$. 
If $\mathcal{E} \in D^{\rm{gr}}(X)$ 
has proper support over 
$Y$, then 
there is a natural isomorphism 
\[\phi_f \colon \Hom_X(\mathcal{E}, \omega_X) \stackrel{\cong}{\to} \Hom_Y(f_{\ast}\mathcal{E}, \omega_Y).\] 
Moreover, let $g \colon Y \to Z$ be another 
$\mathbb{C}^{\ast}$-equivariant morphism 
and assume the support of $\mathcal{E}$ is proper over $Z$. 
Let $h=g \circ f \colon X \to Z$.  
Then we have 
\begin{align*}
    \phi_h=\phi_g \circ \phi_f \colon 
    \Hom_X(\mathcal{E}, \omega_X) \stackrel{\phi_f}{\to} 
    \Hom_Y(f_{\ast}\mathcal{E}, \omega_Y) \stackrel{\phi_g}{\to} \Hom_Z(h_{\ast}\mathcal{E}, \omega_Z). 
\end{align*}
\end{lemma}
\begin{proof}
The lemma is obvious if $f$ and $g$ are proper 
since
$\omega_X=f^{!}\omega_Y$, $\omega_Y=g^{!}\omega_Z$
and $f^!$ and $g^!$ are right adjoints to $f_{\ast}$, $g_{\ast}$. 
In general, let $i \colon T \hookrightarrow X$ be 
a closed subscheme such that $f|_{T}$, $g|_{f(T)}$ are
proper. By a standard dévissage argument, it suffices to check the statement for $\mathcal{E}=i_{\ast}\mathcal{F}$
for some $\mathcal{F} \in D^b(T)$. 
Then $\Hom_X(\mathcal{E}, \omega_X)=\Hom_T(\mathcal{F}, \omega_T)$ as $\omega_T=i^{!}\omega_X$. 
Then the lemma holds from the case of $f$, $g$ 
proper. 
\end{proof}

\begin{defn}\label{defn:trace}
Let $(G, Y, V, s)$ be a good data.
Suppose that $\mathcal{U}\ssslash G$ is Gorenstein 
with trivial canonical module and of dimension $\dim Y\ssslash G-\mathrm{rank}\,V$. 
For $\mathcal{E} \in D^b(\mathcal{U}/G)$ with nilpotent singular support with respect to this data, the 
morphism 
\begin{align*}
    \mathrm{tr}_{\mathcal{E}} \colon 
  \pi_{\mathcal{U}\ast}\mathcal{H}om_{\mathcal{U}/G}(\mathcal{E}, \mathcal{E})
 \to \mathcal{O}_{\mathcal{U}\ssslash G}   
\end{align*}
constructed 
in (\ref{const:tre}) is called the \textit{trace map 
determined by} $(G, Y, V, s)$. 
\end{defn}

The following lemma is immediate from 
the construction of the trace map:
\begin{lemma}\label{compare:trace0}
For another good data $(G', Y', V', s')$, 
suppose that there is a commutative diagram 
of stacks
\begin{align*}
\xymatrix{
V/G \ar[r]^{\cong} \ar[d]
& V'/G' \ar[d] \\  
Y/G \ar[r]^{\cong} \ar@/^18pt/[u]_-{s} & Y'/G', \ar@/_18pt/[u]^-{s'}
}    
\end{align*}
where the horizontal arrows are isomorphisms. 
Let $\mathcal{U}'=(s')^{-1}(0)$ and 
consider the induced equivalence 
$\phi \colon \mathcal{U}/G \stackrel{\cong}{\to} \mathcal{U}'/G'$. 
For $\mathcal{E} \in D^b(\mathcal{U}/G)$ with nilpotent singular support for $(G, Y, V, s)$, the object $\phi_*\mathcal{E}$ has nilpotent singular support with respect to $(G', Y', V', s')$.
Further, the 
trace map $\mathrm{tr}_{\mathcal{E}}$ determined by 
$(G, Y, V, s)$ is identified with 
that of $\mathrm{tr}_{\phi_{\ast}\mathcal{E}}$
determined by $(G', Y', V', s')$, i.e. 
the following diagram commutes
\begin{align*}
\xymatrix{
\pi_{\mathcal{U}'\ast}\mathcal{H}om_{\mathcal{U}'/G'}(\phi_{\ast}\mathcal{E}, \phi_{\ast}\mathcal{E}) \ar[r]^-{\mathrm{tr}_{\phi_{\ast}\mathcal{E}}}\ar[d]_-{\cong} & 
\mathcal{O}_{\mathcal{U}'\ssslash G'} \ar[d]_-{\cong} \\
\phi_{\ast}\pi_{\mathcal{U}_{\ast}}\mathcal{H}om_{\mathcal{U}/G}(\mathcal{E}, \mathcal{E}) \ar[r]^-{\phi_{\ast}\mathrm{tr}_{\mathcal{E}}} & \phi_{\ast}\mathcal{O}_{\mathcal{U}\ssslash G}, 
}    
\end{align*}
where the vertical arrows are natural isomorphisms 
induced by $\phi$. 
\end{lemma}

Suppose that $\mathcal{E} \in D^b(\mathcal{U}/G)$ is a 
perfect complex. 
In this case, there is a canonical trace map 
$\mathcal{H}om_{\mathcal{U}/G}(\mathcal{E}, \mathcal{E}) \to \mathcal{O}_{\mathcal{U}/G}$. 
By taking the push-forward to $\mathcal{U}\ssslash G$, 
we obtain the map 
\begin{align}\label{trperf}
    \pi_{\mathcal{U}\ast}\mathcal{H}om_{\mathcal{U}/G}(\mathcal{E}, \mathcal{E}) \to 
    \mathcal{O}_{\mathcal{U}\ssslash G}. 
\end{align}
Note that the above construction is independent of 
a choice of $(G, Y, V, s)$. 
The following lemma is straightforward to 
check, and we omit the details. 
\begin{lemma}\label{lem:trperf}
If $\mathcal{E}$ is a perfect complex, then 
$\mathrm{tr}_{\mathcal{E}}$ is the same as the map (\ref{trperf}). 
\end{lemma}
\subsection{Comparison of the trace maps}
In this subsection, we compare 
the trace map constructed in the previous 
subsection under a change of the presentations of
quasi-smooth affine derived schemes. 

Suppose that $(G, Y, V, s)$ is a good data and let 
$W$ be another $G$-representation 
such that $\det W$ is a trivial $G$-character. 
Let $i \colon Y/G \hookrightarrow (Y\oplus W)/G$
be the embedding given by $y\mapsto (y, 0)$. 
We have the section $s'$ of the vector bundle 
$V \oplus W \oplus W \to Y \oplus W$ given by 
$(y, w) \mapsto (s(y), w, w)$, whose zero locus
is $\mathcal{U} \subset Y$. 
Then $(G, Y\oplus W, V\oplus W\oplus W, s')$ is a good data. Let $\mathcal{E}\in D^b(\mathcal{U}/G)$ be a complex with nilpotent singular support with respect to $(G, Y, V, s)$. Then $\mathcal{E}$ also has nilpotent singular support with respect to $(G, Y\oplus W, V\oplus W\oplus W, s')$ and we can
consider the trace determined by the good data $(G, Y\oplus W, V\oplus W\oplus W, s')$: 
\begin{align*}
    \mathrm{tr}_{\mathcal{E}}' \colon 
    \mathcal{H}om_{\mathcal{U}/G}(\mathcal{E}, \mathcal{E})
    \to \mathcal{O}_{\mathcal{U}\ssslash G}.
\end{align*}

\begin{lemma}\label{lem:trace=}
Let $\mathcal{E}\in D^b(\mathcal{U}/G)$ have nilpotent singular support with respect to the good data $(G, Y, V, s)$. Then $\mathcal{E}$ also has nilpotent singular support with respect to the good data $(G, Y\oplus W, V\oplus W\oplus W, s')$. 
    Further, we have that $\mathrm{tr}_{\mathcal{E}}=\mathrm{tr}_{\mathcal{E}}'$.
\end{lemma}
\begin{proof}
We have the following diagram 
\begin{align}\label{diagXc2}
    \xymatrix{
    &   &    &   (Y\oplus W)/G 
    \ar[d]^(.3){\pi_{Y\oplus W}}|\hole &  (V^{\vee} \oplus W \oplus W^{\vee})/G\ar[d]^-{\pi_{V^{\vee} \oplus W \oplus W^{\vee}}} \ar[l]_-{p}\\
     \mathcal{U}/G \inclusion^-{j} \ar[d]_-{\pi_{\mathcal{U}}} & Y/G \ar[rru]^-{i_Y}  \ar[d]_-{\pi_Y} 
      & V^{\vee}/G \ar[l]^-{\eta}
      \ar[d]_(.3){\pi_{V^{\vee}}} \ar[rru]^(.3){i_{V^{\vee}}} & 
      (Y\oplus W)\ssslash G 
      & (V^{\vee}\oplus W \oplus W^{\vee})\ssslash G \ar[l]_-{\overline{p}}\\
      \mathcal{U}\ssslash G \inclusion^-{\overline{j}} & Y\ssslash G \ar[rru]^(.3){\overline{i}_{Y}}|\hole 
      & V^{\vee}\ssslash G 
       \ar[rru]_-{\overline{i}_{V^{\vee}}} \ar[l]^-{\overline{\eta}} & & &
      }
\end{align}
Let $q \colon W \oplus W^{\vee} \to \mathbb{C}$
be the natural non-degenerate pairing. 
From the construction of the Koszul equivalences, 
there is a commutative diagram: 
\begin{align*}
    \xymatrix{
D^b(\mathcal{U}/G) \ar[r]_-{\sim}^-{\Theta} \ar@{=}[d] & 
\mathrm{MF}^{\rm{gr}}(V^{\vee}/G, f) 
\ar[d]^-{\Phi}_-{\sim} \\
D^b(\mathcal{U}/G) \ar[r]_-{\sim}^-{\Theta'} &
\mathrm{MF}^{\rm{gr}}((V^{\vee}\oplus W \oplus W^{\vee})/G, 
f+q).
    }
\end{align*}
Here, 
the horizontal arrows are the Koszul equivalences from Theorem~\ref{thm:Kduality}, 
and 
$\Phi$ is the Kn\"{o}rrer periodicity
equivalence, given by 
$\Phi(-)=(-)\otimes_{\mathbb{C}} \mathcal{K}$. 
The Koszul factorization $\mathcal{K}$ 
of $q$ has the form
\begin{align*}
 \mathcal{K}=\left(\bigwedge^{\rm{even}}
 W \otimes {\mathcal{O}_{W \oplus W^{\vee}}} \rightleftarrows
 \bigwedge^{\rm{odd}}W \otimes \mathcal{O}_{W \oplus W^{\vee}}
 \right) \in \mathrm{MF}^{\rm{gr}}((W \oplus W^{\vee})/G, q)
\end{align*}
and is isomorphic to $\mathcal{O}_{(W \oplus \{0\})/G}$, see~\cite[Proposition~3.20]{MR3895631}. 
In the above, the grading is given by the 
$\mathbb{C}^{\ast}$-action on $W \oplus W^{\vee}$ of
weight $(0, 2)$. 
By a diagram chasing, we see that 
\[\mathcal{Q}:=\Theta'(\mathcal{E})=\Phi(\mathcal{P})\] has support in $\mathrm{Im}(\overline{0})$, where $\overline{0}\colon (Y\oplus W)\ssslash G\to (Y\oplus W\oplus W^\vee)\ssslash G$. Then $\mathcal{E}$ has nilpotent singular support with respect to $(G, Y\oplus W, V\oplus W\oplus W, s')$.

Let $i_0 \colon BG \hookrightarrow (W \oplus W^{\vee})/G$
be the inclusion of the origin. 
We have the Koszul equivalence
\begin{align*}
    D^b(BG) \stackrel{\sim}{\to}
    \mathrm{MF}^{\rm{gr}}((W\oplus W^{\vee})/G, q)
\end{align*}
which sends $\mathcal{O}_{BG}$ to $\mathcal{K}$. 
Then $\mathcal{H}om(\mathcal{K}, \mathcal{K})=i_{0\ast}\mathcal{O}_{BG}$, hence
we have the isomorphism 
$i_{V^{\vee}\ast}\mathcal{H}om(\mathcal{P}, \mathcal{P})
    \stackrel{\cong}{\to} \mathcal{H}om(\mathcal{Q}, \mathcal{Q})$.
We have the commutative diagram: 
\begin{align}\label{diagXc22}
    \xymatrix{
i_{V^{\vee}\ast}\mathcal{H}om(\mathcal{P}, \mathcal{P})
\ar[r]^-{\cong}\ar[d]_-{i_{V^{\vee}\ast}
\mathrm{tr}_{\mathcal{P}}} & 
    \mathcal{H}om(\mathcal{Q}, \mathcal{Q}) \ar[d]^-{\mathrm{tr}_{\mathcal{Q}}} \\
    i_{V^{\vee}\ast} \mathcal{O}_{V^{\vee}/G} \ar[r] &
    \mathcal{O}_{(V^{\vee} \oplus W \oplus W^{\vee})/G},
    }
\end{align}
    where the bottom arrow is the morphism obtained by adjunction and using the isomorphism in 
    $D^b(V^{\vee}/G)$:
    \begin{align*}
        i^{!}_{V^{\vee}}\mathcal{O}_{(V^{\vee} \oplus W \oplus W^{\vee})/G}
        \cong \det W \otimes \det (W^{\vee}(2))[-2\dim W]
        =\mathcal{O}_{V^{\vee}/G}. 
    \end{align*}
    Applying $\pi_{V^{\vee} \oplus W \oplus W^{\vee}\ast}$ to the sheaves in the diagram (\ref{diagXc22}), 
    we obtain the commutative diagram: 
   \begin{align*}
    \xymatrix{
\overline{i}_{V^{\vee}\ast}\pi_{V^{\vee}\ast}\mathcal{H}om(\mathcal{P}, \mathcal{P})
\ar[r]^-{\cong}\ar[d]_-{\overline{i}_{V^{\vee}\ast}\pi_{V^{\vee}\ast}
\mathrm{tr}_{\mathcal{P}}} & 
    \pi_{V^{\vee} \oplus W \oplus W^{\vee}\ast}\mathcal{H}om(\mathcal{Q}, \mathcal{Q}) \ar[d]^-{\pi_{V^{\vee}\oplus W \oplus W^{\vee}\ast}\mathrm{tr}_{\mathcal{Q}}} \\
    \overline{i}_{V^{\vee}\ast} \mathcal{O}_{V^{\vee}\ssslash G} \ar[r] &
    \mathcal{O}_{(V^{\vee} \oplus W \oplus W^{\vee}) \ssslash G}.
    }
\end{align*}
Then by Lemma~\ref{lem:psupport}
applied for the map $p$
together with the commutative diagram (\ref{diagXc2}), 
we have the commutative 
diagram in $D^b((Y\oplus W)\ssslash G)$
 \begin{align*}
    \xymatrix{
\overline{i}_{Y\ast}\pi_{Y\ast}\eta_{\ast}\mathcal{H}om(\mathcal{P}, \mathcal{P})
\ar[r]^-{\cong}\ar[d]_-{\overline{i}_{Y\ast}a_{\mathcal{P}}} & 
    \pi_{Y\oplus W\ast}p_{\ast}\mathcal{H}om(\mathcal{Q}, \mathcal{Q}) \ar[d]^-{a_{\mathcal{Q}}} \\
    \overline{i}_{Y\ast} \mathcal{O}_{Y\ssslash G}[\rk V] \ar[r] &
    \mathcal{O}_{(Y \oplus W) \ssslash G}[\rk V+\dim W].
    }
\end{align*}
The bottom arrow is the natural morphism 
by $\overline{i}_{Y}^{!}\mathcal{O}_{(Y\oplus W)\ssslash G}[\dim W]=\mathcal{O}_{Y\ssslash G}$, see (\ref{dualcomp}). 
The lemma follows from the above commutative diagram 
together with the constructions of $\mathrm{tr}_{\mathcal{E}}$ and $\mathrm{tr}_{\mathcal{E}}'$. 
\end{proof}

\subsection{Local triviality of the Serre functor}\label{subsec:proof2}

In this section, we prove Theorem~\ref{thm:Serre:etale} using the trace map to reduce to the local case discussed in Theorem \ref{thm:Serre}. 

We first explain that objects of $\mathbb{T}$ have nilpotent singular support in the sense of Subsection \ref{subsec:trace}. This result is a version of Lemma \ref{cor:support} and Theorem \ref{prop:catsupp}.
To show it follows from Lemma \ref{cor:support}, we need to
mention a stronger form of the \'etale local description of $M$ from Subsection~\ref{subsec:qsmooth}. For each $y\in M$, recall from Remark \ref{rmk:Equiver} the polystable sheaf $F$, the corresponding doubled quiver $Q^{\circ, d}_y$, dimension vector $\bm{d}$, and good moduli spaces of the reduced stacks of representations of the doubled quiver and of the preprojective algebra of $Q^\circ_y$, respectively:
\[\pi_Y\colon \mathscr{Y}(\bm{d})\to Y(\bm{d}),\, \pi_P\colon \mathcal{P}(\bm{d})\to P(\bm{d}).\]
Then there exists a smooth affine scheme $A$ with an action of the reductive group $G:=G(\bm{d})$, a section $s$ of the vector bundle $V:=\mathcal{O}_A\otimes\mathfrak{g}(\bm{d})^\vee$ with zero locus \[\mathcal{Z}:=s^{-1}(0)/G\subset \mathscr{A}:=A/G,\]
and \'etale maps $e''\colon A\ssslash G\to Y(\bm{d})$ and $M\xleftarrow{e} Z:=s^{-1}(0)\ssslash G\xrightarrow{e'}P(\bm{d})$ 
such that the following diagram is Cartesian, the horizontal maps are \'etale, and the vertical maps are good moduli space maps:
\begin{equation}\label{diaggg1}
    \begin{tikzcd}
        \mathscr{A}\arrow[r, "e''"]\arrow[d, "\pi"]&\mathscr{Y}(\bm{d})\arrow[d, "\pi_Y"]\\
        A\ssslash G\arrow[r, "e''"]& Y(\bm{d}),
    \end{tikzcd}
\end{equation}
and such that both squares in the following diagram are Cartesian, the horizontal maps are \'etale, and the vertical maps are good moduli space maps:
\begin{equation}\label{diaggg2}
    \begin{tikzcd}
        \mathcal{M}\arrow[d, "\pi_M"]& \mathcal{Z}\arrow[d, "\pi"]\arrow[l, "e"']\arrow[r, "e'"]& \mathcal{P}(\bm{d})\arrow[d, "\pi_P"]\\
        M&Z\arrow[l, "e"']\arrow[r, "e'"]& P(\bm{d}).
    \end{tikzcd}
\end{equation} See \cite[Theorem 5.11]{DavPurity} for a proof of the second diagram.
To also obtain the first diagram, one can prove a stronger statement accounting for the derived structure of $\mathfrak{M}$ and $\mathscr{P}(\bm{d})$ as in \cite[Theorem 4.2.3]{halpK32}, because $A$ ($R$ in loc.cit.) can be chosen \'etale over $ \mathrm{Ext}^1_S(F, F)=R_{Q^{\circ,d}}(\bm{d})$, see the proof of loc.cit. Then \eqref{diaggg1} and the right square of \eqref{diaggg2} commute, and the left square of \eqref{diaggg2} commutes by \cite[Theorem 4.2.3]{halpK32}.
For such $e\colon Z\to M$ and for $\mathcal{E}\in D^b(\mathcal{M})$, we denote by $\mathcal{E}|_Z=e^*(\mathcal{E})\in D^b(\mathcal{Z})$.

The upshot of the discussion above is that $y\in M$ is in the image of $e\colon Z\to M$ for a good data $(G, A, V, s)$.


\begin{prop}\label{prop710}
Let $\mathcal{E} \in \mathbb{T}$.
    Then $\mathcal{E}|_Z\in D^b(\mathcal{Z})$ has nilpotent singular support with respect to $(G, A, V, s)$.
\end{prop}

\begin{proof}
    The object $\mathcal{E}|_Z$ is in the subcategory of $D^b(\mathcal{Z})$ classically generated by the image of $e''\colon D^b(\mathcal{P}(\bm{d}))\to D^b(\mathcal{Z})$, see \cite[Subsection 2.11, Subsection 9.2]{PTtop}. Then the claim follows from \cite[Lemma 5.4, Corollary 5.5]{PTquiver}.
\end{proof}

\begin{proof}[Proof of Theorem~\ref{thm:Serre:etale}]

    By Proposition \ref{prop710}, the object
    $\mathcal{E}|_Z \in D^b(\mathcal{Z})$ 
    admits a trace map
    determined by $(G, A, V, s)$,
    see
    the construction of Subsection~\ref{subsec:trace}
    and Definition~\ref{defn:trace}:
    \begin{align}\label{tr:E}
    \mathrm{tr}_Z \colon \pi_{\ast}\mathcal{H}om(\mathcal{E}|_{Z}, \mathcal{E}|_{Z}) \to \mathcal{O}_Z.
    \end{align}
    By the definition of the relative Serre functor, 
    it corresponds to a morphism 
    \begin{align}\label{mor:Serre}
    \phi_Z \colon \mathcal{E}|_{Z} \to S_{\mathbb{T}/M}(\mathcal{E})|_{Z}.
    \end{align}
    By Lemma~\ref{lem:Serre}, 
    it is enough to show that the above 
    morphism is an isomorphism. 

Set $\mathscr{A}=A/G$ and $\mathcal{V}=V/G$. 
    For each point $u \in Z \hookrightarrow A\ssslash G$, 
    let $\widehat{\mathscr{A}}_u$ be the formal 
    fiber of $\mathscr{A} \to A\ssslash G$ 
    at $u$, and (by abuse of notation) denote by 
    $u \in \mathscr{A}$ the unique closed 
    point in the fiber of $\mathscr{A} \to A\ssslash G$
    at $u$. 
    Let $G_u=\mathrm{Aut}(u) \subset G$. 
    By the \'{e}tale slice theorem, 
    there is an isomorphism 
    \begin{align}\label{isom:formal}
 \widehat{\mathscr{A}}_u \cong 
 \widehat{\mathcal{H}}^{0}(\mathbb{T}_{\mathscr{A}}|_{u})/G_u.
    \end{align}
    From the triangle 
    $\mathbb{T}_{\mathcal{Z}} \to \mathbb{T}_{\mathscr{A}}|_{\mathcal{Z}} \to \mathcal{V}|_{\mathcal{Z}}\to \mathbb{T}_{\mathcal{Z}}[1]$, 
    there is 
    an exact sequence of $G_u$-representations
    \begin{align*}
        0 \to \mathcal{H}^0(\mathbb{T}_{\mathcal{Z}|_{u})} \to \mathcal{H}^0(\mathbb{T}_{\mathscr{A}}|_{u}) 
        \stackrel{ds|_{u}}{\to} 
        \mathcal{V}|_{u} \to \mathcal{H}^1(\mathbb{T}_{\mathcal{Z}}|_{u}) \to 0. 
    \end{align*}
    Hence there exist
    isomorphisms of $G_u$-representations 
    \begin{align}\label{decom:TW}
    \mathcal{H}^0(\mathbb{T}_{\mathscr{A}}|_{u}) \cong 
    \mathcal{H}^0(\mathbb{T}_{\mathcal{Z}}|_{u}) \oplus W, 
    \ \mathcal{V}|_u \cong \mathcal{H}^1(\mathbb{T}_{\mathcal{Z}}|_{u}) \oplus W
\end{align}
for some $G_u$-representation $W$
such that $ds|_{u}=(0, \id_W)$. 

First assume that $u$ corresponds to a point in the 
deepest stratum, so that 
\begin{align}\label{isom:Tgl}\mathcal{H}^0(\mathbb{T}_{\mathcal{Z}}|_{u})=\mathfrak{gl}(d)^{\oplus 2g}, \ \mathcal{H}^1(\mathbb{T}_{\mathcal{Z}}|_{u})
=\mathfrak{gl}(d)_0, \text{ and } G_u=GL(d).\end{align}
Let $\mu_0 \colon \mathfrak{gl}(d)^{\oplus 2g} \to \mathfrak{gl}(d)_0$ be the 
moment map (\ref{mu0:trace}). 
Note that the zero locus of $s|_{\widehat{\mathscr{A}}_u}$
is isomorphic to the formal fiber of $\mu_0^{-1}(0)/GL(d) \to \mu_0^{-1}(0)\ssslash GL(d)$
at the origin, see Lemma~\ref{lem:py}. 
As both of $s$ and $\mu_0$ are regular sections, 
by a formal coordinate change we 
may replace the isomorphism (\ref{isom:formal})
and assume that $s|_{\widehat{\mathscr{A}}_u}$
corresponds to the map 
\begin{align*}
(\mu_0, \id_W) \colon \mathfrak{gl}(d)^{\oplus 2g} \oplus W \to 
\mathfrak{gl}(d)_0 \oplus W
\end{align*}
under the decompositions (\ref{decom:TW}) and isomorphisms (\ref{isom:Tgl}). 
By Lemmas~\ref{compare:trace0} and \ref{lem:trace=}, the trace map (\ref{tr:E}) 
pulled back via 
$\widehat{Z}_u:=\Spec \widehat{\mathcal{O}}_{Z, u} \to Z$
coincides with 
the trace map determined by the good data
$(GL(d), \mathfrak{gl}(d)^{\oplus 2g}, \mathfrak{gl}(d)_0, \mu_0)$.
Then from Theorem~\ref{thm:Serre}, 
the map (\ref{mor:Serre}) 
is an isomorphism at $\widehat{Z}_u$. 

In general,
let
$p \in \mathfrak{gl}(d)^{\oplus 2g}\ssslash GL(d)$
be a point 
corresponding to $u$ as in Lemma~\ref{lem:py}, i.e. 
there is an equivalence 
\begin{align}\label{equiv:MPp}
\widehat{\mathcal{Z}}_{u} \simeq 
\widehat{\mathscr{P}}(d)_p
\end{align}
for the $g$-loop quiver $Q^{\circ}$. 
Let $\mathscr{Y}(d)=\mathfrak{gl}(d)^{\oplus 2g}/GL(d)$ be the moduli stack of representations of the doubled quiver of $Q^\circ$. We also denote by 
$p \in \mathscr{Y}(d)$
the unique closed point in the fiber of 
$\mathscr{Y}(d) \to \mathfrak{gl}(d)^{\oplus 2g}\ssslash GL(d)$
at $p$. 
Then we have decompositions 
\begin{align}\label{decom:W2}
\mathcal{H}^0(\mathbb{T}_{\mathscr{Y}(d)}|_{p})=\mathcal{H}^0(\mathbb{T}_{\mathcal{Z}}|_{u})
    \oplus W', \    \mathfrak{gl}(d)_0=\mathcal{H}^1(\mathbb{T}_{\mathcal{Z}}|_{u})
    \oplus W'
\end{align}
for some $G_u$-representation $W'$.

By Lemma~\ref{compare:trace0} and the isomorphism 
(\ref{isom:formal}), the trace map 
(\ref{tr:E}) at $\widehat{Z}_u$
equals the trace map determined by 
$(G_u, \mathcal{H}^0(\mathbb{T}_{\mathscr{A}}|_{u}), \mathscr{V}|_{u}, s|_{\widehat{\mathscr{A}}_u})$.  
Then by the decomposition (\ref{decom:TW}) and Lemma~\ref{lem:trace=}, 
it also 
equals the trace map
determined by the good data
$(G_u, \mathcal{H}^0(\mathbb{T}_{\mathcal{Z}}|_{u}), \mathcal{H}^1(\mathbb{T}_{\mathcal{Z}}|_{u}), \kappa)$. 
Then by (\ref{decom:W2}) and Lemma~\ref{lem:trace=}, 
under the equivalence (\ref{equiv:MPp})
the trace map (\ref{tr:E}) at $\widehat{Z}_u$ also equals the trace map determined by 
$(G_p, \mathcal{H}^0(\mathbb{T}_{\mathscr{Y}(d)}|_{p}), \mathfrak{gl}(d)_0, \mu_0)$, 
which in turn equals the trace map determined by 
$(GL(d), \mathfrak{gl}(d)^{\oplus 2g}, \mathfrak{gl}(d)_0, \mu_0)$
at $p$ by Lemma~\ref{compare:trace0}. 
Again by Theorem~\ref{thm:Serre}, the map (\ref{mor:Serre}) 
is an isomorphism on $\widehat{Z}_u$. 
Therefore (\ref{mor:Serre}) is an 
isomorphism at any 
point $u\in Z$, hence it is an isomorphism. 
\end{proof}

The proof of Theorem~\ref{thm:Serre:etale}
also implies the following: 
\begin{cor}\label{cor:isomcoh}
In the situation of Theorem~\ref{thm:Serre:etale}, 
for each $\mathcal{E} \in \mathbb{T}$ there are isomorphisms:
\[\mathcal{H}^i(S_{\mathbb{T}}(\mathcal{E}))
\cong \mathcal{H}^i(\mathcal{E}[\dim M])\]
for all $i\in \mathbb{Z}$. 
In particular, if there exists $k\in \mathbb{Z}$ such that $\mathcal{E}$ is an object in $\mathbb{T} \cap \Coh(\mathcal{M})[k]$, then
$S_{\mathbb{T}}(\mathcal{E}) \cong \mathcal{E}[\dim M]$. 
\end{cor}
\begin{proof}
    Let $\mathrm{tr}_Z$ be the morphism in (\ref{tr:E}). 
    For another \'etale morphism $Z' \to M$, 
    the proof of Theorem~\ref{thm:Serre} shows that 
    the morphism 
    \begin{align*}
        \mathrm{tr}_{Z}|_{Z \times_M Z'}-
        \mathrm{tr}_{Z'}|_{Z\times_M Z'} \colon 
        \pi_{\ast}\mathcal{H}om(\mathcal{E}|_{Z\times_M Z'}, \mathcal{E}|_{Z\times_M Z'}) \to
        \mathcal{O}_{Z \times_M Z'}
    \end{align*}
    is a zero map formally locally at any point in 
    $Z\times_M Z'$. 
    Thus for the morphism $\phi_Z$ in (\ref{mor:Serre}),
    the morphism 
    \begin{align*}
        \phi_Z|_{Z\times_M Z'}-\phi_{Z'}|_{Z\times_M Z'} 
        \colon \mathcal{E}|_{Z\times_M Z'} \to 
        S_{\mathbb{T}/M}(\mathcal{E})|_{Z\times_M Z'}
    \end{align*}
    is a zero map formally locally at each 
    point in $Z\times_M Z'$. 
    Therefore for each $i \in \mathbb{Z}$, the isomorphism
    \begin{align*}
        \mathcal{H}^i(\phi_Z) \colon 
        \mathcal{H}^i(\mathcal{E}|_{Z}) \stackrel{\cong}{\to} \mathcal{H}^i(S_{\mathbb{T}/M}(\mathcal{E})|_{Z})
    \end{align*}
    glues to give 
    an isomorphism $\mathcal{H}^i(\mathcal{E})\cong \mathcal{H}^i(S_{\mathbb{T}/M}(\mathcal{E}))$. 
    Then the corollary follows from Lemma~\ref{lem:Serre}.
\end{proof}

We also have the following: 
\begin{cor}\label{cor:serreperf}
In the situation of Theorem~\ref{thm:Serre}, 
for $\mathcal{E} \in \mathbb{T} \cap \mathrm{Perf}(\mathcal{M})$, we have 
    $S_{\mathbb{T}}(\mathcal{E}) \cong \mathcal{E}[\dim M]$.
\end{cor}
\begin{proof}
    As $\mathcal{E}$ is perfect, there is a 
    trace map $\mathcal{H}om_{\mathcal{M}}(\mathcal{E}, \mathcal{E}) \to \mathcal{O}_{\mathcal{M}}$, 
    thus its push-forward $\pi_{\ast}$ gives 
    a morphism 
    \begin{align*}
        \pi_{\ast}\mathcal{H}om_{\mathcal{M}}(\mathcal{E}, \mathcal{E}) \to \mathcal{O}_M. 
    \end{align*}
    The above morphism corresponds to 
    $\phi \colon \mathcal{E} \to S_{\mathbb{T}/M}(\mathcal{E})$. 
    By Lemma~\ref{lem:trperf}, the above morphism coincides with (\ref{mor:Serre})
    on each \'etale map $Z \to M$, 
    thus $\phi$ is an isomorphism.    
        Then the corollary follows from Lemma~\ref{lem:Serre}.
    \end{proof}
\section{Topological K-theory of quasi-BPS categories for K3 surfaces}\label{sec:topK}
\subsection{Statement of the main result}\label{subsec81}
In this section, we prove Theorem~\ref{intro:thm:K}
using the computation of topological K-theory of quasi-BPS categories of preprojective algebras from~\cite{PTtop}. 
We actually compute the topological K-theory of quasi-BPS categories for all weights $w\in\mathbb{Z}$, not only in the case of $w$ coprime with $v=dv_0$, see Theorem \ref{thmKtop}. 

 For a stack $\X$, 
we denote by $D_{\rm{con}}(\X)$ the 
bounded derived category of constructible sheaves 
on $\X$
and $\mathrm{Perv}(\X) \subset D_{\rm{con}}(\X)$
the subcategory of perverse sheaves \cite{MR2480756}. 
We denote by $D^+_{\mathrm{con}}(\X)$ the category of locally bounded below complexes of constructible sheaves on $\X$: if $\X$ is connected, then $D^+_{\mathrm{con}}(\X)$ is the limit of the diagram of categories $D_n:=D^b_{\mathrm{con}}(\X)$ for all $n\in\mathbb{N}$ and for the functors ${}^p\tau^{\leq n'}\colon D^n\to D^{n'}$; for general $\X$, we have $D^+_{\mathrm{con}}(\X)=\prod_{\X'\in \pi_0(\X)} D^+_{\mathrm{con}}(\X')$.

In this section, we assume that $d\geq 2$, $g\geq 2$, and that
$\sigma \in \mathrm{Stab}(S)$ corresponds to a Gieseker stability condition 
for an ample divisor $H$, see Proposition \ref{prop:LV} and Corollary \ref{cor:lv}. The reason we restrict to Gieseker stability conditions is that, in this case, $\mathcal{M}$ is a global quotient stack and one can construct a cycle map as in \cite[Section 3]{PTtop}.
We fix $v=dv_0$ and $w\in\mathbb{Z}$.

In Subsection \ref{subsub:bpsk3}, we recall the definition of the BPS sheaf 
\[\mathcal{BPS}_v=\mathcal{BPS}^{\sigma}_S(v)\in \mathrm{Perv}\left(M^{\sigma}_S(v)\right).\] 
For a partition $A=(d_i)_{i=1}^k$ of $d$, define the perverse sheaf $\mathcal{BPS}_A$
on $M_S^{\sigma}(v)$ 
to be 
\begin{align*}
    \mathcal{BPS}_A :=\left(\oplus_{\ast}\boxtimes_{i=1}^k 
    \mathcal{BPS}_{d_i v_0}\right)^{\mathfrak{S}_A},
\end{align*}
where $\mathfrak{S}_A \subset \mathfrak{S}_k$ is 
the subgroup of permutations $\sigma \in \mathfrak{S}_k$
such that $d_i=d_{\sigma(i)}$, 
and $\oplus$ is the addition map 
\begin{align*}
    \oplus \colon M_S^{\sigma}(d_1 v_0) \times \cdots 
    \times M_S^{\sigma}(d_k v_0) \to M_S^{\sigma}(d v_0). 
\end{align*}
For $w\in\mathbb{Z}$, 
let $S^d_w$
be the set of partitions of $d$ 
from~\cite[Subsection~6.1.2]{PTtop}. 
From~\cite[Proposition~8.8]{PTtop}, 
it consists of partitions 
$A=(d_i)_{i=1}^k$
such that, for all $1\leq i\leq k$, $w_i:= d_i w/d$ is an integer, thus $S^d_w$ is in bijection with the set of partitions of $\gcd(d,w)$. 
We set  
\[\mathcal{BPS}_{v,w}:=\bigoplus_{A\in S^d_v}\mathcal{BPS}_A.\]

For a dg-category $\mathscr{D}$, we denote by $K^{\mathrm{top}}(\mathscr{D})$ the topological K-theory spectrum as defined by Blanc \cite{Blanc}. We consider its (rational) homotopy groups:
\[K^{\mathrm{top}}_i(\mathscr{D}):=\left(\pi_i K^{\mathrm{top}}(\mathscr{D})\right)\otimes_\mathbb{Z}\mathbb{Q}.\]
For a review of (and references on) topological K-theory, see \cite[Subsection 2.4]{PTtop}. If $\mathcal{M}$ is a quotient stack, we denote by $G^{\mathrm{top}}(\mathcal{M})$ the (rational) K-homology of $\mathcal{M}$. Then, by \cite{HLP}, we have that $G^{\mathrm{top}}(\mathcal{M})=K^{\mathrm{top}}(D^b(\mathcal{M}))$.

For a $\mathbb{Z}$-graded vector space $H^{\ast}$ and $i\in\mathbb{Z}$, 
let $\widetilde{H}^{i}:=\prod_{j \in \mathbb{Z}}H^{i+2j}$. 
In this section, we prove the following result, 
which implies 
Theorem~\ref{intro:thm:K}
as a special case. Note that the second isomorphism is not canonical, see Theorem~\ref{thmKtopgraded} for a statement involving canonical isomorphisms: 

\begin{thm}\label{thmKtop}
For $i\in \mathbb{Z}$, there exist
isomorphisms of $\mathbb{Q}$-vector spaces
\begin{align}\label{isom:thmKtop}
    K_{i}^{\rm{top}}(\mathbb{T}_S^{\sigma}(v)_w^{\rm{red}})
    \stackrel{\cong}{\to}  K_{i}^{\rm{top}}(\mathbb{T}_S^{\sigma}(v)_w)
    \cong \widetilde{H}^{i}\left(M^{\sigma}_S(v), \mathcal{BPS}_{v,w}\right). 
\end{align}
   \end{thm}



\subsection{BPS sheaves for K3 surfaces}

As in the case of symmetric quivers with potential or preprojective algebras, the
BPS cohomology for K3 surfaces is the ``primitive" part of the Hall algebra of $S$ for the chosen stability condition, and is computed as the cohomology of the BPS sheaf. 

In this section, we recall the definition of BPS sheaves for K3 surfaces due to Davison--Hennecart--Schlegel Mejia \cite{DHSM} and we compare these sheaves with BPS sheaves for preprojective algebras.

\subsubsection{BPS sheaves via intersection complexes}\label{subsub:bpsk3}

Let $\mathbb{D}$ be the Verdier duality functor on $D^b_{\mathrm{con}}(\mathcal{M}^{\sigma}_S(v))$ and let $D_d:=\mathbb{D}\mathbb{Q}\in D^b_{\mathrm{con}}(\mathcal{M}^{\sigma}_S(v))$ be the dualizing complex on 
$\mathcal{M}^{\sigma}_S(v)=\mathcal{M}^{\sigma}_S(dv_0)$. 
Recall the good moduli space map 
\[\pi_d:=\pi\colon\mathcal{M}\to M.\]
The BBDG decomposition theorem holds for $\pi_{d*}
D_d\in D^+_{\mathrm{con}}(M^{\sigma}_S(v))$, see \cite[Theorem C]{Dav}. 
The BPS sheaf of $M^{\sigma}_S(v)$ is a certain direct summand of the zeroth perverse truncation (which itself is a perverse sheaf on $M^{\sigma}_S(v)$, see loc. cit.):
\begin{equation}\label{eq:BPS0}
{}^p\tau^{\leq 0}\pi_{d*}D_d\in \mathrm{Perv}(M_S^{\sigma}(v)).
\end{equation}
We now explain the definition of the BPS sheaf. 
The cohomological Hall product $m=p_*q^*$ for the maps $p,q$ in \eqref{PortaSala} induces an algebra structure on the $\mathbb{N}$-graded complex 
\[\mathscr{A}:=\bigoplus_{d\in\mathbb{N}} {}^p\tau^{\leq 0}\pi_{d*}D_d \in 
\bigoplus_{d\in \mathbb{N}} D^+_{\mathrm{con}}(M_S^{\sigma}(dv_0))
.\] 
There is a natural map
\[\mathrm{IC}_{M^{\sigma}_S(v)}\to {}^p\tau^{\leq 0}\pi_{d*}D_d.\]
The main theorem of Davison--Hennecart--Schlegel Mejia \cite[Theorem C]{DHSM} says that the induced map from the free algebra generated by the intersection complexes is isomorphic to $\mathscr{A}$:
\begin{equation*}    \mathrm{Free}\left(\bigoplus_{d\in\mathbb{N}}\mathrm{IC}_{M^{\sigma}_S(dv_0)}\right)\xrightarrow{\sim} \mathscr{A}.
\end{equation*}
The BPS sheaves \[\mathcal{BPS}^{\sigma}_S(v)=\mathcal{BPS}^{\sigma}_S(dv_0)\in \mathrm{Perv}\left(M^{\sigma}_S(v)\right)\] are defined via the free Lie algebra on the intersection complexes  
\begin{equation}\label{defBPS}
\mathrm{Free}_{\mathrm{Lie}}\left(\bigoplus_{d\in\mathbb{N}}\mathrm{IC}_{M^{\sigma}_S(v)}\right)=:\bigoplus_{d\in\mathbb{N}}\mathcal{BPS}^{\sigma}_S(dv_0).
\end{equation}
We obtain that:
\begin{equation}\label{BPSZ}
\mathrm{Sym}\left(\bigoplus_{d\in\mathbb{N}}\mathcal{BPS}^{\sigma}_S(dv_0)\right)\xrightarrow{\sim}\mathscr{A}.
\end{equation}


A precise formulation for the heuristics that the BPS cohomology is the ``primitive" part of the Hall algebra is the following:
the relative Hall algebra of $S$ for the multiples of the Mukai vector $v_0$ and stability condition $\sigma$ has a PBW decomposition in terms of BPS sheaves:
\begin{equation}\label{Hall}
\mathrm{Sym}\left(\bigoplus_{d\in\mathbb{N}}\mathcal{BPS}^{\sigma}_S(dv_0)\otimes H^\cdot(B\mathbb{C}^*)\right)\xrightarrow{\sim}\mathscr{H}:=\bigoplus_{d\in\mathbb{N}}\pi_{d*}D_{d},
\end{equation}
see \cite[Theorem 1.5]{DHSM}, and note that the above isomorphism is of constructible sheaves, not of relative algebras. There is also a version for the absolute Hall algebra \cite[Corollary 1.6]{DHSM}. The above PBW theorem is the analogue for K3 surfaces of the Davison--Meinhardt PBW theorem for cohomological Hall algebras of quivers with potential \cite{DM}.
The results in \cite{DHSM} cited above hold by a computation of all the simple summands of $\pi_{d*}D_d$, which satisfies a version of the BBDG/ Saito decomposition theorem due to Davison \cite{DavPurity}.

\subsubsection{The moduli stack of semistable sheaves on a K3 surface via preprojective algebras}\label{localdescription}
One can describe the map 
$\pi_d \colon \mathcal{M}_S^{\sigma}(v) \to M_S^{\sigma}(v)$
étale, formally, or analytically locally on the target via preprojective algebras \cite{Sacca0}, \cite[Sections 4 and 5]{DavPurity}, \cite[Theorem 4.3.2]{halpK32}, \cite{Todstack}, Subsections \ref{subsection:etale} and \ref{subsec:proof2}. 

We will use the setting of Subsection \ref{subsec:proof2}, see diagrams \eqref{diaggg1} and \eqref{diaggg2}. 
We will continue with the notation from Subsection \ref{subsec:proof2}.

The quiver $Q^\circ_y$ is totally negative in the sense of \cite[Section 1.2.3]{DHSM}, see \cite{KaLeSo}.
Thus the results in \cite{DHSM} about construction of BPS sheaves via intersection complexes apply, so 
the BPS sheaves $\mathcal{BPS}^p(\bm{d})$ of the preprojective algebras of the quiver $Q^\circ_y$ have a similar description via intersection complexes \eqref{defBPS}.
Then the maps $e$ and $e'$ induce isomorphisms:
\begin{equation}\label{localBPS}
e^*\left(\mathcal{BPS}^{\sigma}_S(v)\right)=e'^*\left(\mathcal{BPS}^p(\bm{d})\right).
\end{equation}
\begin{remark}
If we are interested in a local analytic description of $\mathcal{M}^{\sigma}_S(v)$, then it is possible to choose $Y$ an analytic open subset of $P(\bm{d})$ and $M^{\sigma}_S(v)$, that is, we may assume that $e$ and $e'$ are open inclusions of analytic sets. Thus, locally analytically near $p$, the preimage of the map $\pi_d$ is isomorphic to the preimage of the map $\pi_P$. 
\end{remark}
\subsection{Topological K-theory and étale covers}
We use the shorthand notations
$M=M^{\sigma}_S(v)$, $\mathcal{M}:=\mathcal{M}^{\sigma}_S(v)$, 
$\mathfrak{M}=\mathfrak{M}^{\sigma}_S(v)$
and 
\[\mathbb{T}(M)^{\mathrm{red}}:=\mathbb{T}^{\sigma}_S(v)^{\mathrm{red}}_w, \ \mathbb{T}(M):=\mathbb{T}^{\sigma}_S(v)_w.\]
We write the semiorthogonal decomposition for $\mathcal{M}$ as:
\begin{equation}\label{SODredAT}
    D^b(\mathcal{M})=\langle \mathbb{A}(M)^{\mathrm{red}}, \mathbb{T}(M)^{\mathrm{red}}\rangle.
\end{equation}
By the following lemma, 
it suffices to prove Theorem \ref{thmKtop} for $\mathbb{T}(M)^{\mathrm{red}}$. The argument for $\mathbb{T}(M)$ is the same, but we prefer working with the stack $\mathcal{M}$ because the good moduli space map is defined from $\mathcal{M}$.
\begin{lemma}\label{lem:topisom}
The closed immersion 
$\iota \colon \mathcal{M} \hookrightarrow \mathfrak{M}$
induces the isomorphism 
\begin{equation*}\label{isoKtopTTred}
    \iota_{\ast}\colon G^{\mathrm{top}}_\bullet(\mathbb{T}(M)^{\mathrm{red}})\xrightarrow{\sim} G^{\mathrm{top}}_\bullet(\mathbb{T}(M)).
\end{equation*}
\end{lemma}
\begin{proof}
We have the isomorphism 
\begin{align*}\iota_{\ast}\colon G^{\mathrm{top}}_\bullet(\mathcal{M})\xrightarrow{\sim} G^{\mathrm{top}}_\bullet(\mathfrak{M})
\end{align*}
since both spaces have the same underlying topological space. 
     Then the lemma holds since 
     $\iota_{\ast}$ sends $\mathbb{T}(M)^{\mathrm{red}}$ to 
$\mathbb{T}(M)$.
\end{proof}

The semiorthogonal decomposition in Theorem~\ref{thm:sodK32} holds \'etale locally over $M$ by
\cite[Section~9]{PTtop} and the diagram \eqref{diaggg2}.
Indeed, let $R\to M$ be an \'etale map which factors through $R\xrightarrow{h} Z\to M$ as in \eqref{diaggg2}.
Let $\mathcal{R}:=\mathcal{M}^\sigma_S(v)\times_{M^\sigma_S(v)} R$.
By~\cite[Section~9]{PTtop}, there is a semiorthogonal decomposition:
\begin{equation}\label{SODZ}
D^b(\mathcal{R})=\langle \mathbb{A}(R)^{\mathrm{red}}, \mathbb{T}(R)^{\mathrm{red}}\rangle
\end{equation}
such that for an 
 \'etale map $b\colon R'\to R$, 
 the pull-back $b^{\ast}$ induce 
 functors 
 \begin{align*}
     b^{\ast} \colon \mathbb{A}(R)^{\mathrm{red}} \to 
     \mathbb{A}(R')^{\mathrm{red}}, \ 
     b^{\ast} \colon \mathbb{T}(R)^{\mathrm{red}} \to \mathbb{T}(R')^{\mathrm{red}}. 
 \end{align*}
 Consider étale covers 
\[U=(Z\xrightarrow{e} M),\, \mathcal{U}=(\mathcal{Z}\xrightarrow{e}\mathcal{M})
\]
generated by the \'etale covers $Z\to M$ as in \eqref{diaggg2}. 
Consider the presheaves of spectra
$\mathcal{K}$, $\mathcal{A}$ and $\mathcal{T}$ on $U$ defined as follows: 
for $(R\xrightarrow{e}M)\in U$ (and dropping $e$ from the notation), we have:
\[\mathcal{K}(R)=G^{\mathrm{top}}(\mathcal{R}),\, \mathcal{A}(R)=K^{\mathrm{top}}(\mathbb{A}(R)^{\mathrm{red}}),\, \mathcal{T}(R)=K^{\mathrm{top}}(\mathbb{T}(R)^{\mathrm{red}}).\]
By~\cite[Theorem~9.2]{PTtop}, there is a direct sum of presheaves of spectra on $U$:
\begin{equation}\label{sum}
\mathcal{K}=\mathcal{A}\oplus\mathcal{T}.
\end{equation}

Let $\mathcal{F}$ be a presheaf of spectra and consider a cover $(Z_i\xrightarrow{e} M)_{i\in I}$ as in diagram \eqref{diaggg2} for a set $I$.
Consider the cosimplicial sheaf of spectra: 
\[\xymatrix{ \prod_{i\in I}\mathcal{F}(Z_i) \ar@<1ex>[r] \ar@<-1ex>[r] &  \prod_{i,j\in I}\mathcal{F}(Z_i\times_M Z_j) \ar@<0ex>[l] \ar@<2ex>[r] \ar@<0ex>[r] \ar@<-2ex>[r] &  \cdots \ar@<1ex>[l] \ar@<-1ex>[l] },\] which can be used to compute the cohomology of the sheafification of $\mathcal{F}$, and which can be also related to Cech cohomology $\Check{\mathrm{H}}(U, \mathcal{F})$, see \cite[Definition 1.33, Remark 1.38]{Thomason3}. 
There is a natural map
\begin{equation}\label{mapzero}
\eta_\mathcal{F}\colon\mathcal{F}(M)\to \Check{\mathrm{H}}(U, \mathcal{F}).
\end{equation}
For a presheaf of spectra $\mathcal{F}$ and for $i\in\mathbb{Z}$, denote by $\mathcal{F}_i=\pi_i\mathcal{F}$ the corresponding presheaf of abelian groups and by $\mathcal{F}_i^s$ the sheafification of $\mathcal{F}_i$.
\begin{prop}\label{etalecomputationtopK}
    The map \eqref{mapzero} induces a weak equivalence of spectra:
    \[G^{\mathrm{top}}(\mathcal{M})=\mathcal{K}(M)\xrightarrow{\sim} \Check{\mathrm{H}}(U, \mathcal{K}).\]
    Thus there is a spectral sequence 
    \begin{equation}\label{ss1}
    E_{p,q}:=\Check{\mathrm{H}}^p(U, \mathcal{K}^s_q)\implies G^{\mathrm{top}}_{q-p}(\mathcal{M}).
    \end{equation}
\end{prop}

\begin{proof}
    The above statement is proved for (rational) algebraic K-theory by Thomason in \cite[theorem 2.15, Corollary 2.16, Corollary 2.17]{Thomason3}. The proof in loc. cit. also applies to the easier case of (rational) topological K-theory. Indeed, pushforward maps along étale maps exist on topological K-theory, so topological K-theory satisfies the weak transfer property \cite[Definition 2.12]{Thomason3}, thus topological K-theory has etale cohomological descent \cite[Proposition 2.14]{Thomason3}, and then the statement of \cite[Theorem 2.15]{Thomason3} also holds for topological K-theory. 
    
    Alternatively, the analogous statement holds for singular cohomology \cite[Chapter III, Theorem 2.17]{MR0559531}, then by a standard dévissage argument also for Borel-Moore homology, and then the statement for topological K-theory can be obtained using~\cite[Proposition~3.1]{PTtop}. 
\end{proof}
\begin{remark}
Even more, the presheaf $\mathcal{K}$ is a sheaf of spectra. Indeed, let $\mathcal{K}^s$ be the sheafification of $\mathcal{K}$. For any $(E\to M)\in \mathrm{Et}(M)$, we can compute the sections $\mathcal{K}^s(E)$ using Cech cohomology for a cover $U_E$ of $E$:
\[\mathcal{K}^s(E)\xrightarrow{\sim} \Check{\mathrm{H}}(U_E, \mathcal{K}).\]
By the same argument as in Proposition \ref{etalecomputationtopK}, we also have that $\mathcal{K}(E)\xrightarrow{\sim} \Check{\mathrm{H}}(U_E, \mathcal{K})$, thus $\mathcal{K}$ is indeed a sheaf.
\end{remark}
\begin{cor}\label{etalecomputationtopK2}
    The map \eqref{mapzero} induces a weak equivalence:
    \[K^{\mathrm{top}}(\mathbb{T}(M)^{\mathrm{red}})\xrightarrow{\sim} \Check{\mathrm{H}}(U, \mathcal{T}).\]
    Thus there is a spectral sequence 
    \[\Check{\mathrm{H}}^p(U, \mathcal{T}^s_q)\implies K^{\mathrm{top}}_{q-p}(\mathbb{T}(M)^{\mathrm{red}}).\]
\end{cor}

\begin{proof}
    The map $\eta_\mathcal{K}=\eta_\mathcal{A}\oplus\eta_\mathcal{T}$ is an isomorphism by Proposition \ref{etalecomputationtopK}, so $\eta_\mathcal{T}$ is also an isomorphism.
\end{proof}

Let $\mathcal{H}_q$ be the presheaf of $\mathbb{Q}$-vector spaces such that, for $(Z\xrightarrow{e}M)\in U$, we have \[\mathcal{H}_q(Z)=H^{\mathrm{BM}}_q(\mathcal{Z}).\]
Then $\mathcal{H}_q=\pi_q\mathcal{H}$, where $\mathcal{H}$ is the presheaf of Eilenberg-MacLane spectra. As for $\mathcal{K}$, the presheaf $\mathcal{H}$ is actually a sheaf.
There is an spectral sequence analogous to \eqref{ss1}:
\begin{equation}\label{ss2}
E'_{p,q}:=\Check{\mathrm{H}}^p(U, \mathcal{H}^s_q)\implies H^{\mathrm{BM}}_{q-p}(\mathfrak{M})=H^{\mathrm{BM}}_{q-p}(\mathcal{M}),
\end{equation}
see the proof of Proposition \ref{etalecomputationtopK}.

\begin{prop}\label{prop78}
We have $\mathcal{K}_1=\widetilde{\mathcal{H}}^s_1=0$. Thus
    the terms $E_{p,q}$ from \eqref{ss1} and $E'_{p,q}$ from \eqref{ss2} vanish for $q$ odd.
\end{prop}

\begin{proof}
By~\cite[Proposition~3.1]{PTtop}, it suffices to check that $\widetilde{\mathcal{H}}^s_{1}=0$. 
It suffices to check that the stalks of $\widetilde{\mathcal{H}}^s_{1}$ over $y\in M$ are zero.  We can define spectra $\mathcal{H}^{\mathrm{an}}$ in the analytic topology, and $\mathcal{H}^{\mathrm{an}}_y\cong \mathcal{H}_y$ for any $y\in M$, which follows as in \cite[Chapter III, Theorem 3.12]{MR0559531}. It thus suffices to check that $H^{\mathrm{BM}}_{\mathrm{odd}}(V)=0$ for a system of open sets $V\subset M$. 
By the local description from Subsection \ref{localdescription}, we may assume that $V\subset P(\bm{d})$ is an open neighborhood of the origin, where $P(\bm{d})$ is the coarse space of dimension $\bm{d}$ representations of the preprojective algebra of the Ext-quiver $Q^\circ_y$.

Consider the action of $\mathbb{C}^*$ on spaces of representations of the double quiver of $Q^\circ_y$, which acts with weight one.
It induces a scaling action on $P(\bm{d})$
which contracts it onto $0$. 
We can choose a system of open sets $0\in V\subset P(\bm{d})$ such that $V$ is homeomorphic to $P(\bm{d})$ and $\pi_{P}^{-1}(V)$ is homeomorphic to $\mathcal{P}(\bm{d})$. 
It then suffices to check that 
$H^{\mathrm{BM}}_{\mathrm{odd}}(\mathcal{P}(\bm{d}))=0$, 
which was proved by Davison in \cite[Theorem A]{Dav}. 
\end{proof}

Let $i\in\mathbb{Z}$. 
Consider the Chern character for the quotient stack $\mathcal{M}$:
\[\mathrm{ch}\colon G^{\mathrm{top}}_i(\mathcal{M})\to \widetilde{H}^{\mathrm{BM}}_i(\mathcal{M}),\] see \cite[Subsection 3.1]{PTtop}. There are analogous such Chern characters for $\mathcal{Z}$ with $(e\colon \mathcal{Z}\to \mathcal{M})\in\mathcal{U}$.
By Proposition \ref{prop78}, there are compatible spectral sequences with terms for bi-degrees:
\begin{equation}\label{diagsect8}
    \begin{tikzcd}\Check{\mathrm{H}}^{2q-i}(U, \mathcal{T}^s_{2q})\arrow[r, Rightarrow]\arrow[d, hook]& K_{i}^{\mathrm{top}}(\mathbb{T}(M))\arrow[d, hook]\\
    \Check{\mathrm{H}}^{2q-i}(U, \mathcal{K}^s_{2q})\arrow[r, Rightarrow]\arrow[d, "\mathrm{ch}"]& G_{i}^{\mathrm{top}}(\mathcal{M})\arrow[d, "\mathrm{ch}"]\\
        \Check{\mathrm{H}}^{2a+i}(U, \widetilde{\mathcal{H}}^s_{2q})\arrow[r, Rightarrow]& \widetilde{H}^{\mathrm{BM}}_{i}(\mathcal{M}).
    \end{tikzcd}
\end{equation}
Let 
$F_{\bullet} \mathcal{K}_{2g}^s\subset \mathcal{K}_{2g}^s$ and $F_{\bullet} \mathcal{T}_{2g}^s\subset \mathcal{T}_{2g}^s$ be the increasing filtrations 
defined by 
\begin{align*}
   F_j \mathcal{K}_{2q}^s =\ch^{-1}(\mathcal{H}_{\leq 2q+2j}^s), \ 
   F_j \mathcal{T}_{2g}^s=\mathcal{T}_{2g}^s \cap  F_j \mathcal{K}_{2q}^s.
\end{align*} 
We denote by $\mathrm{gr}_{\bullet}$ the associated graded 
with respect to the above filtrations. 
We obtain compatible spectral sequences:
\begin{equation}\label{diagss3}
    \begin{tikzcd}\Check{\mathrm{H}}^{2q-i}(U, \mathrm{gr}_{j}\mathcal{T}^s_{2q})\arrow[r, Rightarrow]\arrow[d, hook, "\alpha"]& \mathrm{gr}_{j} K_{i}^{\mathrm{top}}(\mathbb{T}(M))\arrow[d, hook]\\
    \Check{\mathrm{H}}^{2q-i}(U, \mathrm{gr}_{j}\mathcal{K}^s_{2q})\arrow[r, Rightarrow]\arrow[d, "\mathrm{c}"]& \mathrm{gr}_{j}G_{i}^{\mathrm{top}}(\mathcal{M})\arrow[d, "\mathrm{c}"]\\
        \Check{\mathrm{H}}^{2q+i}(U, \mathcal{H}^s_{2q+2j})\arrow[r, Rightarrow, "d"]& H^{\mathrm{BM}}_{i+2j}(\mathcal{M}),
    \end{tikzcd}
\end{equation} where the cycle maps 
$\mathrm{c}$ are isomorphisms by~\cite[Proposition~3.1]{PTtop}. 

\begin{prop}\label{prop79}
    The image of the map $d\mathrm{c}\alpha$ is $H^{-i-2j}(\mathcal{M}, \mathcal{BPS}_{v,w})$. 
\end{prop}

\begin{proof}
    By~\cite[Theorem~9.2]{PTtop}, the image of $\mathrm{c}\alpha$ is the bi-graded complex with terms $E^\circ_{p,q}:=\Check{\mathrm{H}}^{2q+i}(U, \mathcal{H}^{-2q-2j}(\mathcal{BPS}_{v,w}))$. 
    The restriction of $d$ to $E^\circ_{p,q}$ corresponds to the Čech spectral sequence for $\mathcal{BPS}_{v,w}$:
    \[d\colon E^{\circ}_{p,q}\implies H^{-i-2j}(M, \mathcal{BPS}_{v,w}).\] The conclusion then follows.
\end{proof}

We obtain the following: 
\begin{thm}\label{thmKtopgraded}
For any $i\in\mathbb{Z}$,
    there is an isomorphism
    \[\mathrm{c}\colon \mathrm{gr}_jK^{\mathrm{top}}_i(\mathbb{T}^{\sigma}_S(v)^{\mathrm{red}}_w)\xrightarrow{\sim} H^{-i-2j}\left(M^{\sigma}_S(v), \mathcal{BPS}_{v,w}\right).\]
\end{thm}

\begin{proof}
    The conclusion follows from the diagram \eqref{diagss3} and Proposition \ref{prop79}.
\end{proof}

\begin{proof}[Proof of Theorem \ref{thmKtop}]
By Theorem \ref{thmKtopgraded} and Lemma \ref{lem:topisom},
    it suffices to check that there is a non-canonical isomorphism $K^{\mathrm{top}}_i(\mathbb{T}^{\sigma}_S(v)^{\mathrm{red}}_w)\cong \bigoplus_{j\in\mathbb{Z}}\mathrm{gr}_j K^{\mathrm{top}}_i(\mathbb{T}^{\sigma}_S(v)^{\mathrm{red}}_w)$.
    It suffices to check that the Chern character
    \[\mathrm{ch}\colon K^{\mathrm{top}}_i(\mathbb{T}^{\sigma}_S(v)^{\mathrm{red}}_w)\hookrightarrow G^{\mathrm{top}}_i(\mathcal{M}^\sigma_S(v))\to \widetilde{H}^{\mathrm{BM}}_i(\mathcal{M}^\sigma_S(v))\] is injective. By the diagram \eqref{diagsect8}, it suffices to check that the following Chern character is injective
    \[\mathrm{ch}\colon K^{\mathrm{top}}_i(\mathbb{T}(R)^{\mathrm{red}})\hookrightarrow G^{\mathrm{top}}_i(\mathcal{R})\to \widetilde{H}^{\mathrm{BM}}_i(\mathcal{R}),\] where $(R\xrightarrow{e} M)\in U$. This was proved in \cite[Proposition 9.9]{PTtop}.
\end{proof}

 \bibliographystyle{amsalpha}
\bibliography{math}
\medskip

\textsc{\small Tudor P\u adurariu: Max Planck Institute for Mathematics,
Vivatsgasse 7 
Bonn 53111, Germany.}\\
\textit{\small E-mail address:} \texttt{\small tpadurariu@mpim-bonn.mpg.de}\\

\textsc{\small Yukinobu Toda: Kavli Institute for the Physics and Mathematics of the Universe (WPI), University of Tokyo, 5-1-5 Kashiwanoha, Kashiwa, 277-8583, Japan.}\\
\textit{\small E-mail address:} \texttt{\small yukinobu.toda@ipmu.jp}\\
 \end{document}